\documentclass[english]{article}
\usepackage[T1]{fontenc}
\usepackage[utf8]{inputenc}
\usepackage{geometry}
\usepackage{xcolor}
\usepackage{float}
\usepackage{mathrsfs}
\usepackage{mathtools}
\usepackage{bm}
\usepackage{amsmath}
\usepackage{amsthm}
\usepackage{amssymb}
\usepackage{graphicx}
\usepackage{esint}
\usepackage[unicode=true,pdfusetitle,
 bookmarks=true,bookmarksnumbered=false,bookmarksopen=false,
 breaklinks=false,pdfborder={0 0 1},backref=page,colorlinks=true,linkcolor=magenta,citecolor=blue]
 {hyperref}

\makeatletter

\newcommand{\figone}{Figure~\ref{fig:Stasheff S3}}

\newtheorem{defn}{Definition}
\newtheorem{thm}{Theorem}
\newtheorem{prop}{Proposition}
\newtheorem{lem}[prop]{Lemma}
\newtheorem{cor}[prop]{Corolary}

\makeatother

\usepackage{babel}

\begin{document}

\title{Integrable Systems Arising from Separation of Variables on $S^{3}$ }
\author{Diana M.H.~Nguyen, Sean R.~Dawson,  Holger R.~Dullin\footnote{Emails: \url{diana.nguyen@sydney.edu.au}, \url{sdaw6022@uni.sydney.edu.au}, \url{holger.dullin@sydney.edu.au.}}\\School of Mathematics and Statistics,\\The University of Sydney, Australia}

\maketitle
\begin{abstract}
We show that the space of orthogonally separable coordinates on the sphere $S^3$ induces a natural family of integrable systems, which after symplectic reduction leads to a family of integrable systems on $S^2 \times S^2$.
The generic member of the family corresponds to ellipsoidal coordinates. We use the theory of compatible Poisson structures to study the critical points and critical values of the momentum map.
Interesting structure arises because the ellipsoidal coordinate system can degenerate in a variety of ways, and all possible orthogonally separable coordinate systems on $S^3$ (including degenerations) have the topology of the Stasheff polytope $K^4$, which is a pentagon. We describe how the generic integrable system degenerates, and how the appearance of global $SO(2)$ and $SO(3)$ symmetries is the main feature that organises the various degenerate systems.
For the whole family we show that there is an action map whose image is an equilateral triangle. When higher symmetry is present, this triangle ``unfolds'' into a semi-toric polygon (when there is one global $S^1$-action) or a Delzant polygon (when there are two global $S^1$-actions). We believe that this family of integrable systems is a natural playground for theories of global symplectic classification of integrable systems.
\end{abstract}

\section{Introduction}\label{sec:S3-1}

 Classifying and cataloguing integrable systems is an important unsolved problem. Only for particular classes of systems does a classification exist.
 It was shown by Atiyah-Guillemin-Sternberg \cite{atiyah,Guillemin1982} that the image of the momentum map of a toric system is a convex polytope, called the momentum polytope. This polytope completely classifies the toric system up to an equivariant symplectomorphism. Delzant \cite{BSMF_1988__116_3_315_0} gave an explicit construction of toric manifolds based on their momentum polytopes. Topological classification of Liouville foliations has been extensively studied by Bolsinov and Fomenko \cite{book} and extended towards the orbital classification of integrable systems.
  V\~u Ng\d{o}c and Pelayo have extended the toric classification to semi-toric systems  \cite{VuNgoc07,VuNgoc09}. While a complete theory of classification is still far out of reach, we'd like to investigate the possibility of extending the current theory to broader classes of integrable systems. 
  
  A very well studied class of integrable systems are superintegrable systems, see, e.g., \cite{Fasso05}. A classification of superintegrable systems in two and three degrees of freedom has been achieved by Kalnins, Kress and Miller in a series of works \cite{Kalnins2005-1,Kalnins2005-2,Kalnins2005-3,Kress06,Kalnins2006-5}. In their book \cite{KKM18} they highlight the link between superintegrability and separation of variables. Superintegrable systems that are separable in multiple coordinate systems provide a rich source of integrable systems. This is because each distinct separable coordinate system gives rise to a St\"{a}ckel integrable system \cite{Stackel}. Separable coordinates on  conformally flat spaces has been extensively studied by Kalnins and Miller in \cite{Kalnins1986,kalnins76,miller81}. More recently, Sch\"{o}bel studied the space of separable coordinates on the $n$-sphere as an algebraic variety \cite{Schoebel2014,Schoebel2015,Schoebel2016}. While it is known \cite{KKM18} that all separable coordinates on the sphere can be obtained as appropriate limits of the general Jacobi ellipsoidal coordinates, Sch\"{o}bel's work formalises this by giving a topology to this space in the form of the Stasheff polytope. The inspiration for this paper was to  establish a similar topology in the space of integrable systems that arise from separating the geodesic flow on $S^3$ in this family of coordinates. It is likely that similar constructions can be done for any superintegrable and multi-separable systems. In fact, the idea to use multi-separability to define interesting fibrations has been used in \cite{Dullin2012,Dullin2016} for the most fundamental systems of classical mechanics, the harmonic oscillator and the Kepler problem. To then use the periodic flow of the superintegrable system for reduction was first done in \cite{Dawson2022} and this paper is the natural continuation of that work where instead of a superintegrable system on $\mathbb{R}^3$ a superintegrable system on $S^3$ is the starting point.

Another motivation for our work are the recent studies \cite{Alonso,Alonso2019,Palmer2018,Hohloch2017} of various integrable systems on the compact symplectic manifold $S^2\times S^2$. We will show that the symplectic reduction of the geodesic flow on $S^3$ results in a reduced system on a compact symplectic leaf of $\mathfrak{so}^*(4)$ that is diffeomorphic to $S^2\times S^2$. The $3$-degrees of freedom integrable systems on $T^*S^3$ obtained from separation of variables descend to $2$-degrees of freedom systems on $S^2\times S^2$ through this quotient. We will employ more recent techniques in the theory of compatible Poisson structures and bi-Hamiltonian systems \cite{Bolsinov-Borisov,Bolsinov-Oshemkov} to study these systems in detail. This will allow us to realise these systems as special restricted cases of the Manakov top \cite{Komarov_1991, 339ab604500943c5b69d0da421f62e74}.

A somehow related question is the study of separable systems depending on parameters, foremost the geodesic flow on an ellipsoid  \cite{Moser:2261095}. In \cite{DAVISON20072437,Davison2007} the geodesic flow on 3-dimensional ellipsoids with various sets of equal semi-major axes has been studied. It is astonishing how similar these system -- degenerate or not -- are to the ones studied in this paper. However, the fundamental difference is that there a constant energy slice of a 3-degree of freedom system is studied, while here we reduce and study the resulting 2-degree of freedom system. As a result, here we obtain a system on a compact symplectic manifold, which is better suited as a playground for symplectic classification.
Similar degenerations have also been studied for the Neumann system \cite{Dullin2012} and again there are many similarities.

The paper is structured as follows. Section \ref{sec:S3-2} introduces the basic theory of separation of variables. 
We focus on the  separation of variables in the ellipsoidal coordinate system in Section \ref{sec:S3-3}. Section \ref{sec:S3-4} discusses the symplectic reduction of the geodesic flow on $S^3$ with emphasis on establishing the reduced ellipsoidal integrable system on $S^2\times S^2$. The theory of compatible Poisson structures is applied to study the reduced ellipsoidal integrable system in Section \ref{sec:S3-5}. In Section \ref{sec:degen}, we combine the techniques and results of Sections \ref{sec:S3-3}, \ref{sec:S3-4}, and \ref{sec:S3-5} to study the integrable systems obtained from separation of variables in the degenerate coordinate systems, namely prolate, oblate, Lam\'{e}, spherical and cylindrical coordinates.

\section{Orthogonally Separable Coordinate Systems on $S^{3}$}\label{sec:S3-2}

In this section, we introduce some basic concepts from the theory
of separation of variables. For more details, see \cite{KKM18,Stackel}.
Let $(s_{i},p_{i})$ be local canonical coordinates on $T^{*}M$ where
$M$ is an $n-$dimensional differentiable manifold with metric $g$.
Define $\mathcal{S}^{k}(M)$ to be the space of smooth contravariant
symmetric tensors of order $k$ on $M$, in particular $k=1$ is the
space of vector fields on the manifold. Each $\bm{K}\in\mathcal{S}^{k}(M)$
can be associated with a $C^{\infty}$ real function $E_{\bm{K}}$
on $T^{*}M$ locally expressed in the momenta as 
\begin{equation}
E_{\bm{K}}=\sum_{i}\frac{1}{k!}\bm{K}^{i_{1}\dots i_{k}}p_{i_{1}}\dots p_{i_{k}}.\label{eq:Killing tensor def}
\end{equation}
The Lie bracket between two tensors $\bm{K}\in\mathcal{S}^{k}(M)$
and $\bm{R}\in\mathcal{S}^{r}(M)$, denoted by $[\bm{K},\bm{R}]\in\mathcal{S}^{k+r-1}(M)$,
is defined by 
\begin{equation}
E_{[\bm{K},\bm{L}]}=\{E_{\bm{K}},E_{\bm{L}}\}_{(s_{i},p_{i})}\label{eq:commute}
\end{equation}
 where $\{\cdot,\cdot\}_{(s_{i},p_{i})}$ denotes the canonical Poisson
bracket in the curvilinear coordinates.
\begin{defn}
On a Riemannian manifold $(M,g)$, a symmetric tensor $\bm{K}$ is
a Killing tensor if it commutes with the metric $[\bm{K},g]=0$.
\end{defn}
There is a natural identification of Killing tensors with first integrals
of the geodesic flow:
% \begin{lem}
A function $E_{\bm{K}}$ is a first integral of the geodesic flow
if and only if $\bm{K}$ is a Killing tensor, see, e.g \cite{Stackel}. 
% \end{lem}
Similarly to how a set of $n$ integrals in involution on a manifold
form a Liouville-integrable system, a collection of $n$ Killing tensors
defines a St\"{a}ckel system. 
\begin{defn}
A St\"{a}ckel system on an $n-$dimensional Riemannian manifold $(M,g)$
is a set of $n$ Killing tensors of order $2$ that commute under
the commutator in (\ref{eq:commute}).
\end{defn}
Eisenhart proved in \cite{10.2307/1968433} that there is a bijective
correspondence between equivalence classes of St\"{a}ckel systems
and equivalence classes of orthogonal separable coordinate systems.
Given an orthogonal coordinate system $s_{i}$ on an $n$-dimensional
manifold $M$, we can define a St\"{a}ckel system by constructing
a St\"{a}ckel matrix.
\begin{defn}
A St\"{a}ckel matrix $\Phi$ for a given metric $g$ is any $n\times n$ matrix where each
row depends on only one of the curvilinear coordinates $s_{i}$, and
\begin{equation}
\frac{1}{g_{ii}}=\frac{\det(\Omega_{i})}{\det(\Phi)},\label{eq:stackel relationdship-1}
\end{equation}
where $\Omega_{i}$ is the minor formed by deleting the $i^{th}$
row and first column of $\Phi$ and $g_{ii}$ is the $(i,i)$ element
of the metric tensor. 
\end{defn}
The functional value of the Hamiltonian $H$ will be denoted by $h$.
Let $W\coloneqq W(\bm{s},\bm{\eta})$ where $\bm{s}=(s_{1},\dots,s_{n})$
are the separable curvilinear coordinates and $\bm{\eta}=(\eta_{0},\dots,\eta_{n-1})$
are parameters. The Hamilton Jacobi equation is given by 
\begin{equation}
  \sum_{i}\frac{1}{2}(g^{-1})_{ii}\left(\frac{\partial W}{\partial s_{i}}\right)^2=h.  \label{HJ}
\end{equation}

This can be separated by computing 
\begin{equation}
\Phi^{-1}\mathfrak{p}=\bm{\eta},\label{eq:sep constants-1}
\end{equation}

where $\mathfrak{p}=\left(p_{1}^{2},\dots,p_{n}^{2}\right)^{t}$ is
the vector of squared canonical curvilinear momenta and the parameters
$\bm{\eta}$ are the separation constants.
Comparing (\ref{eq:sep constants-1}) and (\ref{eq:Killing tensor def}),
we see that rows of $\Phi^{-1}$ encode the diagonal entries of the
Killing tensors for separable orthogonal coordinates $s_{i}$. The first row of \eqref{eq:sep constants-1} gives \eqref{HJ} and so we have $\eta_0=h$. 

The work of Kalnins and Miller \cite{Kalnins1986} gave a graphical
algorithm for constructing all orthogonally separable coordinates on constant curvature
manifolds. Recent results by Sch\"{o}bel and Veselov extended
this by giving an algebraic geometric classification of separable coordinate
systems on $S^{n}$ \cite{Schoebel2015,Schoebel2014}. In particular,
they showed that the variety of St\"{a}ckel systems on $S^{n}$ is
given by the Stasheff polytope $K_{n+1}$ which is a convex polytope
of dimension $n-1$. In this paper we work only with $S^{3}$; the
relevant Stasheff polytope $K_{4}$ is shown in Figure~\ref{fig:Stasheff S3}. 

The codimension $0$ face of a Stasheff polytope represents the family
of ellipsoidal coordinates $s_{i}$. These are defined as the roots
of $T(s)=\sum_{i=1}^{n+1}\frac{x_{i}^{2}}{s-e_{i}}=0$ where $x_{i}$
are Cartesian coordinates on $\mathbb{R}^{n+1},$ $e_{j}\le s_{j}\le e_{j+1}$
for all $j=1,\dots,n$, the $e_{i}\ge0$ are all distinct and are called the semi-major
axes. Note that when using a similar coordinate system to separate the geodesic flow on an ellipsoid these parameters actually are semi-major axes (hence the name), while here they describe a separating coordinate system but \emph{not} the underlying manifold. Solving $T(s_{j})=0$ together with $\sum_{i=1}^{n+1}x_{i}^{2}=1$
gives
\begin{equation}
    x_{i}^{2}=\frac{\prod_{j=1}^{n}(s_{j}-e_{i})}{\prod_{k\ne i}(e_{k}-e_{i})}.\label{generalcoord}
\end{equation}

Note that for $S^3$, despite the ellipsoidal coordinates being parametrised by the 4 parameters $e_1<e_2<e_3<e_4$, the Stasheff polytope is only $2$-dimensional. Affine transformation of the parameters of the form $e_i\mapsto \alpha e_i + \beta$ for $\alpha\neq0$ gives a equivalent coodinate system up to scaling. Thus each ellipsoidal coordinate system on $S^3$ with parameters $e_1<e_2<e_3<e_4$ can be transformed to an equivalent system with parameters $0<1<a<b$, see \cite{KKM18} for details.

Higher codimension faces represent families of degenerate coordinate
systems on the sphere. Degenerate coordinate systems on $S^{n}$ are
constructed by gluing ellipsoidal coordinates on lower dimensional
spheres together. Let $\bm{w}_{\alpha}$ and $\bm{z}$ be Cartesian
coordinates expressed in terms of local ellipsoidal coordinates on
$S^{k_{\alpha}-1}$ and $S^{m}$ respectively where $\alpha=1,\dots,m+1$
and $n=k_{1}+\dots k_{m+1}-1$. Then any degenerate coordinate system
on $S^{n}$ is found by recursively applying composition \cite{Schoebel2015}
\[
\begin{aligned}\circ:S^{m}\times S^{k_{1}-1}\times\dots\times S^{k_{m+1}-1} & \to S^{k_{1}+\dots+k_{m+1}-1}\\
(\bm{z},\bm{w}_{1},\dots,\bm{w}_{m+1}) & \to(z_{1}\bm{w}_{1},\dots,z_{m+1}\bm{w}_{m+1})
\end{aligned}
\]
where $z_{\beta}\bm{w}_{\beta}$ are Cartesian coordinates on $\mathbb{R}^{n+1}$
expressed in the new degenerate coordinate system. 

In this paper we have chosen to adopt the notation of Sch\"{o}bel
\cite{Schoebel2014}. The general ellipsoidal coordinates on $S^{n}$
are represented as $(1\ 2\ \dots\ n+1)$. If one attaches an $S^{j}$
to the $m^{\text{th}}$ Cartesian coordinate, then we enclose brackets
around all numbers $m$ to $m+j$. For instance, attaching $S^{1}$
to the the $x_{2}$ coordinate on $S^{2}$ is written as $(1\ (2\ 3)\ 4)$.
Symmetric bracketing results in systems with similar behaviour, i.e.
$(1\ 2\ (3\ 4))$ and $((1\ 2)\ 3\ 4)$ describe equivalent coordinate systems as we will discuss in more detail below, also see \cite{KKM18,Schoebel2016,Schoebel2015,Schoebel2014}. For the various degenerate coordinate systems we are going to use short hand names as indicated in \figone.

\begin{figure}
\begin{centering}
\includegraphics[width=8cm]{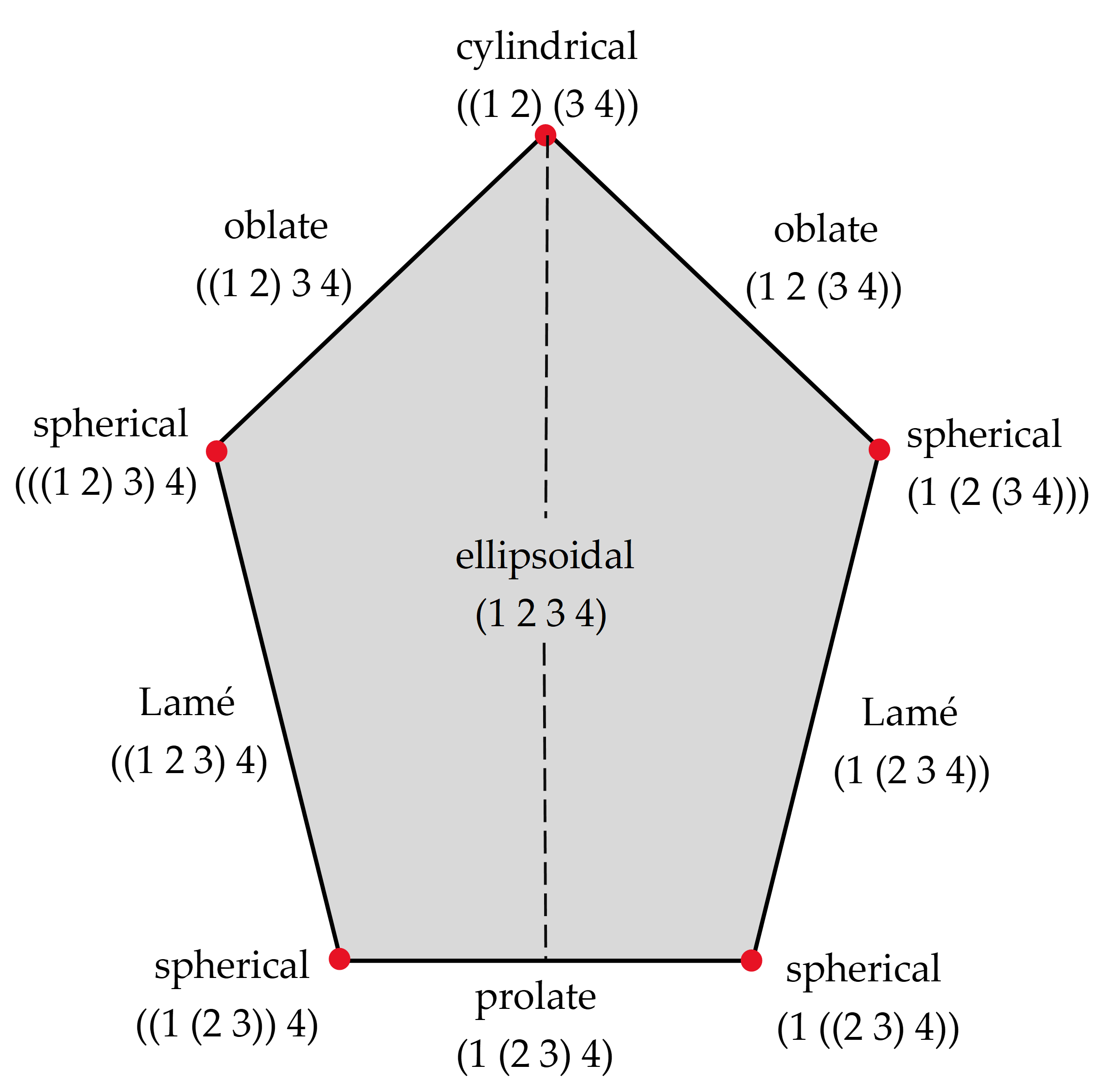}
\par\end{centering}
\caption{Stasheff polytope $K^{4}$ of separating coordinate systems on $S^{3}$ with corresponding names and bracket notation. \label{fig:Stasheff S3}}
\end{figure}

\section{Separation of variables of the Geodesic Flow on $S^{3}$ and Ellipsoidal Coordinates\label{sec:Geodesic-Flow-on-S3}}\label{sec:S3-3}

Consider the geodesic flow on the $3-$sphere as a constrained system
on $T^{*}\mathbb{R}^{4}$ with Cartesian coordinates $(\bm{x},\bm{y})$
where $\bm{x}\cdot\bm{x}=1$ and $\bm{x}\cdot\bm{y}=0$. We define
the Poisson bracket on $T^{*}S^{3}$ to be the Dirac bracket $\{\cdot,\cdot\}$
enforcing these two constraints. Let $\{\cdot,\cdot\}_{T^{*}\mathbb{R}^{4}}$
be the canonical Poisson bracket on $T^{*}\mathbb{R}^{4}$. Set $c_{1}=\bm{x}\cdot\bm{x},c_{2}=\bm{x}\cdot\bm{y}$
and define the matrix 
\[
C_{ij}=\{c_{i},c_{j}\}_{T^{*}\mathbb{R}^{4}}.
\]
The Dirac bracket on $T^{*}S^{3}$ is given by 
\[
\{f,g\}=\{f,g\}_{T^{*}\mathbb{R}^{4}}+\{f,c_{i}\}_{T^{*}\mathbb{R}^{4}}(C^{-1})_{ij}\{c_{j},g\}_{T^{*}\mathbb{R}^{4}}
\]
 with structure matrix 
\begin{equation}
B_{(\bm{x},\bm{y})}=\begin{pmatrix}\bm{0} & id-\bm{x}\bm{x}^{t}\left|\bm{x}\right|^{-2}\\
-id+\bm{x}\bm{x}^{t}\left|\bm{x}\right|^{-2} & -(\bm{x}\bm{y}^{t}-\bm{y}\bm{x}^{t})\left|\bm{x}\right|^{-2}
\end{pmatrix}.\label{eq:BPQ-1-1}
\end{equation}

From (\ref{eq:BPQ-1-1}), the Poisson bracket between two functions
$f$ and $g$ on $T^{*}S^{3}$ is
\begin{equation}
\{f,g\}\coloneqq\left(\nabla f\right)^{T}B_{(\bm{x},\bm{y})}\nabla g\label{eq:PB in QP}
\end{equation}
where $\nabla$ denotes the gradient with respect to the Cartesian
coordinates $(\bm{x},\bm{y})$. 

Let $H=\frac{1}{2}\bm{y}\cdot\bm{y}$ be the Hamiltonian of the geodesic
flow on $S^{3}$. It is well known that this system is superintegrable
and separates in multiple coordinate systems \cite{Schoebel2014}.
The Hamilton Jacobi equation can be separated in general ellipsoidal
coordinates $(s_{1},s_{2},s_{3})$ given by \eqref{generalcoord} for $n=3$ as
\begin{equation}
\begin{aligned}x_{j}^{2} & =\frac{\left(s_{1}-e_{j}\right)\left(s_{2}-e_{j}\right)\left(s_{3}-e_{j}\right)}{\Pi_{i\ne j}\left(e_{i}-e_{j}\right)}, &  & j=1,\dots,4\end{aligned}
\label{eq:def of ellipsoidal coordinates}
\end{equation}
where $0\le e_{1}\le s_{1}\le e_{2}\le s_{2}\le e_{3}\le s_{3}\le e_{4}$ and the $e_i$'s are all distinct. 
In these coordinates, the geodesic Hamiltonian is (see e.g. \cite{DAVISON20072437,KKM18})

\begin{equation}
H=-2\sum_{i=1}^{3}\frac{\prod_{j=1}^{4}(s_{i}-e_{j})}{\prod_{k\ne i}(s_{i}-s_{k})}p_{i}^{2}.\label{eq:Ham in ellipsoidal}
\end{equation}
This system is Liouville integrable. To separate the Hamilton Jacobi
equation we use the following St\"{a}ckel matrix
\begin{equation}
\Phi_{el}=\frac{1}{4}\begin{pmatrix}-\frac{s_{1}^{2}}{A(s_{1})} & -\frac{s_{1}}{A(s_{1})} & -\frac{1}{A(s_{1})}\\
-\frac{s_{2}^{2}}{A(s_{2})} & -\frac{s_{2}}{A(s_{2})} & -\frac{1}{A(s_{2})}\\
-\frac{s_{2}^{2}}{A(s_{3})} & -\frac{s_{1}}{A(s_{3})} & -\frac{1}{A(s_{3})}
\end{pmatrix}.\label{eq:Stackel El S3-1}
\end{equation}
where $A(z)=\prod_{k=1}^{4}(z-e_{k})$. From (\ref{eq:sep constants-1}),
we compute $\Phi_{el}^{-1}\mathfrak{p}=(\eta_{0},-\eta_{1},\eta_{2})^{T}$
where 
\begin{equation}
\begin{aligned}\eta_{0} & =-4\sum_{i=1}^{3}\frac{A(s_{i})}{D(s_{i})}p_{i}^{2}, &  &  & \eta_{1} & =-4\sum_{i=1}^{3}\left(\frac{A(s_{i})}{D(s_{i})}p_{i}^{2}\sum_{k\ne i}s_{k}\right), &  &  & \eta_{2} & =-4\sum_{i=1}^{3}\left(\frac{A(s_{i})}{D(s_{i})}p_{i}^{2}\prod_{k\ne i}s_{k}\right),\end{aligned}
\label{eq:separation constants}
\end{equation}
and $D(s_{i})=\prod_{k\ne i}(s_{i}-s_{k}).$ To express (\ref{eq:separation constants})
in terms of the angular momenta $\ell_{ij}\coloneqq x_{i}y_{j}-x_{j}y_{i}$,
we note that 
\[
\ell_{ij}^{2}=\frac{x_{i}^{2}x_{j}^{2}}{4}\left(\sum_{k=1}^{3}\left(-\frac{1}{s_{k}-e_{i}}+\frac{1}{s_{k}-e_{j}}\right)p_{k}\right)^{2}
\]
with $x_{i}$ given by (\ref{eq:def of ellipsoidal coordinates}).
Let $\bm{L}\coloneqq(\ell_{12},\ell_{13},\ell_{14},\ell_{23},\ell_{24},\ell_{34})$, it can be verified that 
\begin{equation}
\begin{aligned}\eta_{0}(\bm{L}) & =\sum_{i<j}\ell_{ij}^{2}=2H, &  &  & \eta_{1}(\bm{L}) & =\sum_{i<j}\left(\ell_{ij}^{2}\sum_{k\ne i,j}e_{k}\right), &  &  & \eta_{2}(\bm{L}) & =\sum_{i<j}\left(\ell_{ij}^{2}\prod_{k\ne i,j}e_{k}\right).\end{aligned}
\label{eq:Separation constants ellipsoidal}
\end{equation}
The separated equations are obtained by multiplying both sides of
(\ref{eq:sep constants-1}) by $\Phi_{el}$. This gives 
\begin{equation}
p_{i}^{2}=\frac{-R(s_{i})}{4A(s_{i})}\label{eq:psq ellipsoidal-1}
\end{equation}
where $R(z)=2hz^{2}-\eta_{1}^{*}z+\eta_{2}^{*}$ and $(\eta_{1}^{*},\eta_{2}^{*})$
are the values of the integrals $(\eta_{1},\eta_{2})$. Thus,
the geodesic flow is separable on the hyperelliptic curve $w^{2}=-R(z)A(z)$
which has genus $3$. 

It is known \cite{Moser:2261095} that a set of global polynomial
integrals for the geodesic flow are given by 
\begin{equation}
F_{i}=\sum_{j=1,j\ne i}^{4}\frac{\ell_{ij}^{2}}{e_{i}-e_{j}},\label{eq:ulenbech integral def}
\end{equation}
 where $i\in\{1,2,3,4\}.$ The $F_{i}$ are known as the Uhlenbeck
integrals and satisfy $\sum_{i=1}^{4}F_{i}=0$. The separation constants
$\eta_{1}$ and $\eta_{2}$ are related to the $F_{i}$ via the identity
\begin{equation}
\begin{aligned}\sum_{i=1}^{4}\frac{F_{i}}{z-e_{i}} & =\frac{2Hz^{2}-\eta_{1}z+\eta_{2}}{\prod_{k=1}^{4}(z-e_{k})}.\end{aligned}
\label{eq:R on A def}
\end{equation}

Since all integrals are polynomial, wee can easily show that $\eta_{1},\eta_{2}$ and $H$ are functionally
independent almost everywhere and $\{\eta_{1},H\}=\{\eta_{2},H\}=\{\eta_{1},\eta_{2}\}=0$.
This establishes that
% \begin{lem}
the triple $(H,\eta_{1},\eta_{2})$ is an integrable system on $T^{*}S^{3}$.
% \end{lem}
 We call this the ellipsoidal integrable system on $T^{*}S^{3}$. 
 This is a reformulation of the underlying St\"ackel system, whose commuting Killing tensors lead to quadratic (in momenta) integrals, which are also quadratic in angular momenta.

Since $H$ is a global $S^{1}$ action on the energy surface where $2h=1$ we can use it to perform symplectic reduction. Doing so, we obtain a reduced system on the symplectic manifold $S^{2}\times S^{2}$. The integrals $(\eta_{1},\eta_{2})$ descend
to form an integrable system on this quotient space.

\section{Reduction by Geodesic Flow }\label{sec:S3-4}

The orbits of $H$ are oriented great circles on $S^{3}$. Since each
great circle is the intersection of a two dimensional plane through
the origin with $S^{3}$, the orbit space of $H$ is given by \textcolor{black}{
\[
T^{*}S^{3}/S^{1}|_{H=h}=U^{*}S^{3}/S^{1}\cong\widetilde{Gr}(2,4)
\]
}where the oriented Grassmanian $\widetilde{Gr}(2,4)$ is the set
of oriented two dimensional planes in $\mathbb{R}^{4}$ \textcolor{black}{and
$U^{*}S^{3}$ is the unit cotangent bundle of $S^{3}$}. One way to see that $\widetilde{Gr}(2,4)\cong S^{2}\times S^{2}$ is using the
Pl\"{u}cker embedding (for more details on this, see Appendix \ref{appen-gras}).
For our purposes, the reduction will be performed with invariants
of the geodesic Hamiltonian. 

Invariants of $H$ are the six angular momenta $\bm{L}\coloneqq(\ell_{12},\ell_{13},\ell_{14},\ell_{23},\ell_{24},\ell_{34})$.
These form a closed set of invariants under the Dirac bracket $\{\cdot,\cdot\}$
from (\ref{eq:PB in QP}). The Poisson algebra of these invariants has the structure matrix
\begin{equation}
B_{\bm{L}}=\begin{pmatrix}
 0 & \ell_{23} & \ell_{24} & -\ell_{13} & -\ell_{14} & 0 \\
 -\ell_{23} & 0 & \ell_{34} & \ell_{12} & 0 & -\ell_{14} \\
 -\ell_{24} & -\ell_{34} & 0 & 0 & \ell_{12} & \ell_{13} \\
 \ell_{13} & -\ell_{12} & 0 & 0 & \ell_{34} & -\ell_{24} \\
 \ell_{14} & 0 & -\ell_{12} & -\ell_{34} & 0 & \ell_{23} \\
 0 & \ell_{14} & -\ell_{13} & \ell_{24} & -\ell_{23} & 0 \\
\end{pmatrix}
\end{equation}
with 2 Casimirs: $\mathcal{C}_{1}=2H=\sum_{j>i}\ell_{ij}^{2}$ and
$\mathcal{C}_{2}=\ell_{12}\ell_{34}-\ell_{13}\ell_{24}+\ell_{14}\ell_{23}$.
The first is the energy of the geodesic flow which we have the freedom
to set to an arbitrary value $2h$. The second Casimir is the Pl\"{u}cker
relation and must be zero since the angular
momenta $\bm{L}=\bm{x}\wedge\bm{y}$ where $\wedge$ is the wedge
operator. This means that $\bm{L}$ is a totally
decomposable bivector and so must satisfy $C_{2}=\left|\bm{L}\wedge\bm{L}\right|=0$. The Lie Poisson algebra of the $\ell_{ij}$'s is isomorphic to the Lie algebra $\mathfrak{so}(4)$.

Using the $\ell_{ij}$ as new coordinates, we obtain an explicit realisation
of $S^{2}\times S^{2}$ as 
\begin{align}
\mathscr{C}_{1} & =\mathfrak{\mathcal{C}}_{1}+2\mathcal{C}_{2}=(\ell_{12}+\ell_{34})^{2}+(\ell_{13}-\ell_{24})^{2}+(\ell_{14}+\ell_{23})^{2}=2h,\label{eq:Cas in ls}\\
\mathscr{C}_{2} & =\mathcal{C}_{1}-2\mathcal{C}_{2}=(\ell_{12}-\ell_{34})^{2}+(\ell_{13}+\ell_{24})^{2}+(\ell_{14}-\ell_{23})^{2}=2h.\nonumber 
\end{align}
The Poisson bracket of functions on $S^2\times S^2$, denoted by $\{\cdot,\cdot\}_{\bm L}$
is \begin{equation}
\{f,g\}_{\bm L}=(\nabla f)^{T}B_{\bm L}\nabla g\label{eq:XY pb}
\end{equation}
where $\nabla$ denotes the gradient with respect to $\bm L$.

A sometimes more convenient set of coordinates on $S^{2}\times S^{2}$ is obtained by applying the linear transformation \linebreak
$T:\bm{L}~\mapsto(\bm{X},\bm{Y})=~(X_1,X_2,X_3,Y_1,Y_2,Y_3)$ with
\begin{equation}
\begin{aligned}X_{1} & =\frac{1}{2}(\ell_{12}+\ell_{34}), &  &  & Y_{1} & =\frac{1}{2}(\ell_{12}-\ell_{34}),\\
X_{2} & =\frac{1}{2}(\ell_{13}-\ell_{24}), &  &  & Y_{2} & =-\frac{1}{2}(\ell_{13}+\ell_{24}),\\
X_{3} & =\frac{1}{2}(\ell_{14}+\ell_{23}), &  &  & Y_{3} & =\frac{1}{2}(\ell_{14}-\ell_{23}).
\end{aligned}
\label{eq:ls to Xy}
\end{equation}
In these variables we can rewrite (\ref{eq:Cas in ls}) as $\mathscr{C}_{1}=4\left|\bm{X}\right|^{2}$
and $\mathscr{C}_{2}=4\left|\bm{Y}\right|^{2}$ which both have functional
value $2h$. The Poisson structure \eqref{eq:BPQ-1-1} becomes block diagonal
\begin{equation}
B_{\bm{X},\bm{Y}}=\begin{pmatrix}\bm{\hat{X}} & \bm{0}\\
\bm{0} & \bm{\hat{Y}}
\end{pmatrix}\label{eq:Bxy}
\end{equation}
and is isomorphic to $\mathfrak{so}(3)\times\mathfrak{so}(3)$. The
notation in (\ref{eq:Bxy}) is such that for a vector $\bm{v}\in\mathbb{R}^{3}$
the corresponding antisymmetric hat matrix $\hat{\bm{v}}$ is defined
by 
\[
\begin{aligned}\hat{\bm{v}}\bm{u}=\bm{v}\times\bm{u} &  & \forall\bm{u}\in\mathbb{R}^{3}.\end{aligned}
\]
This reduction gives an integrable system on the reduced manifold
$S^{2}\times S^{2}$. While the integrals $(\eta_{1},\eta_{2})$ can be easily rewritten in terms of the $(\bm{X},\bm{Y})$ variables using \eqref{eq:ls to Xy}, they are simplest and most symmetric as functions of the $\ell_{ij}$'s as in \eqref{eq:Separation constants ellipsoidal}.

Under (\ref{eq:XY pb}), $(\eta_{1},\eta_{2})$ are commuting quadratic
functions on $S^{2}\times S^{2}$ and so we arrive at the following
result.
\begin{thm}
The integrable sytem $(H,\eta_{1},\eta_{2})$ on $T^{*}S^{3}$
descends to an integrable system $(\eta_{1}(\bm{L}),\eta_{2}(\bm{L}),\{\cdot,\cdot\}_{\bm L})$
on $S^{2}\times S^{2}$ with two degrees of freedom and integrals quadratic in $\ell_{ij}$. 
\end{thm}

We call this integrable system the reduced ellipsoidal integrable system.
 The symplectic reduction performed in this section also applies to integrable systems obtained from separating the geodesic flow in the degenerate coordinate systems shown in Figure~\ref{fig:Stasheff S3}. 

Separating coordinate systems on $S^3$ are invariant under 
affine transformations $e_i\mapsto \alpha e_i + \beta $ for $\alpha\neq0$. This allows us to normalise the ordered distinct parameters $(e_1, e_2, e_3, e_4)$ to $(0, 1, a, b)$ by a shift by $\beta = -e_1$ and a scaling by $\alpha = e_2 - e_1$. Thus the inside of 
Figure~\ref{fig:Stasheff S3} can be thought of as the region $1 < a < b$. 
 The affine transformations when applied to the family of corresponding integrable systems on $T^*S^3$ gives topologically equivalent integrable systems. This property can be observed directly in the reduced systems on $S^2\times S^2$.
 \begin{lem} \label{eta lemma}
     Affine transformation of the parameters $e_i\mapsto \alpha e_i + \beta $ for $\alpha\neq0$ when applied to the reduced system $(\eta_1(\bm{L}),\eta_2(\bm{L}))$ gives a topologically equivalent integrable system.
 \end{lem}
 \begin{proof}
     Applying $e_i\mapsto \alpha e_i + \beta $ to $(\eta_1(\bm{L}),\eta_2(\bm{L}))$ induces the map
     \[(\eta_1,\eta_2)\mapsto (\alpha \eta_1+4\beta h,\alpha^2\eta_2+\alpha\beta\eta_1+2\beta^2h)\]
     which gives a topologically equivalent system, because it is a linear map of the original integrals, plus affine terms that add the Casimir $h$.
 \end{proof}

This result illustrates nicely how the equivalence of separating coordinates leads to an equivalence of reduced integrable systems. It highlights the fact that the reduced system does not have a Hamiltonian (since we reduced by the flow of $H$) and hence it is natural to consider quadratic integrals up to linear transformations.

\begin{figure}
\begin{centering}
\includegraphics[width=7cm]{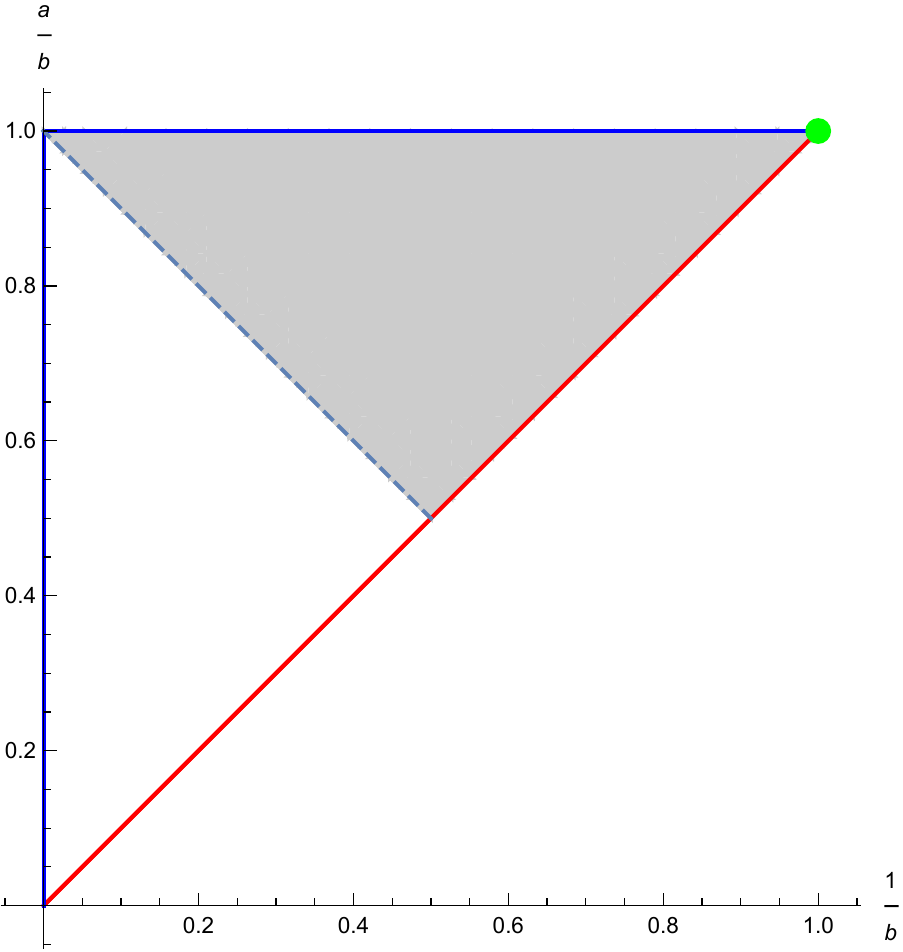}$\quad$
\includegraphics[width=7cm]{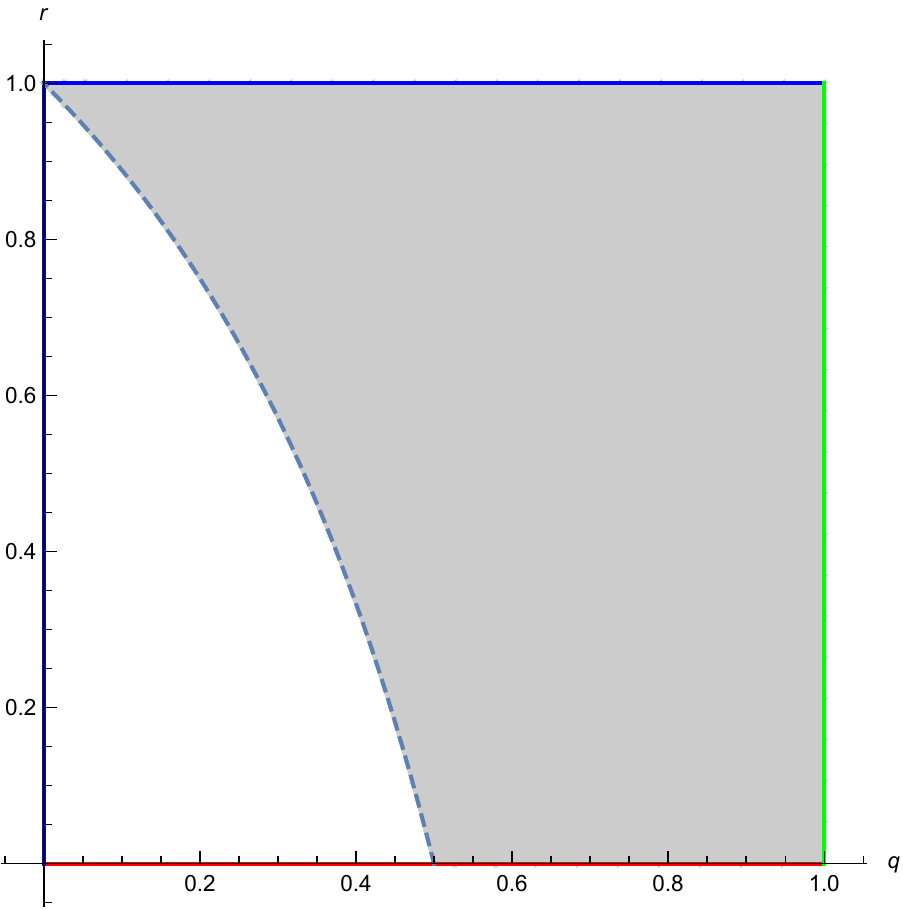}
\end{centering}
\caption{a) Parameter space of ellipsoidal coordinates $1 < a < b$ (red: prolate, blue: oblate, dashed: fixed set).
b) Blown up parameter space giving half the Stasheff polytope (green: Lam\'e). 
\label{fig:iab}}
\end{figure}

There is another equivalence between separating coordinates which maps an ordered quadruple $(e_1, e_2, e_3, e_4)$ to an ordered quadruple $(-e_4, -e_3, -e_2, -e_1)$ . After normalisation this maps $(0, 1, a, b)$ to $(0, 1, a', b')$ 
where $a' = (b-1)/(b-a)$ and $b' = b/(b-a)$. This map is an involution
that can be written as 
\begin{equation} \label{eq:involution}
\frac{1}{b'} = 1 - \frac{a}{b}, \quad
\frac{a'}{b'} = 1 - \frac{1}{b}\,.
\end{equation}
The line of fixed points of the involution is $1/b + a/b = 1$.
This suggests to map the parameter region $1 < a < b$ to the triangle
$0 < a/b < 1/b < 1$ which is cut in half by the line of fixed points, see Figure~\ref{fig:iab}~a). The prolate case $a=1$ and the oblate case $a=b$ correspond to two edges of this triangle. To see the whole parameter space the $e_i$ need to be considered projectively. In particular the point $1=a=b$ representing the Lam\'e family needs to be blown up. 

\begin{lem}
 Define $q=1/b, r = (a-1)/(b-1)$. There is a one-to-one correspondence between equivalence classes of reduced integrable system on $S^2\times S^2$ and points in the region $0 \le q \le 1$, $0 \le r \le 1$, $q \ge (1-r)/(2-r)$.
\end{lem}
\begin{proof}
 The map $r = (a-1)/(b-1)$ blows up the point $1=a=b$ to a line.
The parameter $r$ is the inverse of the slope of a straight line in $ab$-space through the point $a=b=1$. Due to $1 \le a \le b$ we have $0 \le r \le 1$. Consider the images of the edges of the triangle
$0 < a/b < 1/b < 1$.
The prolate line segment is mapped to $r=0$, $1/2 < q < 1$.
The oblate line segment is mapped to $r=1$, $0 < q < 1$.
The line of fixed points is mapped to $q = (1-r)/(2-r)$.
This establishes the claimed boundaries of the region.
The corners of the region are: 
\begin{itemize}
    \item $(q,r) = (0,1/2)$: symmetric prolate coordinates.
    \item $(q,r) = (0, 1)$: spherical coordinates.
    \item $(q,r) = (1,1)$: spherical coordinates.
    \item $(q,r) = (0,1)$: cylindrical coordinates.
\end{itemize}
\end{proof}
The region described in the Lemma in $(r,q)$ space is shown in Figure~\ref{fig:iab}~b).
It represents half of the Stasheff polytope \figone.
 
 In the next section, we study the reduced ellipsoidal system in depth. We find the momentum map, compute the critical points and critical values. The integrable systems arising from the degenerate coordinate systems will be covered in detail in section \ref{sec:degen}.

\section{The Reduced Ellipsoidal Integrable System}\label{sec:S3-5}

To construct and study the bifurcation diagram for the ellipsoidal
integrable system we will employ techniques from \cite{Bolsinov-Borisov} and \cite{Bolsinov-Oshemkov} using
compatible Poisson structures.

\subsection{Compatible Poisson Structures}

In this section, we will be closely following Example B in \cite{Bolsinov-Oshemkov}. Let us consider the reduced system on a symplectic leaf of $\mathfrak{so}^{*}(4)$
defined by $\mathcal{C}_{1}=\bm{L}\cdot\bm{L}=2h$ and
$\mathcal{C}_{2}=\ell_{12}\ell_{34}-\ell_{13}\ell_{24}+\ell_{14}\ell_{23}=0$.
On $\mathfrak{so}(4)$ we have the standard bracket $[X,Y]=XY-YX$ and we can
identify elements $X\in \mathfrak{so}(4)$ with elements $X^{*}\in \mathfrak{so}^{*}(4)$
via 
\begin{align*}
K(X,\cdot) & =X^{*}.
\end{align*}

Here $K$ is the Killing form defined as 
\begin{align*}
K(X,Y) & =\text{Tr}(\text{ad}_{X}\circ\text{ad}_{Y})\in\mathbb{R}
\end{align*}
where $\text{ad}_{X}=[X,\cdot]$. Explicitly, let us define $X_{ij}$ to be the $4\times4$ matrix with $1$ in the $ij^\text{th}$ position, $-1$ in the $ji^\text{th}$ position and $0$ everywhere else. This gives us a basis of $\mathfrak{so}(4)$. An element $X\in \mathfrak{so}(4)$ of the form
\[X=\begin{pmatrix} 0 & \ell_{12}&\ell_{13}&\ell_{14}\\
-\ell_{12}&0&\ell_{23}&\ell_{24}\\
-\ell_{13}&-\ell_{23}&0&\ell_{34}\\
-\ell_{14}&-\ell_{24}&\ell_{34}&0
\end{pmatrix}\] 
can be written as $X=\sum_{j>i}\ell_{ij}X_{ij}$. This allows for the further identification of $T(\mathfrak{so}^*(4))\equiv \mathfrak{so}(4)$ with $\nabla_{\bm{L}}f\leftrightarrow \sum_{i<j}\frac{\partial_f}{\partial\ell_{ij}} X_{ij} $. We can now express the integrals $\eta_1$ and $\eta_2$ as functions on $\mathfrak{so}(4)$ as 
\begin{align*}
\begin{aligned}\eta_1 & =\text{Tr} (X^t\nabla_{\bm{L}} \eta_1) &  &  & \eta_2 & =\text{Tr} (X^t\nabla_{\bm{L}} \eta_2)
\end{aligned}
\end{align*}
with $\nabla_{\bm{L}} \eta_1\coloneqq A_1=\sum_{i<j}(e_m+e_n)\ell_{ij}X_{ij}$ and $\nabla_{\bm{L}} \eta_2\coloneqq A_2=\sum_{i<j}e_me_n\ell_{ij}X_{ij}$ where $i,j,m,n$ are all distinct. Dynamics on $\mathfrak{so}^*(4)$ with Hamiltonian $\eta_1$ or $\eta_2$ can be rewritten in Lax form as 
$\dot{X}=[A_1,X]$ or $\dot{X}=[A_2,X]$. The trace of $X^2$ and $X^4$ recover the Casimirs $-2\mathcal{C}_1$ and $-4\mathcal{C}_2^2+2\mathcal{C}_1^2$, respectively.

Let $C$ be a symmetric matrix and define the Lie bracket $[X,Y]_{C}\coloneqq XCY-YCX$. This lifts to the Poisson bracket $\{\cdot,\cdot\}_{C}$ on $\mathfrak{so}^*(4)$. We can
WOLG assume that $C$ is diagonal of the form $C=\text{diag}(c_{1},c_{2},c_{3},c_{4})$.
If $C$ is invertible, then there exists an isomorphism $\alpha$
from $\mathfrak{so}(4,\mathbb{C})$ to $\mathfrak{so}(4,[\cdot,\cdot]_{C})$ defined by
$\alpha:X\mapsto C^{1/2}XC^{1/2}$. This lifts to a linear map 
\begin{align}
\gamma_C:\mathfrak{so}^{*}(4,\mathbb{C}) & \to \mathfrak{so}^{*}(4,\{\cdot,\cdot\}_{C})\label{eq:isomorphism}\\
\bm{L} & \mapsto M_C\bm{L}\nonumber 
\end{align}
 where $M_C=\text{diag}(\sqrt{c_{1}c_{2}},\sqrt{c_{1}c_{3}},\sqrt{c_{1}c_{4}},\sqrt{c_{2}c_{3}},\sqrt{c_{2}c_{4}},\sqrt{c_{3}c_{4}}).$

The Poisson matrix for $\mathfrak{so}^{*}(4,\{\cdot,\cdot\}_{C})$ in the basis
of $(\ell_{12},\ell_{13},\ell_{14},\ell_{23},\ell_{24},\ell_{34})$
is given by 
\begin{align*}
B_{C} & =\left(\begin{array}{cccccc}
0 & -c_{1}\ell_{23} & -c_{1}\ell_{24} & c_{2}\ell_{13} & c_{2}\ell_{14} & 0\\
c_{1}\ell_{23} & 0 & -c_{1}\ell_{34} & -c_{3}\ell_{12} & 0 & c_{3}\ell_{14}\\
c_{1}\ell_{24} & c_{1}\ell_{34} & 0 & 0 & -c_{4}\ell_{12} & -c_{4}\ell_{13}\\
-c_{2}\ell_{13} & c_{3}\ell_{12} & 0 & 0 & -c_{2}\ell_{34} & c_{3}\ell_{24}\\
-c_{2}\ell_{14} & 0 & c_{4}\ell_{12} & c_{2}\ell_{34} & 0 & -c_{4}\ell_{23}\\
0 & -c_{3}\ell_{14} & c_{4}\ell_{13} & -c_{3}\ell_{24} & c_{4}\ell_{23} & 0
\end{array}\right).
\end{align*}
 
In \cite{Bolsinov-Oshemkov} it is shown that the Lie bundle $\lambda[\cdot,\cdot]-[\cdot,\cdot]_{C}=[\cdot,\cdot]_{\lambda I-C}$
where $\lambda\in\mathbb{R}\cup\infty$ is still a Lie bracket on
$\mathfrak{so}(4)$. Similarly, $\lambda\{\cdot,\cdot\}-\{\cdot,\cdot\}_{C}=\{\cdot,\cdot\}_{\lambda I-C}$
is also a Poisson bracket on $\mathfrak{so}^{*}(4)$ giving us a set of compatible
Poisson structures. Following Example B in \cite{Bolsinov-Oshemkov}, we can
now study our system from the perspective of the compatible Poisson
structures $\{\cdot,\cdot\}_{\lambda I-C}$ on $\mathfrak{so}^{*}(4)$. Expanding the Casimirs in terms of the parameter $\lambda$ gives commuting integrals \cite{Bolsinov-Oshemkov}.
\begin{prop}[\cite{Bolsinov-Oshemkov}]
\label{fact1}The integrals $(I_{0},I_{1},I_{2})$ for the Poisson structure 
$(\mathfrak{so}^{*}(4),\{\cdot,\cdot\}_{\lambda I-C})$ with $C=\text{diag}(c_{1},c_{2},c_{3},c_{4})$
where $c_{i}$ are real distinct constants can be obtained from the coefficients
of the numerator of the rational function given by 
\begin{align}
\psi(\lambda) & =\text{Tr}((X(\lambda I - C)^{-1})^{2})=2\frac{I_{0}\lambda^{2}+I_{1}\lambda+I_{2}}{(\lambda-c_{1})(\lambda-c_{2})(\lambda-c_{3})(\lambda-c_{4})}\label{eq:Trace formula-1}
\end{align}
where $X\in \mathfrak{so}(4)$. They are $I_{0}=-\sum_{i,j}\ell_{ij}^{2}=-2h,\ I_{1}=\sum_{i<j}(c_{n}+c_{m})\ell_{ij}^{2},\ I_{2}=-\sum_{i<j}c_{n}c_{m}\ell_{ij}^{2}$
where the indices $m,n,i,j$ are all distinct. 
\end{prop}

Define $E=\text{diag}(e_1,e_2,e_3,e_4)$ and using $C=E$
gives us the integrals $(I_{0},I_{1},I_{2})=(-2h,\eta_{1},-\eta_{2})$
obtained from separation of variables in (\ref{eq:Separation constants ellipsoidal}).
This allows us to study the ellipsoidal integrable system on the reduced
space $\mathfrak{so}^{*}(4)$ as a system of compatible Poisson structures
and (\ref{eq:Trace formula-1}) becomes 
\begin{align}
\psi(\lambda) & =2\frac{-2h\lambda^{2}+\eta_{1}\lambda-\eta_{2}}{(\lambda-e_{1})(\lambda-e_{2})(\lambda-e_{3})(\lambda-e_{4})}.\label{eq:Trace formula-1-1}
\end{align}
This is precisely the equation for the separated momenta in (\ref{eq:psq ellipsoidal-1})
with $\psi(s_{i})=4p_{i}^{2}$.
\subsection{Critical Points}

To find critical points, consider the lift of the standard endomorphism
from $\mathfrak{so}(4,\mathbb{C})$ to $\mathfrak{so}(3)\oplus \mathfrak{so}(3)$ defined by $T:\text{\ensuremath{\bm{L}\mapsto(\bm{X},\bm{Y})}}$ given in \eqref{eq:ls to Xy}, where here $\bm{X}$ and $\bm{Y}$ are complex vectors.
We will be using the map $\gamma_{\lambda I-E}$ in (\ref{eq:isomorphism}) to construct
the map $T_{2}=\gamma_{\lambda I - E} T^{-1}$  from $[\mathfrak{so}(3)\oplus \mathfrak{so}(3)]^{*}$
to $\mathfrak{so}^{*}(4,\{\cdot,\cdot\}_{C})$. It is known from \cite{Bolsinov-Oshemkov} that the set of singular
points in $\mathfrak{so}(4,\mathbb{C})\equiv \mathfrak{so}(3)\oplus \mathfrak{so}(3)$ under the standard
bracket is given by $\{(\bm{X},0)\}\cup\{(0,\bm{Y})\}$. When the
matrix $\lambda I-E$ is invertible (that is $\lambda\neq e_{i}$), the map
$\gamma_{\lambda I - E}$ is an Poisson isomorphism between $\mathfrak{so}^{*}(4,\mathbb{C})$
and $\mathfrak{so}^{*}(4,\{\cdot,\cdot\}_{\lambda I -E})$ and so $T_{2}$ is also a Poisson
isomorphism. The set of singular points of $\mathfrak{so}^{*}(4,\{\cdot,\cdot\}_{\lambda I -E})$
is the image of the singular points of $\mathfrak{so}(4,\mathbb{C})$ under $T_{2}$,
that is $T_{2}(\{(\bm{X},0)\}\cup\{(0,\bm{Y})\})$. The set of critical
points of the ellipsoidal integrable system is precisely the set of
singular points of $\mathfrak{so}^{*}(4,\{\cdot,\cdot\}_{\lambda I -E})$ by Theorem 2 in \cite{Bolsinov-Oshemkov}. We have an analogous result:
\begin{prop}
\label{Bolsinov Theorem}An element of $\bm{L}\in \mathfrak{so}^{*}(4)$ is critical
if 

1. $\lambda\ne e_{i}$ and $\bm{L}\in\Re e(T_{2}(0,\bm{X})\cup T_{2}(\bm{\bm{Y}},0))$. 

2. $\lambda=e_{i}$ and $\bm{L}$ is such that the Poisson bracket
$\{\cdot,\cdot\}_{C}$ drops rank.
\end{prop}
Using Proposition~\ref{Bolsinov Theorem} we start by finding the general
solutions for case 1 with $\lambda\ne e_{i}$. The critical points are given
by $\bm{L}\in\Re e(T_{2}(0,\bm{Y})\cup T_{2}(\bm{\bm{X}},0))$.
Let $\bm{z}=(z_1,z_2,z_3) \in \mathbb{C}^3$, define $\text{Sing}_{+}=\{(\bm{\bm{z}},0)\},\ \text{Sing}_{-}=\{(0,\bm{\bm{z}})$\}
and $\text{Sing}=\text{Sing}_{+}\cup\text{Sing}_{-}$. 
We have 
\begin{align}
T_{2}(\lambda)(\text{Sing}_{\pm})= & \left(\frac{z_{1}\sqrt{\lambda-e_{1}}\sqrt{\lambda-e_{2}}}{\sqrt{2}},\pm\frac{z_{2}\sqrt{\lambda-e_{1}}\sqrt{\lambda-e_{3}}}{\sqrt{2}},\frac{z_{3}\sqrt{\lambda-e_{1}}\sqrt{\lambda-e_{4}}}{\sqrt{2}},\right.\label{eq:t2sing}\\
 & \left.\pm\frac{z_{3}\sqrt{\lambda-e_{2}}\sqrt{\lambda-e_{3}}}{\sqrt{2}},-\frac{z_{2}\sqrt{\lambda-e_{2}}\sqrt{\lambda-e_{4}}}{\sqrt{2}},\pm\frac{z_{1}\sqrt{\lambda-e_{3}}\sqrt{\lambda-e_{4}}}{\sqrt{2}}\right),\nonumber 
\end{align}
where $z_{j}=a_{j}+ib_{j}$ with $a_i ,b_i \in \mathbb{R}$. Substituting these into the integrals gives $(I_{1},I_{2})=(4h\lambda,2h\lambda^{2})$
which is the curve $I_{2}=\frac{I_{1}^{2}}{8h}$ . This is precisely when $\lambda$ is a double root of \eqref{eq:Trace formula-1-1}. It is easily seen that the values of $\lambda$ that permit a double root
while keeping $\psi>0$ are in the interval $\lambda\in[e_{2},e_{3}]$.
After taking the real part this gives the critical points
\begin{equation}
\begin{aligned}\bm{L}_{\pm}= & \left(\frac{a_{1}\sqrt{\lambda-e_{1}}\sqrt{\lambda-e_{2}}}{\sqrt{2}},\mp\frac{b_{2}\sqrt{\lambda-e_{1}}\sqrt{e_{3}-\lambda}}{\sqrt{2}},-\frac{b_{3}\sqrt{\lambda-e_{1}}\sqrt{e_{4}-\lambda}}{\sqrt{2}},\right.\\
 & \left.\mp\frac{b_{3}\sqrt{\lambda-e_{2}}\sqrt{e_{3}-\lambda}}{\sqrt{2}},\frac{b_{2}\sqrt{\lambda-e_{2}}\sqrt{e_{4}-\lambda}}{\sqrt{2}},\mp\frac{a_{1}\sqrt{e_{3}-\lambda}\sqrt{e_{4}-\lambda}}{\sqrt{2}}\right).
\end{aligned}
\label{eq:lcurcrit}
\end{equation}
We can verify that these are critical points of the system
$(\eta_{1},\eta_{2})$ by noting that $B(\nabla\eta_{0}-\lambda\nabla\eta_{1}-\nabla\eta_{2})|_{\bm{L=L_{\pm}}}=0$.
Substituting (\ref{eq:lcurcrit}) into the Pl\"ucker relation forces
\begin{equation}
    a_{1}^{2}=b_{2}^{2}+b_{3}^{2}.\label{plu-el}
\end{equation}
Using $\bm{L\cdot L}=2h$ gives
the conic
\begin{equation}
 (e_{1}e_{2}-e_{1}e_{3}-e_{2}e_{4}+e_{3}e_{4})b_{2}^{2}+(e_{1}e_{2}-e_{2}e_{3}-e_{1}e_{4}+e_{3}e_{4})b_{3}^{2}=4h.\label{ham-el}   
\end{equation}
 The conditions \eqref{plu-el} and \eqref{ham-el}
when combined with (\ref{eq:lcurcrit}) gives the explicit parametrisation
for the 4 topological $S^{1}$ of critical points for case 1 with $\lambda\neq e_i$.

In case 2 where $\lambda=e_{i}$, the map $\gamma_{\lambda I - E}$ still exists
but it is not invertible and so $T_{2}(e_{i})(\text{Sing})$ is still a subset of
the critical points of the bracket $\{\cdot,\cdot\}_{C(e_{i})}$.
Let us consider $\lambda =e_1$, then  we have
\begin{align*}
T_{2}(e_{1})(\text{Sing}_{\pm}) & =\left(0,0,0,\pm\frac{z_{3}\sqrt{e_{1}-e_{2}}\sqrt{e_{1}-e_{3}}}{\sqrt{2}},-\frac{z_{2}\sqrt{e_{1}-e_{2}}\sqrt{e_{1}-e_{4}}}{\sqrt{2}},\pm\frac{z_{1}\sqrt{e_{1}-e_{3}}\sqrt{e_{1}-e_{4}}}{\sqrt{2}}\right).
\end{align*}

These naturally satisfies the Pl\"{u}cker relations giving us solutions
of the form $\bm{L}=(0,0,0,\ell_{23},\ell_{24},\ell_{34})$ after taking the real part with
constraint $\ell_{23}^{2}+\ell_{24}^{2}+\ell_{34}^{2}=2h$. This means
that the set of all critical points corresponding to case 2 with $\lambda=e_1$
is the sphere $\ell_{23}^{2}+\ell_{24}^{2}+\ell_{34}^{2}=2h$.

In order to show these are all the critical points for $\lambda=e_1$, recall that the singular points of $\{\cdot , \cdot\}_{C(e_i)}$ occurs when
$B_{C(e_{i})}$ drops rank. We perform the change of variables $(\bm{U},\bm{V})$
where $\bm{U}=(u_{12},u_{13,},u_{14})$ and $\bm{V}=(v_{34},v_{24},v_{23})$
with $u_{ij}=\sqrt{(e_{i}-e_{k})(e_{i}-e_{m})}\ell_{ij}$ , $v_{km}=\frac{1}{\sqrt{(e_{i}-e_{k})(e_{i}-e_{m})}}\ell_{km}$
and $i,j,k,m$ are all distinct. This transforms $B_{e_1 I - E}$ into
the standard $\mathfrak{e}^{*}(3)$ algebra given by 
\begin{align*}
B_{3} & =\begin{pmatrix}\bm{0} & \bm{\hat{U}}\\
\bm{\hat{U}} & \bm{\hat{V}}
\end{pmatrix}.
\end{align*}
Singular orbits of $B_{3}$ are given by $\bm{U}=\bm{0}$, giving
$\ell_{12}=\ell_{13}=\ell_{14}=0$ for all $j\ne i$. These are the critical points described
above.

For $\lambda=e_{2},e_{3},e_{4}$ we have an isomorphism between
$B_{\lambda I - E}$ and $\mathfrak{e}^{*}(1,2),\ \mathfrak{e}^{*}(2,1),\ \mathfrak{e}^{*}(0,3)$ respectively,
all of which have singular orbits iff $\bm{U}=\bm{0}$ giving $\ell_{ik}=0$
for all $k\neq i$ if $\lambda=e_{i}$.

\subsection{Bifurcation diagram}
Using the critical points described in the previous section we have the following result for the critical values.
\begin{cor}
\label{biham-critical-values}The critical values of the integrals
$(\eta_1,\eta_2)$ occur when $\lambda$ 
 is a  real root of
$\psi(\lambda)$ so that $\psi(\lambda)\geq0$ and 
\begin{enumerate}
\item $\lambda=e_{i}$ or
\item $\lambda$ is a double root of the numerator $\psi(\lambda)\coloneqq 2h\lambda^{2}-\eta_1\lambda+\eta_{2}$ 
\end{enumerate}
\end{cor}
\begin{proof}
By the Cayley Hamilton theorem, it is known that 
\begin{align*}
\text{Tr}(M^{2})-(\text{Tr}(M))^{2}+2\det(M) & =0.
\end{align*}
We observe that if $M=XC^{-1}$, $\text{Tr}(M)=0$ and $\det(M)=\frac{(X_{12}X_{23}-X_{13}X_{24}+X_{12}X_{34})^{2}}{(\lambda-c_{1})(\lambda-c_{2})(\lambda-c_{3})(\lambda-c_{4})}=\frac{\mathcal{C}_{2}^{2}}{(\lambda-c_{1})(\lambda-c_{2})(\lambda-c_{3})(\lambda-c_{4})}=0$
due to constraint on the Casimir. This implies that $\lambda$ has
to be a root of $\psi(\lambda)$ for valid motion. For critical points, we either have $\lambda=e_i$ or $\lambda$ such that $(I_{1},I_{2})=(4h\lambda,2h\lambda^{2})$, that is $\lambda$ is a double root of $\psi(\lambda)$.
\end{proof}
\begin{prop}
\label{The-bifurcation-diagram Lemma}The set of critical values for
the reduced ellipsoidal integrable system $(\eta_1,\eta_2):S^2\times S^2 \to \mathbb{R}^2$ is composed of 4 straight lines and a quadratic curve. The lines
are $\mathcal{L}_{i}:\eta_{2}-e_{i}(\eta_{1}-e_{i})=0$ for $i\in\{1,2,3,4\}$
and part of the parabola $\eta_{2}=\frac{\eta_{1}^{2}}{4}$ given
by $\mathcal{C}:(\eta_{1},\eta_{2})=\left(2t,t^2\right)$ for $e_{2}\le t\le e_{3}$.
There are $6$ transverse intersections of the lines $\mathcal{L}_{i}\cap\mathcal{L}_{j}$
which occur at $(\eta_{1},\eta_{2})=d_{ij}\coloneqq(e_{i}+e_{j},e_{i}e_{j})$
where $i\ne j$ and $i,j\in\{1,2,3,4\}$. The points $d_{i}\coloneqq (2e_i,e_i^2)$
where $i\in\{2,3\}$ correspond to the two tangential intersections
of $\mathcal{L}_{2}$ and $\mathcal{L}_{3}$ with $\mathcal{C}$. The bifurcation diagram with $2h=1$
is shown in Figure \ref{fig:root diagram and mm ellipsoidal} b). 
\end{prop}
\begin{proof}
Using Proposition \ref{fact1} and Corollary \ref{biham-critical-values}, when $\lambda=e_{i}$
in (\ref{eq:Trace formula-1-1}) we must have $\eta_{2}=e_{i}(\eta_{1}-2he_{i})$
for the numerator of $\psi(\lambda)$ to be identically $0$ . If
$\lambda$ is a double root, then taking the discriminant of the numerator
gives the curve $\eta_{2}=\frac{\eta_{1}^{2}}{8h}$. With $2h=1$, we obtain
the formulae for the lines $\mathcal{L}_{i}$ and the curve $\mathcal{C}$
which make up the boundary of the image of the momentum map. Since the bifurcation
diagram is necessarily compact we must also determine the regions for which the momenta are real. To do this, recall that $\psi(s_{i})=4p_{i}^{2}$ can be
factored as follows 
\begin{equation}
\psi(s_{i})=4p_{i}^{2}=-2\frac{(s_{i}-r_{1})(s_{i}-r_{2})}{(z-e_{1})(z-e_{2})(z-e_{3})(z-e_{4})},\label{eq:psq new}
\end{equation}
 where $e_{1}\le r_{1}\le r_{2}\le e_{4}$, $\eta_{1}=r_{1}+r_{2}$
and $\eta_{2}=r_{1}r_{2}.$ The denominator of (\ref{eq:psq new})
defines $4$ poles at $e_{j}$ and so divides the interval $[e_{1},e_{4}]$
into three intervals $[e_{i},e_{i+1}]$ where $i\in\{1,2,3\}$. To
distribute the roots $r_{k},$ we require that $p_{i}^{2}$ takes
on non negative values in each interval $[e_{i},e_{i+1}]$ for valid
motion. This gives $4$ regions of motion which we represent in Figure
\ref{fig:root diagram and mm ellipsoidal} a). We call this the
root diagram for the ellipsoidal system. The mapping from the root
diagram to the bifurcation diagram is smooth on the interior and all edges
of the root diagrams except on the diagonal cyan segment where $r_1=r_2$. The image of
the momentum map is the region enclosed by the lines $\mathcal{L}_{i}$
and the curve $\mathcal{C}$ presented in Figure~\ref{fig:root diagram and mm ellipsoidal}~b).

From the root diagram, we find that each of the lines $\mathcal{L}_{i}$
is defined over $\eta_{1}\in[e_{1}+e_{j},e_{4}+e_{k}],\eta_{2}\in[e_{1}e_{j},e_{4}e_{k}]$
where $j=\max(2,i)$ and $k=\min(3,i)$. Hence the end points of $\mathcal{L}_i$ are $d_{ij}$ and $d_{k4}=d_{4k}$. The critical points on each line $\lambda=e_{i}$ described in the previous section represents the geodeosic subflow on the great $2$-sphere $x_{i}=0$ under
elliptical coordinates on $S^2$ with axes given by the remaining $e_{k}$ with
$k\neq i$. Consider the case with $\lambda=e_1$, for each critical value $(\eta_{1},\eta_{2})$ on the line $\mathcal{L}_1$, its set of
critical points is the intersection of the sphere $\ell_{23}^{2}+\ell_{24}^{2}+\ell_{34}^{2}=2h$ with the ellipsoids
$\eta_{2}(\bm{L})=e_{1}(e_{4}\ell_{23}^{2}+e_{3}\ell_{24}^{2}+e_{2}\ell_{34}^{2})=\eta_{2}$.
These are precisely the fibres of the geodesic flow on
$S^{2}$ when separation of variables is performed in the elliptical
coordinates on $S^2$ with semi-axes $(e_{2},e_{3},e_{4})$ (see Appendix \ref{s2ellip}). Indeed, when $\ell_{12}=\ell_{13}=\ell_{14}=0$,
we have $x_{1}=y_{1}=0$ and we have geodesic motion restricted on
the great 2-sphere $x_{1}=0$.

In the case where there is a double root in the numerator, i.e.
$t=r_{1}=r_{2}$, we obtain the curve $\mathcal{C}:(\eta_{1},\eta_{2})=(2t,t^{2})$
where $e_{2}\le t\le e_{3}$.

It is clear that the intersections between $\mathcal{L}_{i}$ and
$\mathcal{L}_{j}$ are transverse and are located at $(\eta_{1},\eta_{2})=(e_{i}+e_{j},e_{i}e_{j})=d_{ij}$
where $i\ne j$ and $i,j\in\{1,2,3,4\}$. Similarly, it is easy to
see by computing the tangents that only $\mathcal{L}_{2}$ and $\mathcal{L}_{3}$
intersect $\mathcal{C}$ tangentially at $d_2=(2e_{2},e_{2}^{2})$ and
$d_3=(2e_{3},e_{3}^{2})$ respectively. 
\end{proof}
\begin{cor}
The Uhlenbeck integral $F_{i}=0$ if and only if $\eta_{2}-e_{i}(\eta_{1}-e_{i})=0$,
i.e. $F_{i}$ vanishes along $\mathcal{L}_{i}$. 
\end{cor}
\begin{proof}
Taking the residue at $e_{i}$ of both sides of (\ref{eq:R on A def})
gives 
\[
F_{i}=-\frac{e_{i}^{2}-e_{i}\eta_{1}+\eta_{2}}{(e_{i}-e_{k})(e_{i}-e_{l})(e_{i}-e_{m})}
\]
 where indicies $i,k,l,m$ are all distinct. Since we have assumed
all semi major axes are distinct, $F_{i}=0$ if and only if $(\eta_{1},\eta_{2})$
lie on $\mathcal{L}_{i}$.
\end{proof}

\begin{figure}
\begin{centering}
\includegraphics[width=14cm]{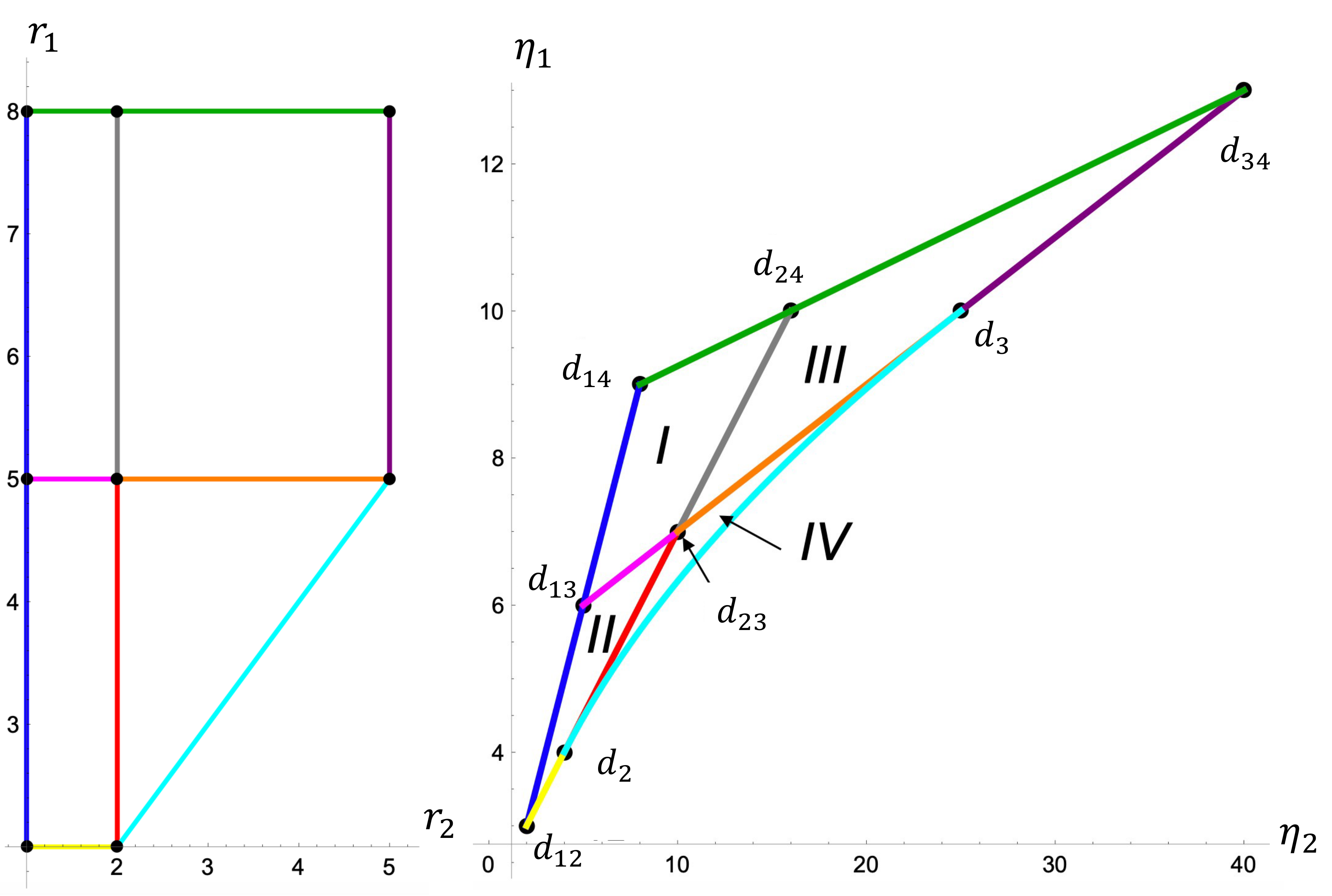}
\par\end{centering}
\caption{a) Root diagram of the reduced ellipsoidal system with $(e_{1},e_{2},e_{3},e_{4})=(1,2,5,8)$.
b) Corresponding bifurcation diagram. The four chambers of the image of the
momentum map are labelled $I-IV$. The colours of the various lines will be kept in further figures. \label{fig:root diagram and mm ellipsoidal}}
\end{figure}
To classify the nature of the critical points, we compute the eigenvalues
of the linearisation $\nabla[B(\lambda\nabla\eta_{1}+\nabla\eta_{2})]$.
The intersections $d_{12},d_{14}$ and $d_{34}$ are all elliptic-elliptic
critical values, $d_{13}$ and $d_{24}$ are of elliptic-hyperbolic
type and $d_{23}$ is hyperbolic-hyperbolic. The tangential intersections
$d_{2}$ and $d_{3}$ are degenerate. The lines $\mathcal{L}_{1},\mathcal{L}_{4}$,
the curve $\mathcal{C}$, as well as the yellow and purple parts of
$\mathcal{L}_{2}$ and $\mathcal{L}_{3}$ respectively have one pair
of imaginary eigenvalues and so are codimension one elliptic. The
magenta, orange and grey, red segments of $\mathcal{L}_{3}$ and $\mathcal{L}_{2}$
give one pair of real eigenvalues and so are codimension one hyperbolic.

\subsection{Critical Fibres}

Unlike the critical points, the parametrisation of the critical fibre
cannot be computed algebraically. We will instead provide informal description
and topological classification of the fibres on $S^2\times S^2$ instead.

Firstly, by Louvile-Arnold theorem, the preimage of regular values
in regions $I,II,III,IV$ are $\mathbb{T}^{2}$ in $S^2\times S^2$.

Since $d_{12},d_{14},d_{34}$ are elliptic-elliptic, their preimage
on $S^2\times S^2$ are 2 points each, the critical points found earlier. 

Next, consider the lines immediately connected to $d_{12},d_{14},d_{34}$
These are the lines $\mathcal{L}_{1},\mathcal{L}_{4}$, as well as
the yellow and purple parts of $\mathcal{L}_{2}$ and $\mathcal{L}_{3}$
respectively. They are codimension one elliptic and so their fibres
are circles $S^1$ and only contain critical points.
The multiplicity of these circles is two as a result of extending the
multiplicities of $d_{12},d_{14}$ and $d_{34}$.

Similarly, $d_{13}$ and $d_{24}$ are elliptic-hyperbolic critical
values. There are 2 intersecting circles of critical points in their fibres. To obtain
the full fibre, we observe that as we move along $\mathcal{L}_{1}$
(resp. $\mathcal{L}_{4})$ and pass by $d_{13}$ (resp. $d_{24})$,
two $S^1$ bifurcate into two $S^1$. Such a bifurcation is represented
by the Fomenko atom $C_{2}$.

The magenta and grey segments of $\mathcal{L}_{3}$ and $\mathcal{L}_{2}$
are extensions of $d_{13}$ and $d_{24}$ respectively, the fibres
of these segments are $S^1\times C_{2}$.

Since the grey and magneta lines both have a $C_{2}$ type singularity,
we know that the fibre of $d_{23}$ is of type $(C_{2},C_{2})$ $l$-type
of complexity $2$ with loop molecule number 17 in \cite{book}. The critical fibre is simpler, there are only 4 types after symmetry reduction \cite{DullinVuNgoc07}.

Since the $17$ saddle saddle singularity contains $4$ $B$ atoms, it follows that the fibres of the red
and orange lines are $2B\times S^1$. 

The cyan curve is codimension one elliptic and so its fibre is a circle
$S^{1}$. To find the multiplicity of these $S^1$, we note that the
fibre of $d_{23}$ contains $4S^1$. Extending $d_{23}$ into chamber
$IV$, we multiply by $S^{1}$. Hence, the fibre of a regular value
in chamber $IV$ is $4T^{2}$. Continuing onto $\mathcal{C}$, we
see that the fibre along the curve must be $4S^1$.

The fibre type does not change at the degenerate points $d_{2}$ and $d_{3}$, hence are simply $2S^1$. Approaching $d_{2}$
and $d_{3}$ along the yellow and purple lines respectively, we see
$2S^1$ bifurcate into $4S^1$ (along the curve $\mathcal{C}$)
as well as a hyperbolic fibre ($2B\times S^1$). This means $d_{2}$
and $d_{3}$ are pitchfork bifurcations as described in \cite{book}.

Extending the codimension $1$ lines $\mathcal{L}_{1}$ and $\mathcal{L}_{4}$
into chambers $I-III$ gives the following corollary. 
\begin{cor}
The fibre of a regular point on the momentum map is a torus $T^{2}$.
The multiplicity of the tori in chambers $I-III$ is $2$, while tori
in chamber $IV$ have multiplicity $4$. 
\end{cor}

\subsection{The Action Map}\label{ell-act}

Recall that the Liouville tori of the ellipsoidal integrable system
are certain coverings of the real parts of the Jacobi variety of the genus
3 hyperelliptic curve defined by $w^{2}(z)=-R(z)A(z)$. The actions of
this system are the periods of the Abelian integral
\[
I_{j}=\frac{1}{2\pi}\oint_{\gamma_{j}}p_{j}(s)ds=\frac{1}{2\pi}\oint_{\gamma_{j}}\sqrt{\frac{-R(s)}{4A(s)}}ds=\frac{1}{2\pi}\oint_{\gamma_{j}}\frac{-R(s)}{2w(s)}ds.
\]

The actions $I_{j}$ are discontinuous on phase space
as we cross boundaries of chambers of the bifurcation diagram. To construct
a set of continuous actions, we perform discrete symmetry reduction by
the $2^{4}$ discrete symmetries generated by the reflections 
\[
\sigma_{i}:(x_{i},p_{i})\to(-x_{i},-p_{i}).
\]

Following \cite{GURNeumanQuantum}, we can construct symmetry reduced actions
that are continuous accross all chambers of the momentum map. A detailed
explanation can be found in \cite{DullinNeuman}.
\begin{lem}
\label{def:action}The continuous actions of the ellipsoidal integrable
system $(J_{1},J_{2},J_{3})$ are 
\begin{equation}
\begin{aligned}J_{1}=\frac{2}{\pi}\int_{e_{1}}^{\min(r_{1},e_{2})}p(s)ds, &  & J_{2}=\frac{2}{\pi}\int_{\max(r_{1},e_{2})}^{\min(r_{2},e_{3})}p(s)ds, &  & J_{3}=\frac{2}{\pi}\int_{\max(r_{2},e_{3})}^{e_{4}}p(s)ds\end{aligned}
.\label{eq:ACtions ell dfef}
\end{equation}
This discrete symmetry reduction reduces the multiplicities of $T^2$ in all chambers to $1$ and this is the reason why the discrete symmetry reduced system has a globally continuous action map. The actions $J_{i}$ are independent on $T^*S^3$ but since $H$ is a superintegrable Hamiltonian, they are related on an energy surface, and hence become dependent for the reduced system on $S^2\times S^2$. 
\end{lem}
\begin{lem}
\label{Action lemma} The continuous actions satisfy the relation
\begin{equation}
J_{1}+J_{2}+J_{3}=\sqrt{2h}.\label{eq:action cond el}
\end{equation}
\end{lem}
\begin{proof}
Let $\beta_{i}$ be cycles that enclose the bounds of $J_{i}$ respectively.
I.e. $\beta_{1}$ encloses the interval $[e_{1},\min(r_{1},e_{2})]$ and similarly
for $\beta_{2},\beta_{3}$. By deforming the cycles on the hyperelliptic
curve $w^{2}=-R(z)A(z)$ we have 
\begin{equation}
J_{1}+J_{2}+J_{3}=-\frac{1}{2\pi}\oint_{\gamma}pdz\label{eq:sum of actions}
\end{equation}
 where $\gamma$ is a cycle that encircles the point at infinity.
It is easily shown that 
\begin{equation}
\text{Res}(p,\infty)=-\sqrt{-2h}.\label{eq:Residue at infinity}
\end{equation}
Combining (\ref{eq:Residue at infinity}) with (\ref{eq:sum of actions})
gives the desired result.
\end{proof}
\begin{thm}
The image of the action map \eqref{eq:ACtions ell dfef} is an equilateral triangle $\mathcal{T}$
(see Figure \ref{fig:Action map}).\label{action thm}
\end{thm}
\begin{proof}
From Lemma 5, we know that the image of the action map is constrained
to the plane $J_{1}+J_{2}+J_{3}=\sqrt{2h}$. This is bounded by $J_{i}\geq0$
and hence the image is contained in the intersection of the plane
$J_{1}+J_{2}+J_{3}=\sqrt{2h}$ with the positive quadrant. Since the
maps $J_{i}$ are continuous on the reduced phase space, every point
in the interior of $\mathcal{T}$ must be a point in the image of
the action map. On the boundary of $\mathcal{T},$ the lines $J_{1}=0$ and $J_3=0$ are the image of the lines $\mathcal{L}_1$ and $\mathcal{L}_4$ respectively while $J_2=0$ is the image of the cyan curve $\mathcal{C}$ together with the yellow and purple segments of $\mathcal{L}_2$ and $\mathcal{L}_3$. Thus, the triangle $\mathcal{T}$ formed
by intersecting the plane $J_{1}+J_{2}+J_{3}=\sqrt{2h}$ with the
positive quadrant is the image of the symmetry reduced phase space
under the action map.
\end{proof}

The action map calculated using (\ref{eq:ACtions ell dfef}) is shown
in Figure \ref{fig:Action map} for $h=\frac{1}{2}$. Let lines
$J_{i}=0$ be $\mathfrak{J}_{i}$ and call the interior
lines $\gamma_{1}$ (red and grey) and $\gamma_{2}$ (magenta and
orange). We denote by $A_{ij}$ the intersection of $\mathfrak{J}_{i}$
and $\gamma_{j}$. Similar polytopes are known to classify toric systems
(two dimensional integrable systems where both integrals are global
$S^{1}$ actions). For more details on this, see \cite{BSMF_1988__116_3_315_0}.
In our case, the $\gamma_{i}$ in the interior of the action map reflect
the non-toric nature of the ellipsoidal integrable system. Along these lines the hyperelliptic action integrals \eqref{def:action} become elliptic. Let 
\[
\mathcal{T}(u,v,\alpha)=\frac{(u-v)K(k)+(e_{4}-e_{1})\Pi(\alpha,k)}{\pi\sqrt{(e_{1}-e_{3})(e_{2}-e_{4})}}
\]
where $k^{2}=\frac{(e_{4}-e_{3})(e_{2}-e_{1})}{(e_{4}-e_{2})(e_{3}-e_{1})}$
and $K(k),\ \Pi(\alpha,k)$ are the complete elliptic integrals of the
first and third kind respectively. The tangential intersections of
$\mathfrak{J}_{2}$ with $\gamma_{1},\gamma_{2}$ occur at $A_{21}$
and $A_{22}$. These are given by 
\[
\begin{aligned}A_{21}=(\mathcal{T}(e_{2},e_{4},\alpha_{1}),0,\mathcal{T}(e_{1},e_{2},\alpha_{2})), &  & A_{22}=(\mathcal{T}( & e_{3},e_{4},\alpha_{1}),0,\mathcal{T}(e_{1},e_{3},\alpha_{2}))\end{aligned}
\]
 where $\alpha_{1}=\frac{e_{2}-e_{1}}{e_{2}-e_{4}}$ and $\alpha_{3}=\frac{e_{4}-e_{3}}{e_{1}-e_{3}}$.
Transverse intersections of $\mathfrak{J}_{i}$ with $\gamma_{j}$
occur at $A_{31}$ and $A_{12}$ given by
\[
\begin{aligned}A_{31}=\frac{2}{\pi}(\sin^{-1}(u_{1}),\cos^{-1}(u_{1}),0) &  & A_{12} & =\frac{2}{\pi}(0,\sin^{-1}(u_{2}),\cos^{-1}(u_{2}))\end{aligned}
\]
where $u_{1}=\sqrt{\frac{e_{1}-e_{2}}{e_{1}-e_{3}}}$ and $u_{2}=\sqrt{\frac{e_{2}-e_{3}}{e_{2}-e_{4}}}$.
The intersection of $\gamma_{1}$ and $\gamma_{2}$ is located at
\begin{equation} \label{eq:HHptaction}
(J_{1},J_{2},J_{3})=\frac{2}{\pi}(\sin^{-1}\left(v_{1}\right),\sin^{-1}\left(v_{2}\right)-\sin^{-1}\left(v_{1}\right),\cos^{-1}\left(v_{2}\right))
\end{equation}
 where $v_{1}=\sqrt{\frac{e_{1}-e_{2}}{e_{1}-e_{4}}}$ and $v_{2}=\sqrt{\frac{e_{1}-e_{3}}{e_{1}-e_{4}}}$. Notice that these intersection points are invariant under affine transformations of $(e_1,e_2,e_3,e_4)$ as expected.

\begin{figure}
\begin{centering}
\includegraphics[width=8cm]{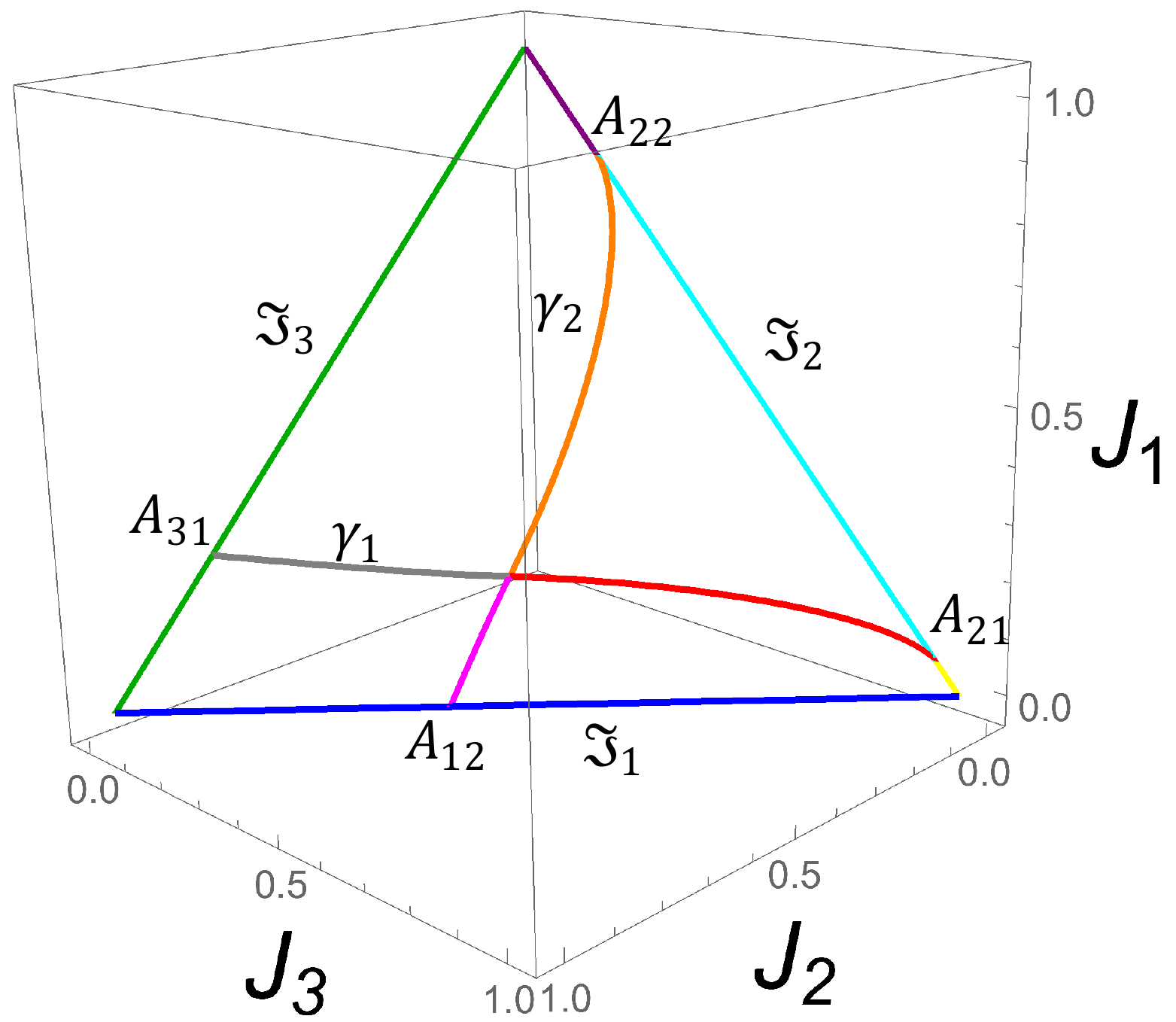} $\quad$
\includegraphics[width=7cm]{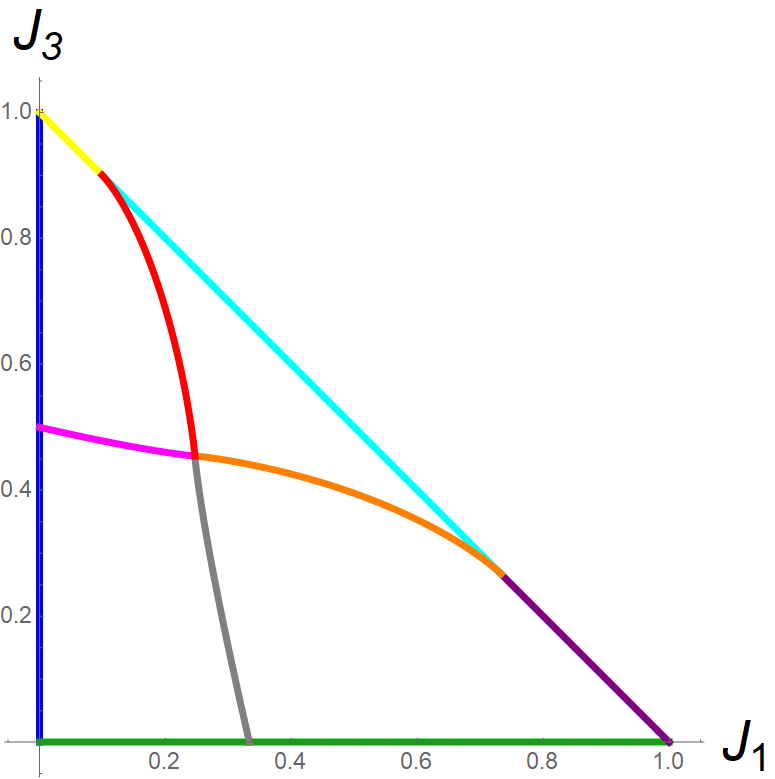}
\par\end{centering}
\caption{a) Action map of the ellipsoidal integrable system with $(e_{1},e_{2},e_{3},e_{4})=(1,2,5,8)$. b) Corresponding reduced action map $(J_1,J_3)$ \label{fig:Action map}}
\end{figure}
The action map $(J_1,J_2,J_3)$ is equivalent to the map $(J_1,J_1+J_2+J_3,J_3)$ by a uni-modular transformation. Since $J_1+J_2+J_3=1$ for the reduced system on $S^2\times S^2$, we have 2 actions only, ie. $(J_1,J_3)$. By using similar transformations, any pair $(J_i,J_j)$ can be chosen as the actions for the reduced system. The resulting image of the corresponding action map is the projection of the action map in Figure \ref{fig:Action map} onto the $J_i J_j$ plane.

\begin{cor}
    A possible set of actions for the reduced ellipsoidal integrable system is $(J_1,J_3)$. The image of the reduced space $S^2\times S^2$ under this map is a right-angled isosceles triangle obtained from projecting Figure \ref{fig:Action map} onto the $J_2=0$ plane.
\end{cor}

\begin{proof}
    The reduction can be done in appropriate action variables directly. Since $J_1 + J_2 + J_3$ is a global action variable that is equal to the square root of twice the Hamiltonian, reduction means the following two things. Fix the action (i.e.~fix the energy), and quotient by its flow. The flow of this action only changes it's conjugate angle, and so the quotient identifies this angle to a point. The remaining system with two degrees of freedom has action variables $J_1, J_3$, and the image of the action map of the reduced system is the projection of the ``spatial'' fixed energy triangle in Figure~\ref{fig:Action map}~a) onto the appropriate coordinate plane in Figure~\ref{fig:Action map}~b).
\end{proof}

The choice of which action variable to present the reduced system in is somewhat arbitrary, and we prefer not to make any choice and hence keep showing the ``spatial'' picture of the triangle in $J_1J_2J_3$-space in the following section, even when discussing the reduced system on $S^2\times S^2$. In general an integrable system can be represented by its energy surface in action space, which means the surface $H(J_1, J_2, J_3) = h$, which will depend on $h$, and may not even be continuous. However, in our setting we have simplest possible case of a maximally superintegrable system for which $H(J_1+J_2+J_3) = h$, and so the triangle in action space occurs for every superintegrable system for which globally continuous actions can be defined.

The position of the hyperbolic-hyperbolic point in the image of the action map (the intersections of the lines $\gamma_1$ and $\gamma_2$) is uniquely determined by the reduced parameters $1 < a < b$ of the system. In fact, the map $(1/b, a/b) \mapsto \frac{2}{\pi}(\sin^{-1}\sqrt{1/b}, \cos^{-1}\sqrt{a/b})$ maps the triangle in parameter space Figure~\ref{fig:iab}~a) to the triangle in action space Figure~\ref{fig:Action map}~b). The involution in parameter space \eqref{eq:involution} becomes the reflection across the diagonal $(J_1, J_3) \mapsto (J_3, J_1)$.

Performing the affine transformation of $e_i$ directly in the action integral and applying the same transformation to the integration variable $s$ does change the integral. However, then also transforming $(\eta_1, \eta_2)$ according to Lemma~\ref{eta lemma} recovers the original integral, as expected. 
% When considering the involution \eqref{eq:involution}} together with $s \mapsto -s(b-a) + b$ the action integral is mapped into the action integral with parameters $a', b'$, and in addition $J_1$ and $J_3$ are 

\section{Degenerate Systems on $S^{2}\times S^{2}$}\label{sec:degen}

In this section, we study all systems arising from separating the
geodesic flow in degenerate coordinates on $S^{3}$. We begin by focusing
on the following $3$ systems: prolate $(1\ (2\ 3)\ 4)$, oblate $(1\ 2\ (3\ 4))$
and Lam\'{e} $(1\ (2\ 3\ 4))$. These correspond to the edges of
the Stasheff polytope. Further degenerations of these coordinates,
cylindrical $((1\ 2)(3\ 4))$ and two forms of spherical $(1\ (2\ (3\ 4)),(1\ ((2\ 3)\ 4))$
form the corners of the polytope.

It will be shown that the integrable systems (reduced and un-reduced) corresponding to these degenerate coordinate systems
can also be obtained by smoothly deforming their ellipsoidal counterparts. Thus, we will  establish the analogue of the result by  Sch\"{o}bel and Veselov \cite{Schoebel2014}, namely that the correct moduli space for this family of integrable systems is the Stasheff polytope.

The main feature of the degenerate systems corresponding to the lower-dimensional faces of the Stasheff polytope is the appearance of global symmetries, in the case of $S^3$ either $SO(2)$ and $SO(3)$.
The results of this section described in detail below can be sumarised in the following theorem.
Consider the designation of a separable coordinate system on $S^3$ by pairs of nested brackets inserted into 4 objects as shown in \figone.

\begin{thm}
    For each pair of brackets that enclose two adjacent members, the corresponding (reduced) integrable system has an $SO(2)$ symmetry. For each pair of brackets that enclose three adjacent members the corresponding (reduced) integrable system has a global $SO(3)$ symmetry.
The generic ellipsoidal integrable system with quadratic integrals degenerates to an integrable system with quadratic integrals. If there is an $SO(k)$ symmetry the corresponding quadratic integral
is replaced by its square root. For $SO(2)$ this gives a global $S^1$ action, while for $SO(3)$ this gives an almost global $S^1$ action. 
\end{thm}

In particular, the oblate, prolate, and spherical systems have one global $S^1$ action each, the cylindrical system has two global $S^1$ actions, and the spherical and the Lam\'e system have an almost-global $S^1$ action each. In addition, we find that the prolate system is generalised semi-toric and the cylindrical system is toric but the $S^1 \times S^1$ action is not effective. By almost global $S^1$ action we mean that the action fails in the preimage of an isolated point of the image of the momentum map. The corresponding spherical singularity 
in the spherical system \cite{RonanThesis} also appears in the Lam\'e system.

\subsection{Prolate Coordinates }

We begin by considering prolate coordinates. It will
be shown that the corresponding integrable system is generalised semi-toric and has non-trivial monodromy. In addition, the global action triangle is half of the semi-toric polygon invariant.

\subsubsection{Separation of Variables}

Prolate coordinates on $S^{3}$, denoted by $(1\ (2\ 3)\ 4)$, are
a degeneration of ellipsoidal coordinates arising from setting the
middle two semi major axes equal, i.e. $e_{2}=e_{3}$. We normalise the $e_{j}$ according to $(e_{1},e_{2}=e_{3},e_{4})=(0,1=a,b)$.
From \cite{Kalnins1986}, an explicit representation of prolate coordinates
is 
\[
\begin{aligned}x_{1}^{2} & =\frac{s_{1}s_{3}}{b}, &  & x_{2}^{2}=-\frac{\left(s_{1}-1\right)s_{2}\left(s_{3}-1\right)}{b-1},\\
x_{3}^{2} & =\frac{\left(s_{1}-1\right)\left(s_{2}-1\right)\left(s_{3}-1\right)}{b-1}, &  & x_{4}^{2}=\frac{\left(b-s_{1}\right)\left(b-s_{3}\right)}{(b-1)b},
\end{aligned}
\]
where $0\le s_{1},s_{2}\le1\le s_{3}\le b$. Since $s_2$ is an ignorable coordinates, the Hamilton Jacobi equation can be separated easily to give integrals
$(2H,G_{pro},\ell_{23})$ where $G_{pro}=b\ell_{12}^{2}+b\ell_{13}^{2}+\ell_{14}^{2}$. 
The corresponding momenta are

\begin{equation}
\begin{aligned}p_{i}^{2} & =\frac{-2hs_{i}^{2}+(g+2h+(b-1)l^{2})s_{i}-g}{4s_{i}(s_{i}-b)(s_{i}-1)^{2}}, &  &  & p_{2}^{2} & =\frac{l^{2}}{4s_{2}(1-s_{2})}\end{aligned}
\label{eq:sep momenta prolate}
\end{equation}
where $i\in\{1,3\}$ and $(l,g)$ are the
values of $\ell_{23}$ and $G_{pro}$, respectively. We call the triple $(2H,\ell_{23},G_{pro})$ on
$T^{*}S^{3}$ the $1$-parameter family of prolate integrable systems. Similarly, $(\ell_{23},G_{pro})$ gives a $1$-parameter family of reduced prolate integrable systems on $S^2\times S^2$.

Note that
we have chosen $\ell_{23}$ as an integral since it is naturally a
global $S^{1}$ action, unlike its square.
The quadratic integrals can be obtained as a limit of the ellipsoidal integrable system.

\begin{lem}
\label{Prolate degen proof }The integrals $\left(\ell_{23}^2,G_{pro}\right)$
as well as the separated momenta \eqref{eq:sep momenta prolate} can
be obtained by smoothly degenerating their ellipsoidal counterparts
\eqref{eq:Separation constants ellipsoidal} and \eqref{eq:psq ellipsoidal-1}.
\end{lem}
\begin{proof}
The transformation from ellipsoidal to prolate coordinates
is given by 
\begin{equation}
\begin{aligned}e_{3} & =e_{2}+\epsilon &  &  & s_{2} & =e_{2}+\epsilon\tilde{s}_{2} & & &p_2=\frac{\tilde{p}_2}{\epsilon}
\end{aligned}
\label{eq:legit pro trans}
\end{equation}

in the limit $\epsilon\to0$ where $\tilde{s}_{2}\in[0,1]$. The transformation from $(s_2,p_2)$ to $(\tilde{s_2},\tilde{p_2})$ is canonical.

Let $(\tilde{\eta}_{1},\tilde{\eta}_{2})=\left.(\eta_{1},\eta_{2})\right|_{e_{3}=e_{2}+\epsilon}.$ Substituting (\ref{eq:legit pro trans}) into (\ref{eq:Separation constants ellipsoidal})
and taking the limit gives 
\[
\begin{aligned}G_{pro} & =\tilde{\eta}_{2} &  &  & \ell_{23}^{2} & =\frac{1}{b-1}(\tilde{\eta}_{1}-\tilde{\eta}_{2}-2H)\end{aligned}
\]
where we have normalised the $e_{j}$ by setting $(e_{1},e_{2}=e_{3},e_{4})=(0,1,b)$.

For the separated equations, we insert (\ref{eq:legit pro trans})
into (\ref{eq:psq ellipsoidal-1}) and expand about $\epsilon=0$
to obtain
\begin{equation}
\begin{aligned}\tilde{p_{i}}^{2} & =\frac{\tilde{\eta}_{1}+s_{i}\left(s_{i}-\tilde{\eta}_{2}\right)}{4\left(e_{1}-s_{i}\right)\left(s_{i}-e_{2}\right){}^{2}\left(s_{i}-e_{4}\right)}+O(\epsilon) &  &  & \tilde{p}_{2}^{2} & =\frac{e_{2}(e_{2}-\tilde{\eta}_{2})+\tilde{\eta}_{1}}{4\tilde{s}_{2}(\tilde{s}_{2}-1)(e_{1}-e_{2})(e_{2}-e_{4})}+O(\epsilon).\end{aligned}
\label{eq:V pro}
\end{equation}
Taking the limit of (\ref{eq:V pro}) as $\epsilon\to0$ and dropping
the tildes gives (\ref{eq:sep momenta prolate}).
\end{proof}

In the prolate limit, $F_{2}$ and $F_{3}$ become singular. However,
multiplying both integrals by $(e_{3}-e_{2})$ gives 
\[
\lim_{\epsilon\to0}(e_{3}-e_{2})F_{2}=-\lim_{\epsilon\to0}(e_{3}-e_{2})F_{3}=\ell_{23}^{2}.
\]
The other two integrals $F_{1}$ and $F_{4}$ degenerate smoothly
to 
\[
\begin{aligned}F_{1,pro}=\frac{-G_{pro}}{b}, &  & F_{4,pro} & =\end{aligned}
\frac{2Hb-G_{pro}-b\ell_{23}^{2}}{(b-1)b}.
\]

\subsubsection{Critical Points and Momentum Map}
Since the integrals of the prolate system are significantly simpler than those of the ellipsoidal system, we can easily compute the critical points and values directly. However, it is interesting to note that we can also use the method of compatible Poisson structures with the matrix
$C=\text{diag}(0,1,1+\epsilon,b)$ for $0<\epsilon<b-1$ for this computation.
\\
Using Proposition \ref{fact1} with this $C$,
we get 
\[
\psi_{pro}(\lambda)=2\frac{-2h\lambda^{2}+I_{1}\lambda+I_{2}}{\lambda(\lambda-1)(\lambda-1-\epsilon)(\lambda-b)}
\]
with $(I_{1},I_{2})=(G_{pro}+2h+(b-1)\ell_{23}^{2}+\epsilon(\ell_{12}^{2}+\ell_{14}^{2}+\ell_{24}^{2}),G_{pro}+\epsilon(b\ell_{12}^{2}+\ell_{14}^{2}))$. While the integrals $I_1$ and $I_2$ may appear complicated, the system $(I_1,I_2)$ is equivalent to the system $(G_{pro},\ell_{23}^2)$ in the limit as $\epsilon\to 0$. We can find the critical points and values of the system $(I_1,I_2)$ then taking the
limit as $\epsilon\to0$ at the end of calculations to recover the correct results for the system $(G_{pro},\ell_{23}^2)$.

\begin{figure}
\begin{centering}
\includegraphics[width=7cm,height=6cm]{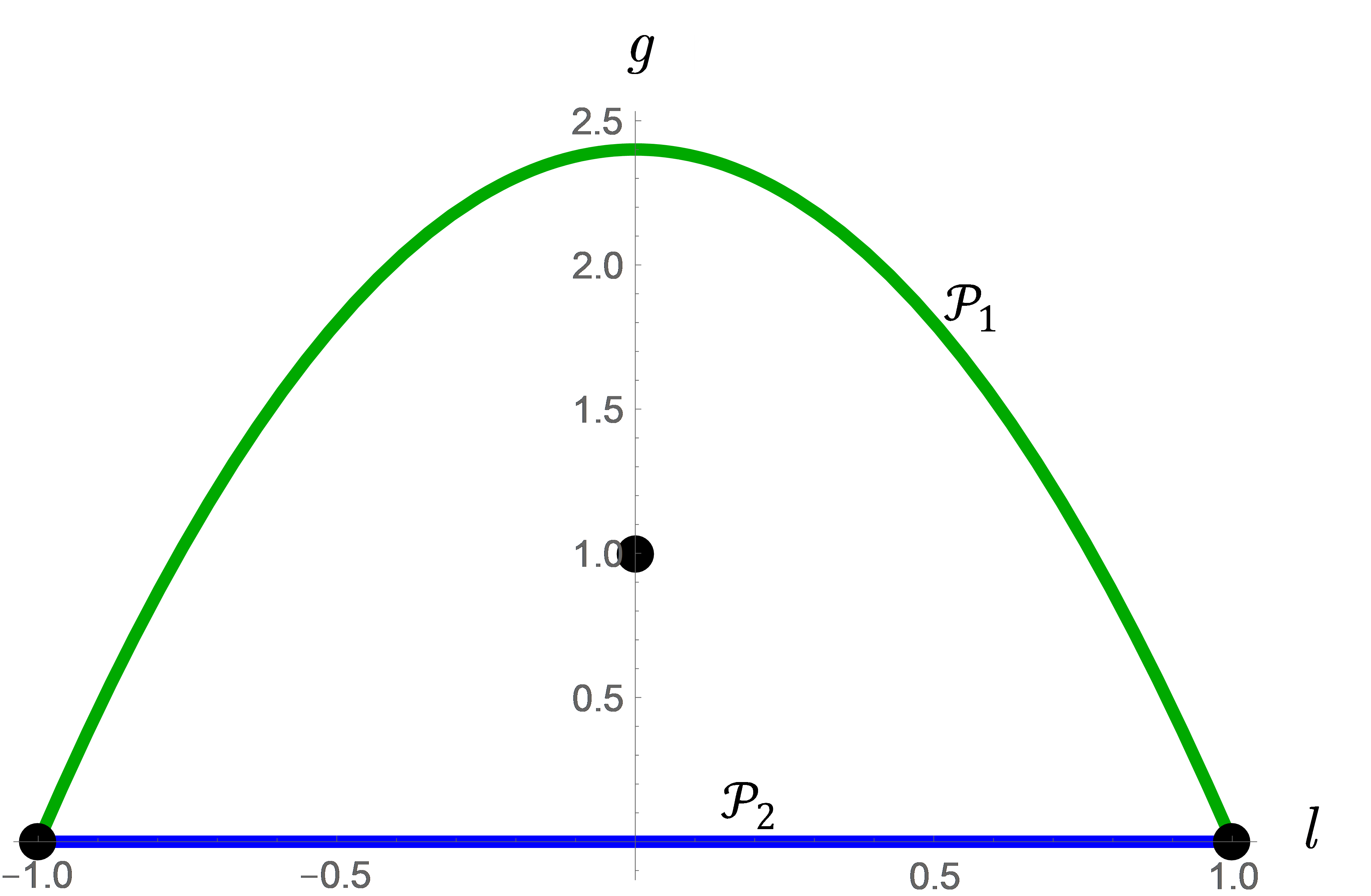}\quad\includegraphics[width=7cm,height=6cm]{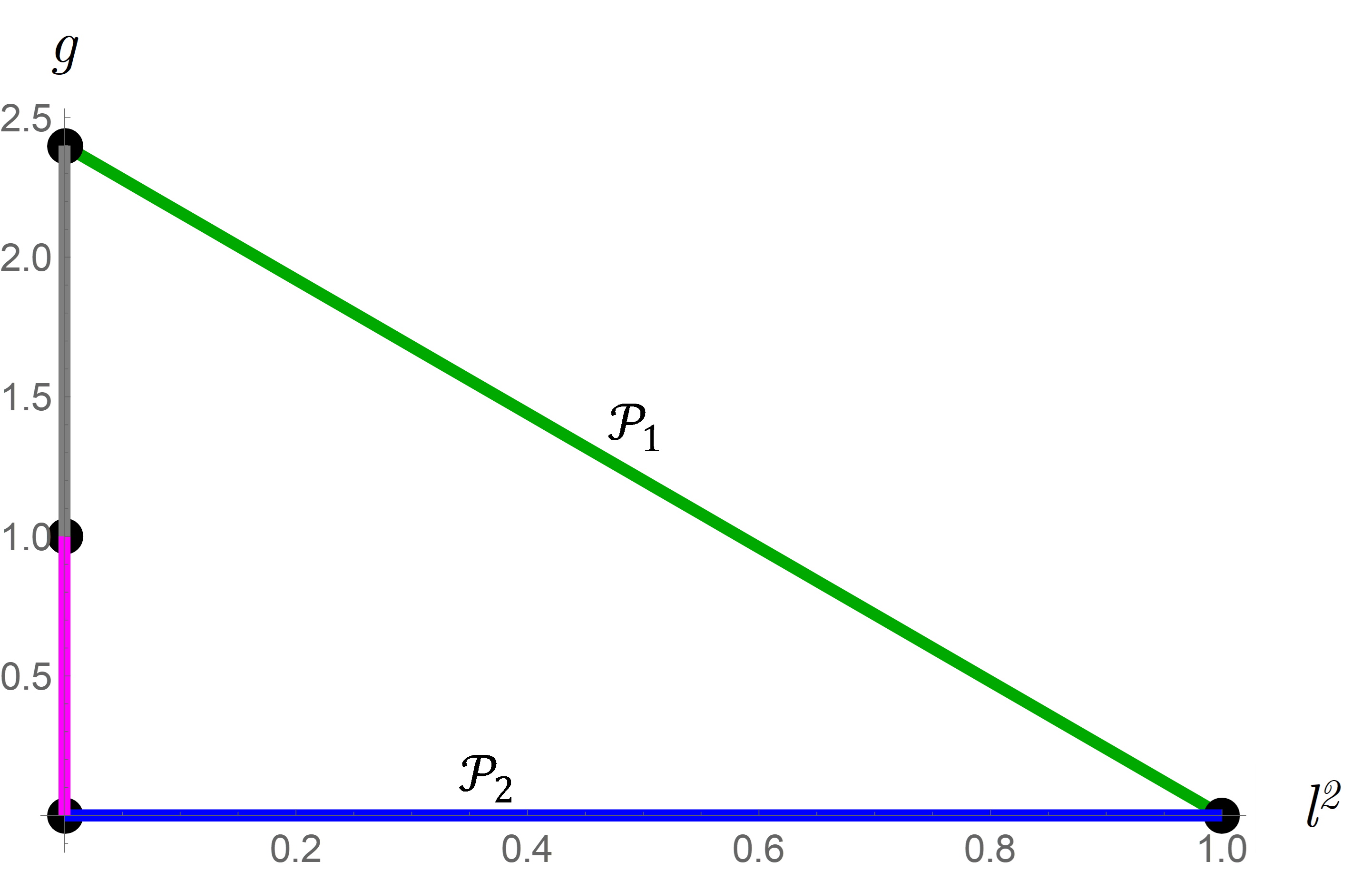}\caption{a) ``Opened'' momentum map with $b=2.4$. b) ``Unopened'' momentum
map using $\ell_{23}^{2}$ as an integral. \label{fig:Prolate Root and Momentum map}}
\par\end{centering}
\end{figure}

\begin{prop}
The momentum map for the reduced prolate integrable system is the region bounded
by the curve $\mathcal{P}_{1}:G_{pro}=b(1-\ell_{23}^{2})$ and line
$\mathcal{P}_{2}:G_{pro}=0$ shown in Figure~\ref{fig:Prolate Root and Momentum map}~a). There is an isolated critical value at $(0,1)$. 
\end{prop}
\begin{proof}
Critical points of the system $(\ell_{23}^2,G_{pro})$ can be computed directly or by applying Proposition~\ref{Bolsinov Theorem}:
\begin{enumerate}
\item The blue line $\mathcal{P}_{2}:G_{pro}=0$ has $\lambda=0$ and the critical
points are parametrised by $\bm{L}=(0,0,0,\ell_{23},\ell_{24},\ell_{34})$
with $\ell_{23}^{2}+\ell_{24}^{2}+\ell_{34}^{3}=1$ and $\ell_{23}=l$.
These are co-dimension 1 elliptic points.
\item The green line $\mathcal{P}_{1}:G_{pro}=b(1-\ell_{23}^{2})$ in Figure \ref{fig:Prolate Root and Momentum map} b) has $\lambda=a$
and the critical points are parametrised by $\bm{L}=(\ell_{12},\ell_{13},0,\ell_{23},0,0)$
with $\ell_{12}^{2}+\ell_{13}^{2}+\ell_{23}^{3}=1$ and $\ell_{23}=l$.
These are co-dimension 1 elliptic point.
\item The lines $\ell_{23}^2=0$ has both $\lambda=1$, $\lambda=1+\epsilon$ as well as the curve $I_{2}=\frac{I_{1}^{2}}{8h}$ corresponding to the double root $\lambda\in[1,1+\epsilon]$.  These are degenerate critical values of system $(\ell_{23}^2,G_{pro})$. 
\end{enumerate} 
For the reduced prolate integrable system $(\ell_{23},G_{pro})$, we see that the line $\mathcal{P}_2$ remains critical and has $l\in[-1,1]$. The line $P_1$ for $(\ell_{23}^2,G_{pro})$ becomes a parabola for $(\ell_{23},G_{pro})$. The line $\ell_{23}=0$ becomes regular values after changing from $\ell_{23}^2$ to $\ell_{23}$ with the exception of the isolated point $(\ell_{23},G_{pro})=(0,1)$ which has critical
points $\ell_{14}=\pm 1$. This
is a focus-focus point and its fibre on $S^{2}\times S^{2}$ a doubly
pinched torus.
\end{proof}

\subsubsection{Action Map and Monodromy}\label{pro-act}

From (\ref{eq:sep momenta prolate}) and the same reasoning used to
obtain (\ref{eq:ACtions ell dfef}), we have the following formulae
for the actions
\begin{equation}
\begin{aligned}J_{1}=\frac{2}{\pi}\int_{0}^{\min(r_{1},1)}p_{1}ds &  & J_{2} & =\frac{2}{\pi}\int_{0}^{1}p_{2}ds &  &  & J_{3}=\frac{2}{\pi}\int_{\max(1,r_{2})}^{a}p_{3}ds\end{aligned}
\label{eq:action prol}
\end{equation}
where the $p_{k}$ are given in (\ref{eq:sep momenta prolate}).
Here $(r_{1},r_{2})$ are the roots of $p_{1,3}^{2}$ from (\ref{eq:sep momenta prolate})
where $0\le r_{1}\le1\le r_{2}\le b$. Note that $J_{2}$ simplifies
to $\left|\ell_{23}\right|$. Like for the ellipsoidal system, the
prolate actions also satisfy (\ref{eq:action cond el}). The action
map for the prolate system is shown in Figure \ref{fig:Action and polygon invariant}
a) where the black dot corresponding to the focus-focus point is located
at $\frac{2}{\pi}(\sin^{-1}(\frac{1}{b}),0,\frac{\pi}{2}-\sin^{-1}(\frac{1}{b}))$. 

\begin{figure}
\begin{centering}
\includegraphics[width=7cm,height=6cm]{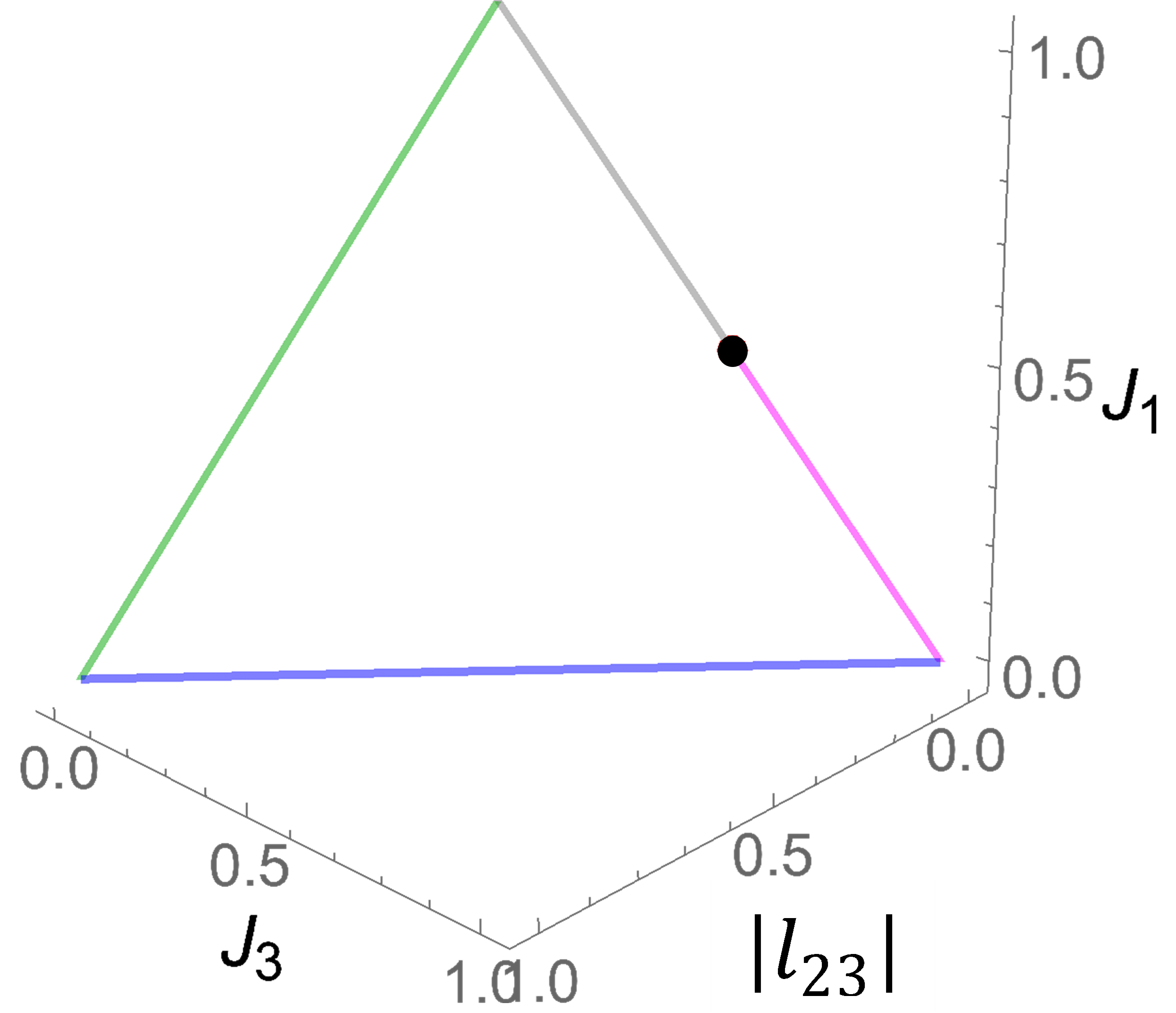}\quad\includegraphics[width=6.5cm,height=6cm]{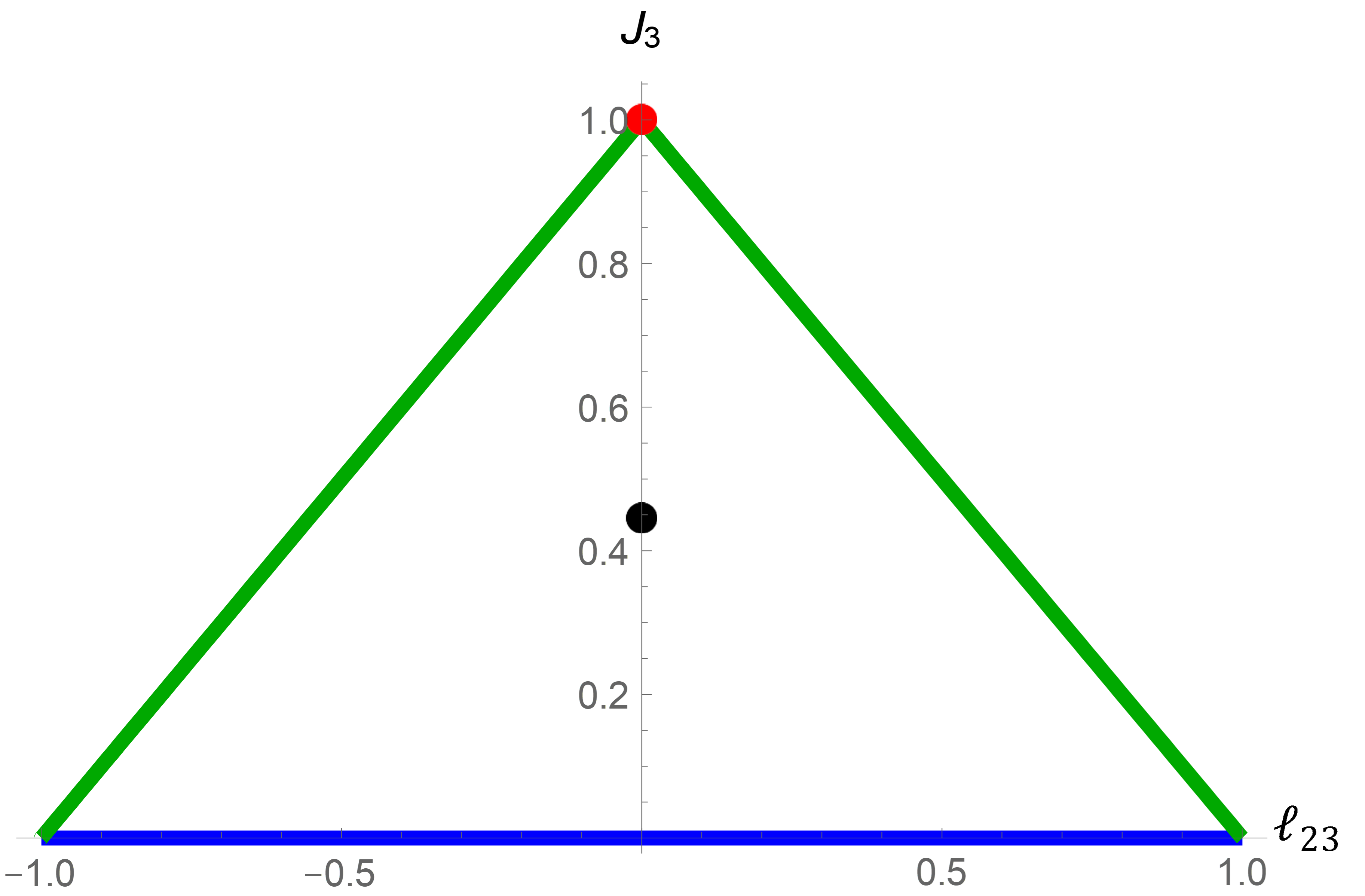}
\par\end{centering}
\caption{a) Action map for the prolate system with $b=2.4$ where the black dot is the image of the focus-focus
point. b) Semi-toric polygon invariant for the reduced prolate system. \label{fig:Action and polygon invariant}}
\end{figure}

The reduced prolate system is a two degree of freedom integrable system
where one of the integrals is a global $S^{1}$ action and all singularities
are either elliptic or focus-focus type. Thus, we have the following.
\begin{cor}
The reduced prolate system $(\ell_{23},G_{pro})$ on $S^{2}\times S^{2}$
is a generalised semi-toric system.
\end{cor}
Semi-toric systems have been globally classified using $5$ symplectic
invariants \cite{Pelayo2009}. One of these is the polygon invariant,
which is a family of rational convex polygons. This is a generalisation
of the Delzant polytope (see, e.g., \cite{BSMF_1988__116_3_315_0})
and allows us to compare the standard affine structure of $\mathbb{R}^{2}$
with that of the momentum map \cite{Alonso2019,Sepe2017}.
In Figure \ref{fig:Action and polygon invariant} b) we show one representative of the polygon invariant. This is simply the projection of the action map onto the $(\ell_{23},J_{1})$
axes with both signs of $\ell_{23}$ considered. The red vertex at
$(0,1)$ is a fake corner and is the result of ``opening up'' from $|\ell_{23}|$ to $\ell_{23}$. 
For more information on the classification of semi-toric systems,
see \cite{Sepe2017,Alonso2019}. 

Another symplectic invariant of a semi-toric system is the height invariant, which is the position of the focus-focus point in the image of the action map. Note that this is the limit of the image of the hyperbolic-hyperbolic point in the action map for the degeneration $a=1$. Specialising \eqref{eq:HHptaction} to this  case gives $\frac{2}{\pi} \cos^{-1} \sqrt{1/b}$ for the height invariant.

Another property of semi-toric systems
is the non-trivial monodromy of the actions; the focus-focus
equilibrium implies that one can only locally construct a smooth set
of action variables. Monodromy has been well studied, both classically
and quantum mechanically, see e.g. \cite{Dawson2022,Dullin2016,DAVISON20072437,Chiscop_2019}.
The prolate system has non-trivial monodromy. 
\begin{lem}
The reduced prolate system has non-trivial monodromy with monodromy matrix
\begin{equation}
\mathfrak{M}=\left(\begin{array}{ccc}
1 & 2 & 0\\
0 & 1 & 0\\
0 & -2 & 1
\end{array}\right).\label{eq:Monodromy matrix}
\end{equation}
\end{lem}
\begin{proof}
Let $C_{1}$ and $C_{3}$ be cycles that enclose the intervals $[0,\min(1,r_{1})]$
and $[\max(1,r_{2}),b]$ respectively. We rewrite (\ref{eq:action prol})
as 
\begin{equation}
\begin{aligned}J_{1}\coloneqq\frac{1}{2\pi}\oint_{C_{1}}p_{1}ds_{1}, &  & J_{3}\coloneqq\frac{1}{2\pi}\oint_{C_{3}}p_{3}ds_{3}\end{aligned}
.\label{eq:action prolate}
\end{equation}
We observe that if $(l,g)=(0,g>1)$ then $r_{1}\to1^{-}$.
while if $g_{pro}<1$, then $r_{2}\to1^{+}$. We now show
that the actions \eqref{eq:action prolate}, while continuous everywhere, are not globally smooth.
Consider the slope of the action $J_i$ considered as a function of $l$,  
\[
2\pi W_{i}\coloneqq\frac{\partial J_{i}}{\partial l}=\oint_{C_{i}}\frac{(b-1)l}{2(b-s_{i})(s_{i}-1)^{2}p(s_{i})}ds_{i}.
\]
When $l=0$, if $C_{i}$ encloses $s_{i}=1$ then $W_{i}=-\text{sgn}(l)$,
otherwise it vanishes. From our analysis of $C_{1}$ and $C_{2}$
around $s_{i}=1$ we have
\begin{equation}
\begin{aligned}\lim_{l\to0}\frac{\partial J_{1}}{\partial l}= & -\kappa_{1}\text{sgn}(l) &  & \lim_{l\to0}\frac{\partial J_{3}}{\partial l}=-\kappa_{3}\text{sgn}(l)\end{aligned}
\label{eq:K1 <-1}
\end{equation}
where $(\kappa_{1},\kappa_{3})=(1,0)$ when $g_{pro}<1$, otherwise
$(\kappa_{1},\kappa_{3})=(0,1)$. Thus, the actions $J_{1}$ and $J_{3}$
are continuous but not differentiable at $l=0$. 

For $l>0$, let $J_{+}=\left(J_{1},J_{2},J_{3}\right)^{t}$
and similarly for $J_{-}$. Note that $J_{1}$ and $J_{3}$ are even
functions of $l$ while $J_{2}$ is odd. This means $J_{-}(-l)=SJ_{+}(l)$
where $S=\text{diag}(1,-1,1)$. We are now interested in finding unimodular
matricies $M_{1},M_{2}\in SL(3,\mathbb{Z})$ such that $J_{+}$ and
$M_{i}J_{-}$ are locally smooth across $l=0$. To ensure continuity at $l=0$
we require
\[
\begin{aligned}J_{+}=M_{1}J_{-}=M_{1}SJ_{-}=M_{1}J_{+}, &  &  & g>1\\
J_{+}=M_{2}J_{-}=M_{2}SJ_{-}=M_{2}J_{+}, &  &  & g<1.
\end{aligned}
\]
The above relations imply that $\left(J_1,0,J_{3}\right)^{t}$ is
an eigenvector of both $M_{1}$ and $M_{2}$ with eigenvalue $+1$.
For arbitrary $J_{1},J_{3}$ the corresponding eigenvector equation implies
that $M_{i}$ has the form 
\[
M_{i}=\begin{pmatrix}1 & \alpha_i & 0\\
 0& 1 & 0\\
 0& \beta_{i} & 1
\end{pmatrix}.
\]
For the actions to be smoothly joined when $g>1$ we require 
\[
M_{1}\frac{\partial J_{-}}{\partial l}=\frac{\partial J_{+}}{\partial l}
\]
 and similarly for $g<1$ and $M_{2}$. The limits in (\ref{eq:K1 <-1}) force
$(\alpha_{1},\beta_{1})=(0,-2)$ and $(\alpha_{2},\beta_{2})=(-2,0)$. The corresponding
monodromy matrix is given by $M=(M_{2}S)^{-1}(M_{1}S)$ which we compute
to be (\ref{eq:Monodromy matrix}).
\end{proof}
The monodromy of the reduce system with actions $(J_1,\ell_{23})$ is obtained by the restriction to the top left block of $\mathfrak{M}$.
It should be stressed that the integrable system with Hamiltonian $H$ does not have monodromy, it is superintegrable, and does not even have dynamically defined tori. 
However,  the fibration defined by the three commuting functions $(H, \ell_{23}, G_{pro})$ 
on $T^*S^3$ has monodromy as computed. Similarly, the commuting function $(\ell_{23}, G_{pro})$ on $S^2 \times S^2$ have monodromy given by the top left block of $\mathfrak{M}$.

\subsection{Oblate Coordinates}

Eventhough the definition
of these coordinates is similar to the prolate case, the corresponding integrable
system is significantly different. In particular, even though it does have a global $S^1$ action it is not semi-toric because of the appearance of hyperbolic and degenerate singularities.

\subsubsection{Separation of Variables}

Oblate coordinates, denoted by $(1\ 2\ (3\ 4))$ and $((1\ 2)\ 3\ 4)$,
lie on opposite sides of the dotted line in Figure \ref{fig:Stasheff S3}. They are equivalent by flipping the ordering of the $e_i$ by applying $e_i\mapsto-e_i$ then reordering. 
Consequently, the corresponding integrable systems are equivalent. We focus on the $(1\ 2\ (3\ 4))$ coordinates
and normalise according to $(e_{1},e_{2},e_{3}=e_{4})=(0,1,a)$. This system is equivalent to the $((1\ 2)\ 3\ 4)$ with $(e_{1}=e_{2},e_{3},e_{4})=(0,1,\frac{a}{a-1})$. An
explicit definition of these coordinates is given by 
\begin{equation}
\begin{aligned}x_{1}^{2} & =\frac{s_{1}s_{2}}{a}, &  & x_{2}^{2}=\frac{-\left(s_{1}-1\right)\left(s_{2}-1\right)}{a-1},\\
x_{3}^{2} & =\frac{\left(s_{1}-a\right)\left(s_{2}-a\right)s_{3}}{a\left(a-1\right)}, &  & x_{4}^{2}=\frac{\left(s_{1}-a\right)\left(s_{2}-a\right)\left(1-s_{3}\right)}{a\left(a-1\right)},
\end{aligned}
\label{eq:oblate def}
\end{equation}
where $0\le s_{1},s_{3}\le1\le s_{2}\le a$. The integrals are $(2H,\ell_{34},G_{obl})$
where $G_{obl}=a\ell_{12}^{2}+\ell_{13}^{2}+\ell_{14}^{2}$. Let $(l,g)$ be functional values of $\ell_{34}$ and
$G_{obl}$ respectively, we obtain the separated
momenta
\begin{equation}
\begin{aligned}p_{i}^{2} & =\frac{-2hs_{i}^{2}+(2ah+g-(a-1)l^{2})s_{i}-ag}{4s_{i}(s_{i}-1)(s_{i}-a)^{2}} &  &  & p_{3}^{2} & =\frac{l^{2}}{4s_{3}(1-s_{3})}\end{aligned}
\label{eq:psq oblate}
\end{equation}
where $i\in\{1,2\}$. We call the triple $(2H,\ell_{34},G_{obl})$ on $T^{*}S^{3}$ the
1-parameter family of oblate integrable systems and $(\ell_{34},G_{obl})$ the corresponding reduced oblate integrable system. 

We have a similar result to  Lemma \ref{Prolate degen proof } for the oblate system.
\begin{lem}
    The integrals $(\ell_{34}^2,G_{obl})$
as well as the separated momenta \eqref{eq:psq oblate} can
be obtained by smoothly degenerating their ellipsoidal counterparts
\eqref{eq:Separation constants ellipsoidal} and \eqref{eq:psq ellipsoidal-1}.
\end{lem}
\begin{proof}
    Using the transformation 
    \[(e_{4},s_{3},p_3)=(e_{3}+\epsilon,e_{3}+\epsilon\tilde{s}_{3}, \frac{\tilde{p}_3}{\epsilon})\]
where $\tilde{s}_{3}\in[0,1]$ and following the same procedure as Lemma \ref{Prolate degen proof } gives the result.
\end{proof}
In the oblate limit, the Uhlenbeck integrals $\tilde{F}_i$ are
\[
\begin{aligned}\tilde{F}_{1}=-\frac{G_{obl}}{a}, &  & \tilde{F}_{2}=\frac{G_{obl}+\ell_{34}^2-2h}{a-1}, &  & \tilde{F}_{3}=\tilde{F}_4=\ell_{34}^2.\end{aligned}
\]

\subsubsection{Critical Points and Momentum Map}

\begin{figure}
\begin{centering}
\includegraphics[width=7cm,height=6cm]{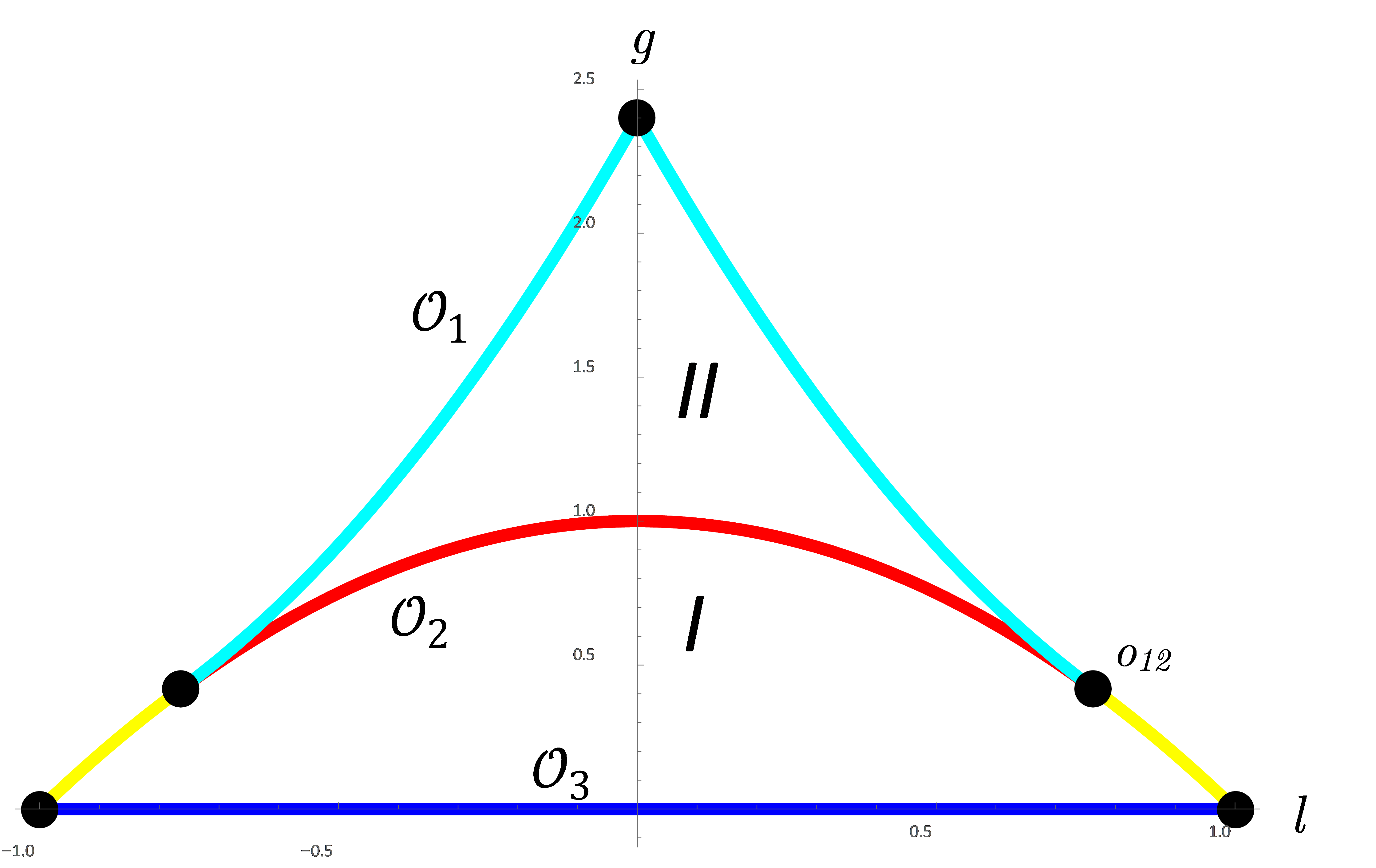}\quad\includegraphics[width=7cm,height=6cm]{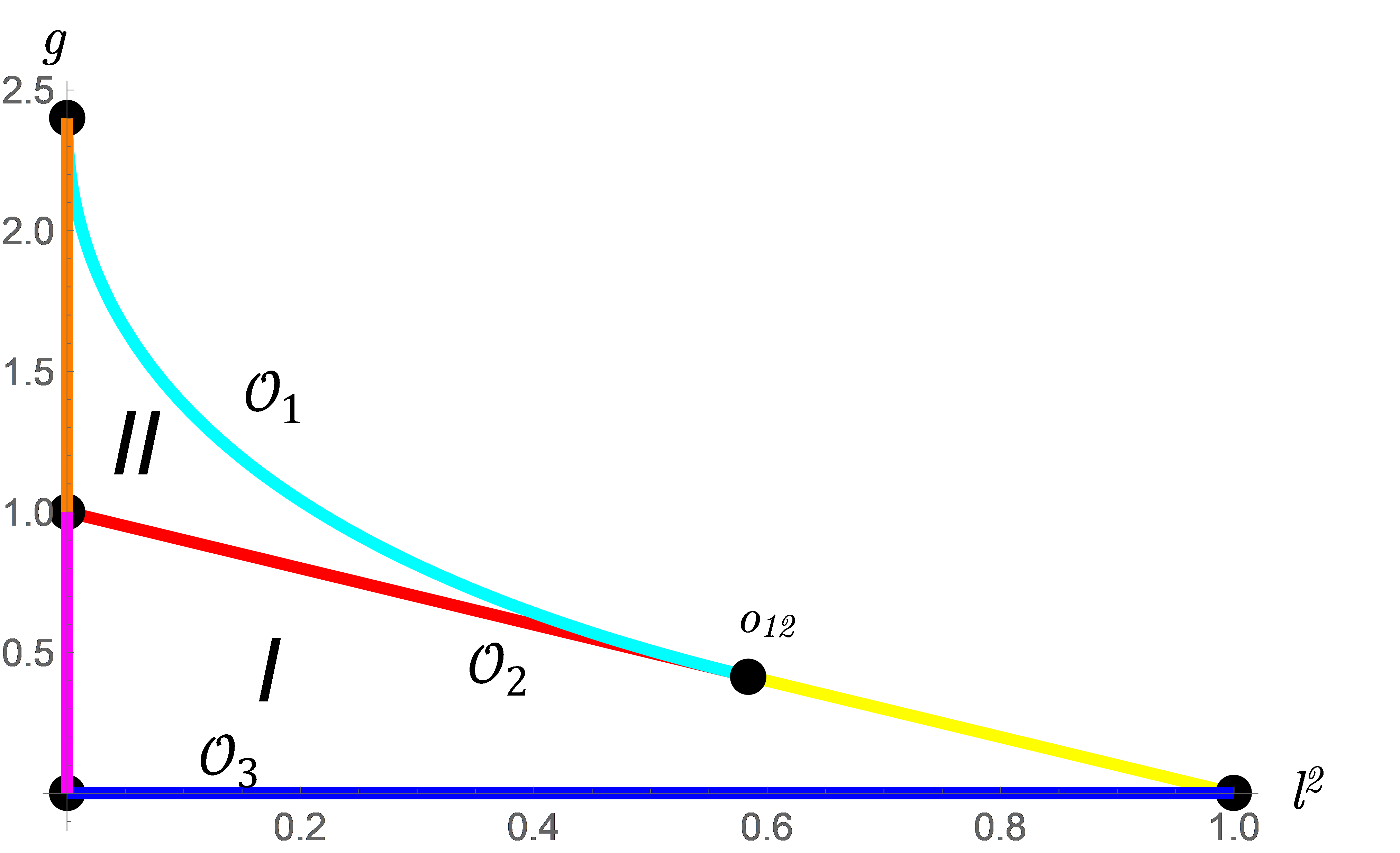}
\par\end{centering}
\caption{a) Momentum map for the oblate system $(\ell_{34},G_{obl})$ with $a=2.4$.
There are two chambers, labelled $I$ and $II$. b) ``Unopened''
momentum map taking $\ell_{34}^{2}$ as integral. \label{fig:Oblate root and MM}}
\end{figure}

We can also use the method of compatible Poisson structures to aid
in studying this system. Using the matrix $C=\text{\text{diag}}(0,1,a,a+\ensuremath{\epsilon})$
in Proposition \ref{fact1} gives the equation 
\[
\psi_{obl}(\lambda)=2\frac{-2h\lambda^{2}+I_{1}\lambda+I_{2}}{\lambda(\lambda-1)(\lambda-a)(\lambda-a-\epsilon)}
\]
where $(I_{1},I_{2})=(G_{obl}+2ah-(a-1)\ell_{34}^{2}+\epsilon(\ell_{12}^{2}+\ell_{13}^{2}+\ell_{23}^{2}),-aG_{obl}-a\epsilon\ell_{12}^{2}-\epsilon\ell_{13}^{2})$.
According to Corollary~\ref{biham-critical-values}, the critical values
occurs at the curve $I_{2}=\frac{I_{1}^{2}}{8h}$ as well as the lines
$I_{2}=0$, $I_{2}=2h-I_{1}$, $I_{2}=2ha^{2}-aI_{1}$ and $I_{2}=2h(a+\epsilon)^{2}-(a+\epsilon)I_{1}$
when $\lambda=0,1,a,a+\epsilon$ respectively. 
\begin{prop}
The critical values of the momentum map for the reduced oblate integrable system are the curves $\mathcal{O}_{1}:G_{obl}=(\sqrt{a}-\sqrt{a-1}\left|\ell_{34}\right|)^{2},\mathcal{O}_{2}:G_{obl}=1-\ell_{34}^{2}$
and $\mathcal{O}_{3}:G_{obl}=0$. The momentum map is shown in Figure
\ref{fig:Oblate root and MM} a) with $2h=1$.  
\end{prop}
\begin{proof} Applying Proposition~\ref{Bolsinov Theorem} and following a similar calculation as the ellipsoidal case we get:
\begin{enumerate}
\item The  blue line $\mathcal{O}_{3}:G_{obl}=0$ has $\lambda=0$ and critical
points parametrised by $\bm{L}=(0,0,0,\ell_{23},\ell_{24},\ell_{34})$
with $\ell_{23}^{2}+\ell_{24}^{2}+\ell_{34}^{3}=1$ and $\ell_{34}=l_{obl}$.
\item The red and yellow curve 
 $\mathcal{O}_{2}:G_{obl}=1-\ell_{34}^{2}$ has $\lambda=1$
and the critical points are parametrised by $\bm{L}=(0,\ell_{13},\ell_{14},0,0,\ell_{34})$
with $\ell_{13}^{2}+\ell_{14}^{2}+\ell_{34}^{3}=1$ and $\ell_{34}=l_{obl}$.

\item The cyan curve $\mathcal{O}_{1}:G_{obl}=(\sqrt{a}-\sqrt{a-1}\left|\ell_{34}\right|)^{2}$
has $\lambda\in[1,a]$ and critical points
\[
\begin{aligned}\bm{L}_{\pm}= & \left(\frac{a_{1}\sqrt{\lambda}\sqrt{\lambda-1}}{\sqrt{2}},\mp\frac{b_{2}\sqrt{\lambda}\sqrt{a-\lambda}}{\sqrt{2}},-\frac{b_{3}\sqrt{\lambda}\sqrt{a-\lambda}}{\sqrt{2}},\right.\\
 & \left.\mp\frac{b_{3}\sqrt{\lambda-1}\sqrt{a-\lambda}}{\sqrt{2}},\frac{b_{2}\sqrt{\lambda-1}\sqrt{a-\lambda}}{\sqrt{2}},\mp\frac{a_{1}(a-\lambda)}{\sqrt{2}}\right)
\end{aligned}
\]
with Pl\"{u}cker relation $a_{1}^{2}=b_{2}^{2}+b_{3}^{2}$ and $\bm{L}\cdot\bm{L}=\frac{1}{2}a(a-1)(b_{2}^{2}+b_{3}^{2})=1$.
This forces $a_{1}=\pm\frac{\sqrt{2}}{\sqrt{a(a-1)}}$ and gives a
parametrisation $(\ell_{34},G_{obl})=\left(\mp\frac{(a-\lambda)}{\sqrt{a(a-1)}},\frac{\lambda^{2}}{a}\right)$ in terms of $\lambda$. This curve exists only for $0<|\ell_{34}|<\sqrt{\frac{a-1}{a}}$
which corresponds to a double root $1<\lambda<a$.
\end{enumerate}

The lines $I_{2}=2ha^{2}-aI_{1}$ and $I_{2}=2h(a+\epsilon)^{2}-(a+\epsilon)I_{1}$
both gives the line $\ell_{34}=0$ (orange and magenta segments in Figure \ref{fig:Oblate root and MM} b)) which are degenerate critical values
for the system $(\ell_{34}^{2},G_{obl})$ since the vector field
generated by $\ell_{34}^{2}$ vanishes at $\ell_{34}=0$. However,
for the oblate integrable system $(\ell_{34},G_{obl})$ the line
$\ell_{34}=0$ becomes a set of regular values (except at three points). Direct computation
using the vector fields of $(\ell_{34},G_{obl})$ confirms these results.
\end{proof}

Let the intersection of $\mathcal{O}_{i}$ and $\mathcal{O}_{j}$
be denoted by $o_{ij\pm}$ where the sign is determined by whether
the intersection occurs for a positive or negative value of $\ell_{34}$.
The intersections $o_{12\pm}$ at $(\ell_{34},G_{obl})=(\pm\sqrt{\frac{a-1}{a}},\frac{1}{a})$
are tangential. The other $3$ intersections at $o_{11}=(0,a)$ and
$o_{23\pm}=(\pm1,0)$ are transverse.

The tangential intersections $o_{12\pm}$ are degenerate pitchfork singularities and their fibres are single circles $S^{1}$ on $S^2\times S^2$.
The points $o_{23\pm}$ are of elliptic-elliptic type. The point $o_{11}$ is also elliptic-elliptic with 2 critical points $\bm{L}=(\pm 1,0,0,0,0,0)$. 

The curves $\mathcal{O}_{1},\mathcal{O}_{3}$ as well as the yellow
parts of $\mathcal{O}_{2}$ are all codimension one elliptic. The
fibre of $\mathcal{O}_{3}$ and the yellow segments are single
$S^1$, while the fibre of $\mathcal{O}_{1}$ is $2S^1$. The
red part of $\mathcal{O}_{2}$ is codimension one hyperbolic and its
fibre is $B\times S^1$. 

The fibre of a regular value in chamber $I$ is $T^{2}$ while the
fibre of a regular value in chamber $II$ is $2T^{2}$.

\subsubsection{Actions} \label{obl-act}

\begin{figure}
\begin{centering}
\includegraphics[width=8cm,height=7cm]{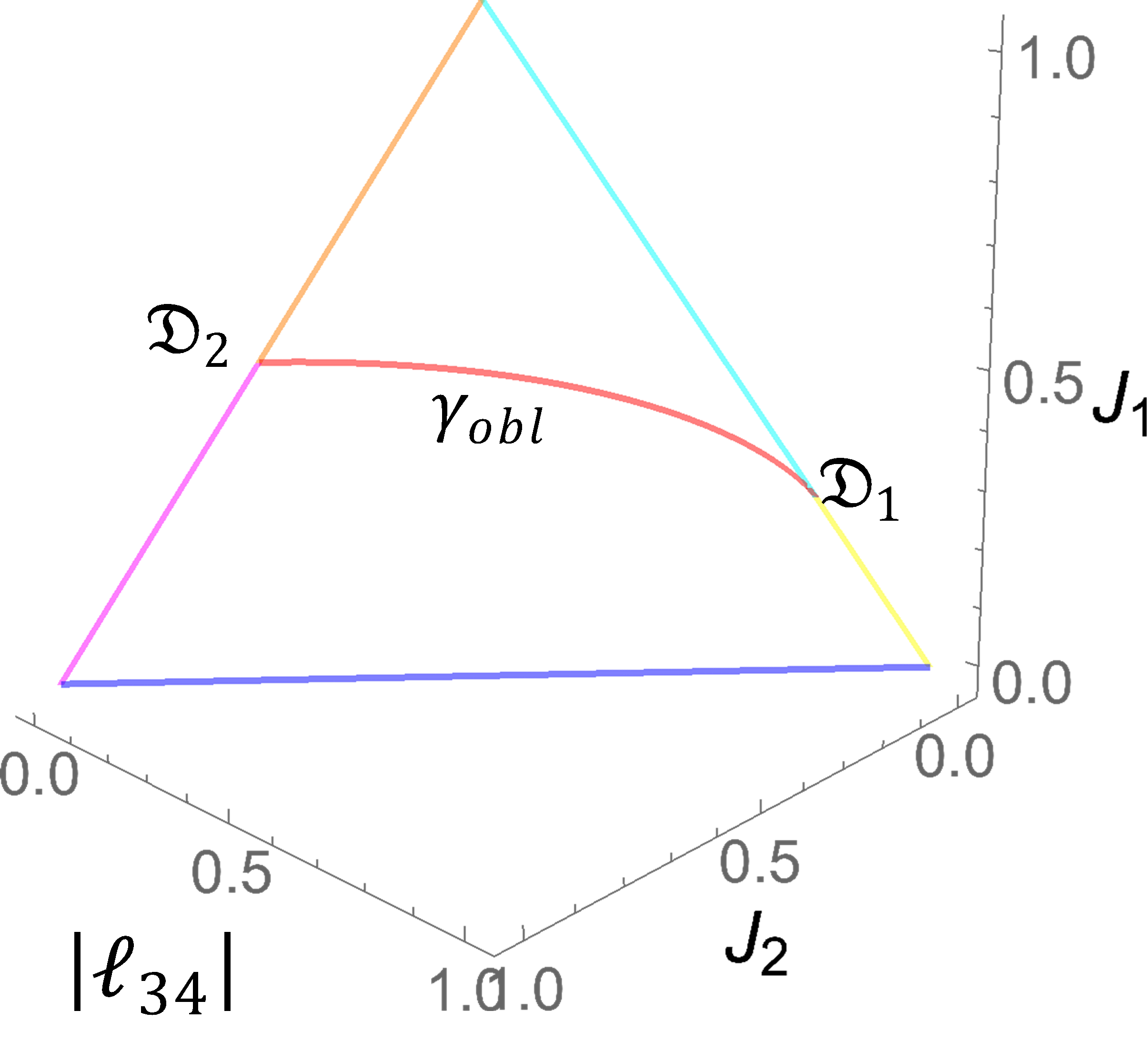}
\par\end{centering}
\caption{Action Map for the oblate system with $a=2.4$.\label{fig:Action Map Oblate}}
\end{figure}

Like in the prolate case, one action for the oblate system $J_{3,obl}\coloneqq\left|\ell_{34}\right|$
is trivial. The other two non trivial actions are
\[
\begin{aligned}J_{1}=\frac{2}{\pi}\int_{0}^{\min(r_{1},1)}p_{1}ds, &  &  & J_{2}=\frac{2}{\pi}\int_{\max(r_{1},1)}^{\min(r_{2},a)}p_{2}ds\end{aligned}
\]
where $r_{2}\ge1$ and $0\le r_{1}\le r_{2}\le a$. Theorem~\ref{action thm} also applies here. The action map is shown in Figure \ref{fig:Action Map Oblate}. For the interior (red)
curve $\gamma_{obl}$ we have $r_1=1$, $r_2=a(1-l^2)\geq1$ with $g=1-l^2$ for $|l|\leq\sqrt{\frac{a-1}{a}}$. The parametrisation of this curve in terms of the angular momentum $l$ where $|l|\leq\sqrt{\frac{a-1}{a}}$ is given by
\[
\gamma_{obl}(l)=\left(\frac{2}{\pi}\left(\sin ^{-1}\left(t_1\right)+l \tan ^{-1}\left(t_2\right)\right)-|l|,1-\frac{2}{\pi}\left(\sin ^{-1}\left(t_1\right)+l \tan ^{-1}\left(t_2\right)\right),|l|\right)
\]
 where $t_1=\frac{1}{\sqrt{a(1-l^2)}}$ and $t_2=\frac{\sqrt{a(1-l^2)-1}}{l}$.
Call the intersections of $\gamma_{obl}$ with the boundary of the
action map $\mathfrak{O}_{1}$ (yellow/cyan) and $\mathfrak{O}_{2}$
(magenta/orange). The point $\mathfrak{O}_{1}$ has coordinates $\gamma_{obl}(\sqrt{\frac{a-1}{a}})=(1-\sqrt{\frac{a-1}{a}},0,\sqrt{\frac{a-1}{a}})$,
while $\mathfrak{O}_{2}$ is located at $\gamma_{obl}(0)=\frac{2}{\pi}(\sin^{-1}\left(\frac{1}{\sqrt{a}}\right),\cos^{-1}\left(\frac{1}{\sqrt{a}}\right),0)$.

\subsection{Lam\'{e} Coordinates}

The Lam\'e system is unusual in a number of ways. In this case we actually need to make use of the fact that the parameters $e_i$ live on the real projective line. This is the reason why this case is not visible in the original normalised $ab$-parameter space and a blow-up is required. Furthermore, it is a family that has larger symmetry group $SO(3)$. A larger symmetry group is in some sense related to super-integrability, however, since we don't have a Hamiltonian in the reduced system it is harder to define what this means. As we will see after another reduction, this system becomes the Euler top (see the Appendix).

\subsubsection{Separation of Variables and St\"{a}ckel System}

Lam\'{e} coordinates are an extension of ellipsoidal coordinates from $S^{2}$
onto $S^{3}$ and arise from limiting to three equal semi-major axes. As with oblate, there are two cases to consider: $(1\ (2\ 3\ 4))$
and $((1\ 2\ 3)\ 4)$ which are equivalent. Here we only discuss the $(1\ (2\ 3\ 4))$
coordinates which are defined as follows 
\begin{equation}
\begin{aligned}x_{1}^{2} & =s_{1}, & x_{3}^{2} & =\frac{\left(s_{1}-1\right)\left(f_{2}-s_{2}\right)\left(f_{2}-s_{3}\right)}{\left(f_{1}-f_{2}\right)\left(f_{2}-f_{3}\right)},\\
x_{2}^{2} & =-\frac{\left(s_{1}-1\right)\left(s_{2}-f_{1}\right)\left(s_{3}-f_{1}\right)}{\left(f_{2}-f_{1}\right)\left(f_{3}-f_{1}\right)} & x_{4}^{2}, & =\frac{\left(s_{1}-1\right)\left(f_{3}-s_{2}\right)\left(f_{3}-s_{3}\right)}{\left(f_{2}-f_{3}\right)\left(f_{3}-f_{1}\right)},
\end{aligned}
\label{eq:Lame coord def}
\end{equation}
where $0\le s_{1}\le1$ and $0\le f_{1}\le s_{2}\le f_{2}\le s_{3}\le f_{3}$.
A possible St\"{a}ckel matrix for these coordinates is 
\begin{equation}
\Phi_{L}=\frac{1}{4}\left(\begin{array}{ccc}
-\frac{1}{\left(s_{1}-1\right)s_{1}} & -\frac{1}{\left(s_{1}-1\right){}^{2}s_{1}} & 0\\
0 & \frac{1}{\left(f_{3}-s_{2}\right)\left(s_{2}-f_{2}\right)} & \frac{1}{\left(f_{3}-s_{2}\right)\left(s_{2}-f_{1}\right)\left(s_{2}-f_{2}\right)}\\
0 & \frac{1}{\left(f_{3}-s_{3}\right)\left(s_{3}-f_{2}\right)} & \frac{1}{\left(f_{3}-s_{3}\right)\left(s_{3}-f_{1}\right)\left(s_{3}-f_{2}\right)},
\end{array}\right).\label{eq:Lame stack}
\end{equation}
with integrals $(2H,2H-F_{L},G_{L}-f_{1}(2H-F_{L}))$
where $(F_{L},G_{L})=(\ell_{12}^{2}+\ell_{13}^{2}+\ell_{14}^{2},f_{1}l_{34}^{2}+f_{2}l_{24}^{2}+f_{3}l_{23}^{2})$. From (\ref{eq:Lame stack}), the separated momenta are given by 
\begin{equation}
\begin{aligned}p_{1}^{2} & =\frac{f_{L}-2hs_{1}}{4\left(s_{1}-1\right){}^{2}s_{1}}, &  &  & p_{k}^{2} & =-\frac{(f_{L}-2h)s_{k}+g_{L}}{4\left(f_{3}-s_{k}\right)\left(s_{k}-f_{1}\right)\left(s_{k}-f_{2}\right)},\end{aligned}
\label{eq:lame sep momenta}
\end{equation}
where $k=2,3$ and $(f_{L},g_{L})$ are functional values of $(F_{L},G_{L})$. We call the triple $(2H,F_{L},G_{L})$ on $T^{*}S^{3}$ the
Lam\'{e} integrable system and $(F_L,G_L)$ the corresponding reduced Lam\'{e} integrable system on $S^2\times S^2$.

The important feature of this case is the appearance of the integral $F_L$ with a $SO(3)$ symmetry given by the rotations generated by $\ell_{1i}$ using $B_L$. Similar to the prolate and oblate families, the Lam\'{e} system can also be obtained as a limit of the ellipsoidal system.
\begin{lem}
The integrals $(F_{L},G_{L})$ and separated momenta (\ref{eq:lame sep momenta})
for the Lam\'{e} integrable system can be obtained from their ellipsoidal counterparts \eqref{eq:Separation constants ellipsoidal} and \eqref{eq:psq ellipsoidal-1}.
\end{lem}
\begin{proof}
A possible limiting process from ellipsoidal to Lam\'{e} coordinates is given in \cite{KKM18}. However in this case, it is much simpler to use the transformation
\begin{equation}
\begin{aligned}\left(e_{1},e_2,e_3,e_{4}\right) & =\left(-\frac{1}{\epsilon},f_1,f_2,f_3\right)\\
\left(s_{1},p_1\right) & =\left(f_2-\frac{\tilde{s}_1}{\epsilon},\epsilon \tilde{p}_1\right)
\end{aligned}
\label{eq:sub lame}
\end{equation}
for $\epsilon>0$ and $\tilde{s}_1\in [0,1]$. Applying \eqref{eq:sub lame} and taking the limit as $\epsilon\to 0$ immediately gives \[(\tilde{\eta}_1,\tilde{\eta}_2)=(-\frac{2H-F_L}{\epsilon},\frac{-G_L}{\epsilon})\]

The separated momenta are obtained using the same method as in the prolate case. 
\end{proof}

In the Lam\'{e} limit, the Uhlenbeck integrals $\tilde{F}_{i}$ are as
follows
\[
\begin{aligned}\tilde{F}_{1}=-F_{L}, &  & \tilde{F}_{2} & =\frac{\ell_{23}^{2}}{f_{1}-f_{2}}+\frac{\ell_{24}^{2}}{f_{1}-f_{3}}, &  & \tilde{F}_{3}=\frac{\ell_{23}^{2}}{f_{2}-f_{1}}+\frac{\ell_{34}^{2}}{f_{2}-f_{3}}, &  & \tilde{F}_{4}=\frac{\ell_{24}^{2}}{f_{3}-f_{1}}+\frac{\ell_{34}^{2}}{f_{3}-f_{2}}.\end{aligned}
\]
Note that up to projective transformations, the system only has one parameter $\frac{f_3-f_1}{f_2-f_1}$, but we use $f_i$ to keep the higher symmetry.

The vector field of $G_L$ given by $B_{\bm{L}} \nabla G_L$ has a semi-direct product stucture: the equations for $\ell_{23}, \ell_{24}, \ell_{34}$ decouple from the others, and they are in fact Euler's equations for the rigid body on $SO(3)$ with moments of inertia given by $f_i^{-1}$. The equation for the remaining three variables are linear equations with time-varying coefficients given by the solution of Euler's equation.

\subsubsection{Critical Points and Momentum Map}
Using the matrix $C=\text{diag}(-\frac{1}{\epsilon},f_1,f_2,f_3)$
gives the integrals 
\[
(I_{1},I_{2})=-\frac{1}{\epsilon}\left(2h-F_L+O(\epsilon),G_L+O(\epsilon)\right).
\]
We need consider the $\frac{1}{\epsilon}$
order term since the limit as $\epsilon\to0$ of $(I_{1},I_{2})$
is infinite.
\begin{prop}
The set of critical values of the momentum map for the Lam\'{e} system is composed of four straight lines $\mathfrak{L}_{j}:F_{L}=1-\frac{1}{f_{j}}G_{L}$
for $j=1,2,3$ and $\mathfrak{L}_{4}:F_{L}=0$. This is shown in Figure \ref{fig:Lame MM and Action Map} a).
\end{prop}
\begin{figure}
\begin{centering}
\includegraphics[width=8cm,height=7cm]{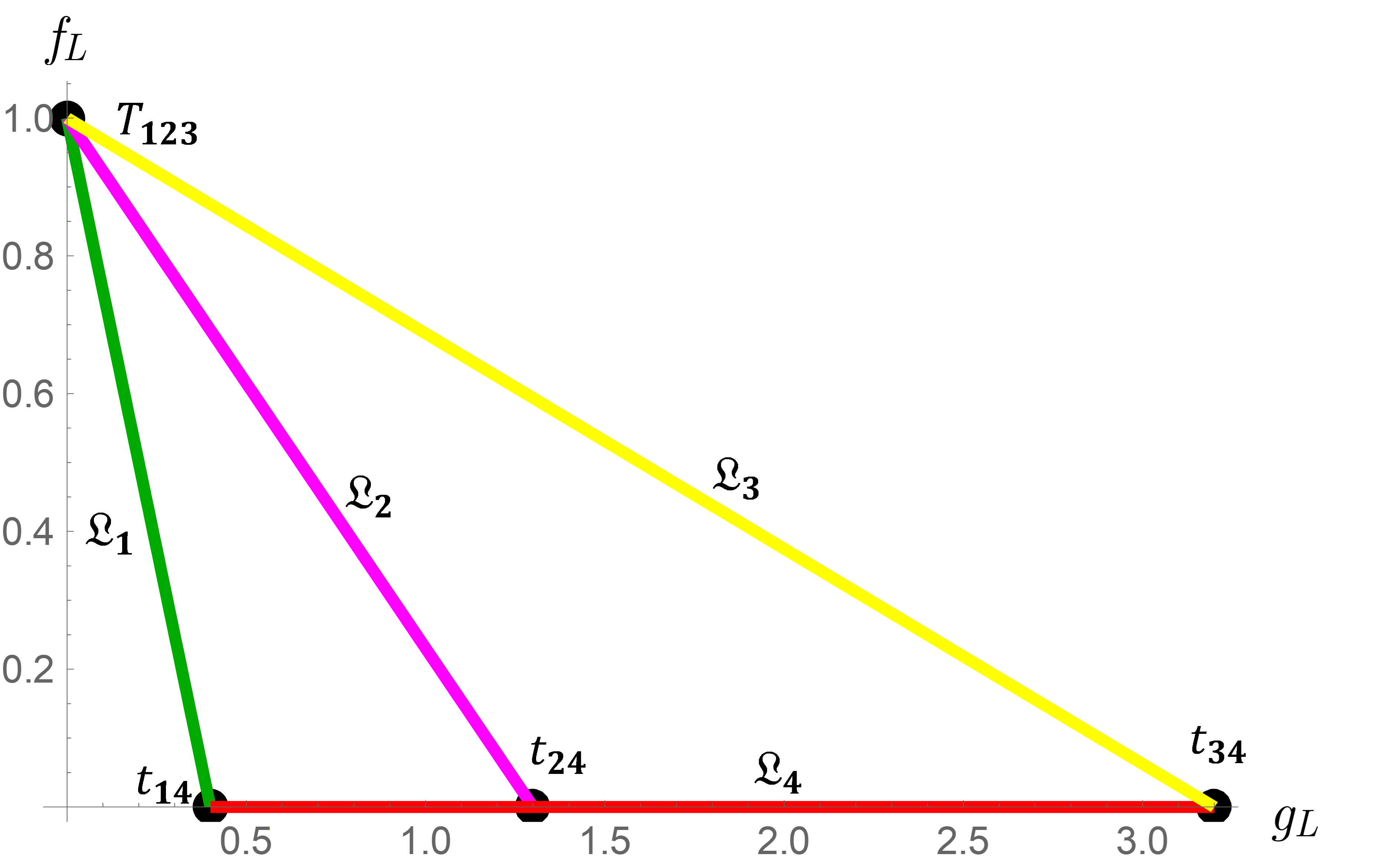}\includegraphics[width=7cm,height=7cm]{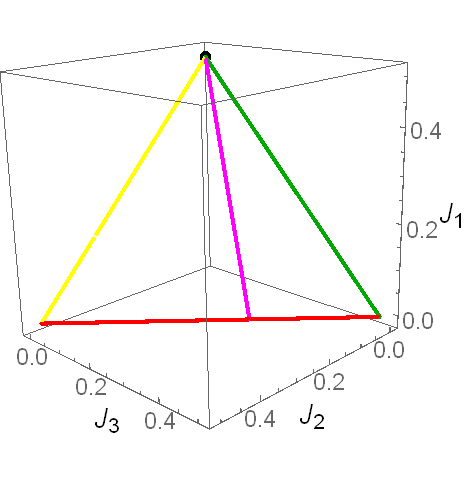}
\par\end{centering}
\caption{a) Momentum map for the Lam\'{e} system with $(f_{1},f_{2},f_{3})=(0.4,1.3,3.2)$.
b) Corresponding action map. \label{fig:Lame MM and Action Map}}
\end{figure}

\begin{proof}The critical points for the Lam\'{e} integrable systems are similar to
those for the lines in the ellipsoidal system:
\begin{enumerate}
\item The line $\mathfrak{L}_{4}:F_{L}=0$ is the limit as $\epsilon\to 0$ of the line $I_2=-\frac{1}{\epsilon}(I_1+\frac{1}{\epsilon})$ with $\lambda=-\frac{1}{\epsilon}$ and the critical
points are parametrised by $\bm{L}=(0,0,0,\ell_{23},\ell_{24},\ell_{34})$
with $\ell_{23}^{2}+\ell_{24}^{2}+\ell_{34}^{2}=1$ and $f_{1}\ell_{34}^{2}+f_{2}\ell_{24}^{2}+f_{3}\ell_{23}^{2}=g_{L}$.

\item The line $\mathfrak{L}_{1}:F_{L}=1-\frac{1}{f_{1}}G_{L}$ is the limit as $\epsilon\to 0$ of the line $I_2=f_1(I_1-f_1)$ with
$\lambda=f_1$ and the critical points are parametrised by $\bm{L}=(0,\ell_{13},\ell_{14},0,0,\ell_{34})$
with $\ell_{13}^{2}+\ell_{14}^{2}+\ell_{34}^{2}=1$ and $f_{1}\ell_{34}^{2}=g_{L}$.

\item The line $\mathfrak{L}_{2}:F_{L}=1-\frac{1}{f_{2}}G_{L}$ is the limit as $\epsilon\to 0$ of the line $I_2=f_2(I_1-f_2)$ with
$\lambda=f_2$ and the critical points are parametrised
by $\bm{L}=(\ell_{12},0,\ell_{14},0,\ell_{24},0)$ with $\ell_{12}^{2}+\ell_{14}^{2}+\ell_{24}^{2}=1$
and $f_{2}\ell_{24}^{2}=g_{L}$. 

\item The line $\mathfrak{L}_{3}:F_{L}=1-\frac{1}{f_{3}}G_{L}$ is the limit as $\epsilon\to 0$ of the line $I_2=f_3(I_1-f_3)$ with 
$\lambda=f_3$ and the critical points are parametrised
by $\bm{L}=(\ell_{12},\ell_{13},0,\ell_{23},0,0)$ with $\ell_{12}^{2}+\ell_{13}^{2}+\ell_{23}^{2}=1$
and $f_{3}\ell_{23}^{2}=g_{L}$. 
\end{enumerate}

The curve $I_{2}=\frac{I_{1}^{2}}{8h}$ for $\lambda\in[f_1,f_2]$
shrinks to the degenerate point $(F_{L},G_{L})=(1,0)$ in the limit $\epsilon\to 0$.
\end{proof}

Let the intersections of $\mathfrak{L}_{i}$ and $\mathfrak{L}_{j}$
be denoted by $t_{ij}$ and the three way intersection of $\mathfrak{L}_{1}$, $\mathfrak{L}_{2}$, $\mathfrak{L}_{3}$ by $T_{123}$.
The intersections $t_{14}$ and $t_{34}$ are elliptic-elliptic critical
values with 2 points in their fibres $2S^{1}$ on $S^2\times S^2$. The point $t_{24}$
is elliptic hyperbolic and its fibre is $C_{2}$. 

\begin{lem}
   The three-way intersection $T_{123}$ at $(F_{L},G_{L})=(1,0)$ is a degenerate singularity of spherical type. %  as defined in \cite{RonanThesis}.
\end{lem}
\begin{proof}
  The linearisation of the vector field generated by $F_L$ and $G_L$ is $\nabla_{\bm L}B_{\bm L}(\alpha \nabla_{\bm L}F_L + \beta \nabla_{\bm L} G_L)$. At the three-way intersection $T_{123}$, this matrix becomes 
  \begin{equation}
      \left(\begin{array}{cc}
\bm 0 & \begin{array}{ccc}
2(\alpha-f_3\beta)\ell_{13} & 2(\alpha-f_2\beta)\ell_{14} & 0\\
-2(\alpha-f_3\beta)\ell_{12} & 0 & 2(\alpha-f_1\beta)\ell_{14}\\
0 & -2(\alpha-f_2\beta)\ell_{12} & -2(\alpha-f_1\beta)\ell_{13}
\end{array}\\
\bm 0 & \bm 0
\end{array}\right)\label{linlame}
  \end{equation}
  where $\bm 0$ is a $3\times 3$ matrix of all zeros.
 The eigenvalue of \eqref{linlame} is $(0,0,0,0,0,0)$ meaning that $T_{123}$ is a degenerate critical value. The rank of the differential of the moment map drops by $1$ at $T_{123}$.  This is known as a spherical type singularity studied in the thesis \cite{RonanThesis}. Systems with a spherical type singularity are characterised by the presence of a globally defined, continuous but not smooth action. The fibre of a spherical type singularity is diffeomorphic to a product of spheres. In this case the preimage of $T_{123}$ is $S^{2}$ and every point in this fibre is critical.  
\end{proof}
 It was shown in \cite{RonanThesis} that the geodesic flow on $S^n$ in polyspherical coordinates gives rise to systems containing a spherical type singularity. While the Lam\'e coordinate system is not polyspherical, the reduced Lam\'e integrable system is an example of a system that has a spherical singularity that was not studied in \cite{RonanThesis}. Since the Lam\'e coordinate system is obtained by extending the ellipsoidal coordinates on $S^2$ to $S^3$, we can conjecture that systems originating from the geodesic flow on $S^n$ in coordinates obtained from extending a coordinate system from $S^k$ to $S^n$ where $2\leq k<n$ will contain a spherical type singularity.

\subsubsection{Actions}\label{lame-act}

The actions of the Lam\'{e} system are given by 
\[
\begin{aligned}J_{1} & =\frac{2}{\pi}\int_{0}^{f_L}p_{1}ds &  &  & J_{2} & =\frac{2}{\pi}\int_{f_{1}}^{\min(r_{2},f_{2})}p_{2}ds &  &  & J_{3} & =\frac{2}{\pi}\int_{\max(f_{2},r_{2})}^{f_{3}}p_{3}ds\end{aligned}
\]
 where  $r_{2}=\frac{g_{L}}{1-f_{L}}$. Theorem~\ref{action thm} applies for the Lam\'e system also. The
first action evaluates to 
\begin{equation}
J_{1}=\text{\ensuremath{1-\sqrt{1-f_{L}}.}}\label{eq:J1 Lame}
\end{equation}

Notice that $F_L$ has a vector field $B_{\bm L} \nabla F_L$ has a flow that is the rotation of $(\ell_{12}, \ell_{13}, \ell_{14})$ about the fixed axis given by $(\ell_{34}, -\ell_{24}, \ell_{23})$. The frequency of this rotation is given by the length of the axis, and is hence not constant. 

\begin{lem}
    $J_{1} = 1 - \sqrt{ 1 - F_L}$ is an almost global $S^1$-action.
\end{lem}
\begin{proof}
    The vector field $B_{\bm L}\nabla J_{1}$ has periodic flow which is given by the rotation about the same axis as the flow of $F_L$, but here the axis is normalised because we need to divide by $\sqrt{ 1 - F_L}$ and using the Casimir $2h=1$ this means to divide by the length of the axis. It is only ``almost'' global because when $2h - F_L = \ell_{34}^2 + \ell_{24}^2 + \ell_{23}^2 = 0$ the normalisation factor vanishes and the vector field is not defined. Because of $2h=1$ this occurs on the sphere $S^2$ given by $\ell_{12}^2 + \ell_{13}^2 + \ell_{14}^2 = 1$.
\end{proof}
Note that \eqref{eq:J1 Lame} means that lines of constant $f_{L}$ correspond to
lines of constant $J_{1}$ in action space. In general, we have 
\begin{lem}
    Straight lines in the image of the momentum map $(F_L,G_L)$ of the reduced Lam\'e system on $S^2\times S^2$ given by $F_L=1-\frac{1}{r_2}G_L$ maps to straight lines in action space.
\end{lem}
\begin{proof}

Observe that
    \[
    J_{2}=\sqrt{1-f_L}\int_{f_1}^\gamma\sqrt{\frac{r_2-s}{(s-f_1)(s-f_2)(s-f_3)}}ds=(1-J_{1,L})\mathcal{F}(r_2)
    \]
    where $\gamma=\min(r_{2},f_{2})$,
    $r_2=\frac{g_L}{1-f_L}$ and $\mathcal{F}$ is function of $r_2$ only. Since $J_{3,L}=1-J_{1,L}-J_{2}$, this implies that the image of a straight line $F_L=1-\frac{1}{r_2}G_L$ with constant slope $\frac{1}{r_2}$ under the action map is again a straight line. 
%    
% This can also be observed by expressing $J_{2}$ in terms of elliptic integrals as 
%\begin{equation}
%\pi J_{2}=\sqrt{1-f_{L}}(r_{2}-f_{1})\frac{\mathcal{G}}{\alpha^{2}}\left[F(\varphi,k)+(\alpha^{2}-1)\Pi(\varphi,\alpha^{2},k)\right]\label{eq:J2 lame}
%\end{equation}
%where $\mathcal{G}\coloneqq\frac{2}{\sqrt{\left(f_{2}-f_{1}\right)\left(f_{3}-r_{2}\right)}},k^{2}\coloneqq\frac{\left(f_{3}-f_{2}\right)\left(r_{2}-f_{1}\right)}{\left(f_{2}-f_{1}\right)\left(f_{3}-r_{2}\right)},\alpha^{2}\coloneqq\frac{f_{1}-r_{2}}{f_{3}-r_{2}},\varphi\coloneqq\text{sin}^{-1}\left[\sqrt{\frac{\left(\text{\ensuremath{\gamma}}-f_{1}\right)\left(f_{3}-r_{2}\right)}{\left(f_{3}-\text{\ensuremath{\gamma}}\right)\left(r_{2}-f_{1}\right)}}\right]$
%where $\gamma=\min(r_{2},f_{2})$ and $F(\varphi,k),\Pi(\varphi,\alpha^{2},k)$ are the elliptic integrals of the first and third kind respectively. 
\end{proof}

Figure \ref{fig:Lame MM and Action Map} b) shows an example of
the action map. Let the magenta line in the interior of the action
map be denoted by $\mathfrak{L}_{M}$. This has parameterisation 

\[
(J_{1},J_{2},J_{3})=(J_{1},\frac{2}{\pi}(1-J_{1})\sin^{-1}\Delta,\frac{2}{\pi}(1-J_{1})\cos^{-1}\Delta)
\]
 where $\Delta=\sqrt{\frac{f_{1}-f_{2}}{f_{1}-f_{3}}}$.
The line intersects the boundary of the action map at $(1,0,0)$ and
$\frac{2}{\pi}(0,\sin^{-1}(\Delta),\cos^{-1}(\Delta)$.

The fact that lines of constant $f_L$ and lines through $T_{123}$ are mapped to straight lines does not imply it is a linear map, because the map along these lines is determined by the non-linear map $\mathcal{F}$.

Since away from $f_l = 1$, the action variable $J_{1}$ is defined we can consider reduction with respect to the flow of $J_{1}$ on levels with $0 < f_L < 1$. Fixing the action and identifying the corresponding angle variable to a point gives the action $J_{2}$ for that constant value of $f_L$, and up to an overall constant factor this is the action of the Euler top. 

\subsection{Spherical Coordinates}

Spherical coordinates (or rather poly-spherical coordinates) correspond to the case where simultaneously there is an $SO(2)$ and an $SO(3)$ symmetry. Accordingly we do have a global $S^1$ action. However, the induced integrable system on $S^2 \times S^2$ is not semi-toric because it has a degenerate point, which corresponds to the critical values at which the almost global action is not differentiable.

\subsubsection{Separation of Variables}

The two forms of spherical coordinates are found by setting $f_{1}=f_{2}$
or $f_{2}=f_{3}$ in Lam\'{e} coordinates. We call these the $12$ and
$23$ spherical coordinates. The two systems are equivalent by a permutation of coordinates. These can also be obtained by setting
$a=1$ in prolate and oblate coordinates, respectively. The $23-$spherical
coordinate system $(1\ (2\ (3\ 4)))$ is defined by 
\begin{equation}
\begin{aligned}x_{1}^{2} & =s_{1}, &  & x_{2}^{2}=\left(1-s_{1}\right)s_{2},\\
x_{3}^{2} & =\left(1-s_{1}\right)\left(1-s_{2}\right)s_{3}, &  & x_{4}^{2}=\left(1-s_{1}\right)\left(1-s_{2}\right)\left(1-s_{3}\right),
\end{aligned}
\label{eq:RR sph}
\end{equation}
where $0\le s_{k}\le1$ and $k=1,2,3$. \textcolor{black}{Due to
the simplicity of these coordinates, we can manually separate the
corresponding Hamilton-Jacobi equation. The geodesic Hamiltonian can
be expressed as 
\begin{equation}
H_{23}=\frac{2\left(p_{2}^{2}\left(s_{2}-1\right)s_{2}-p_{1}^{2}\left(s_{1}-1\right){}^{2}s_{1}-\frac{p_{3}^{2}\left(s_{3}-1\right)s_{3}}{s_{2}-1}\right)}{s_{1}-1}.\label{eq:H23}
\end{equation}
The integrals are $(2H_{23},\ell_{34},G_{23})$ with separated momenta 
\begin{equation}
\begin{aligned}p_{1}^{2} & =\frac{g_{23}-2hs_{1}}{4s_{1}\left(s_{1}-1\right)^{2}} &  &  & p_{2}^{2} & =\frac{(g_{23}-2h)(s_{2}-1)-l_{34}^{2}}{4s_{2}\left(s_{2}-1\right)^{2}} &  &  & p_{3}^{2} & =\frac{l_{34}^{2}}{4s_{3}\left(1-s_{3}\right)}\end{aligned}
\label{eq:Sep 23 system}
\end{equation}
}where $G_{23}=\ell_{12}^{2}+\ell_{13}^{2}+\ell_{14}^{2}$ and $(l_{34},g_{23})$ are functional values of $(\ell_{34}^2,G_{23})$.
% We call $(2H,\ell_{34},G_{23})$ on $T^{*}S^{3}$ the $23-$spherical system and $(\ell_{34},G_{23})$ the corresponding reduced $23-$spherical system on $S^2\times S^2$.

To obtain $(\ell_{34},G_{23})$ and (\ref{eq:Sep 23 system}) from
the Lam\'{e} system, we set $(f_{3},s_{3},p_3)=(f_{2}+\epsilon,f_{2}+\epsilon\tilde{s}_{3},\tilde{p}_3/\epsilon)$
where $\tilde{s}_{3}\in[0,1]$ and normalise $(f_{1},f_{2})=(0,1)$.
To come from oblate, we let $(a,s_{3},p_3)=(1+\epsilon,1+\epsilon\tilde{s}_{3},\tilde{p}_3/\epsilon)$.
\subsubsection{Critical Points and Momentum Map}
The critical points and values are easily obtained by direct computation to give
\begin{prop}
\label{Theorem SPh}The image of momentum map for the $23-$spherical system
$(\ell_{34},G_{23})$ with $2h=1$ has critical values $\mathfrak{C}_{1}:G_{23}=1-\ell_{34}^{2}$
and $\mathfrak{C}_{2}:G_{23}=0$ which are both codimension
one elliptic (see Figure \ref{fig: RR sph }). 
\end{prop}
\begin{proof} The computation of critical points and values are straight forward for this system.
\end{proof}
\textcolor{black}{The fibre of a regular value on $S^2\times S^2$ is a
torus $T^{2}$ with multiplicity one.} The fibres along $\mathfrak{C}_{1}$ and $\mathfrak{C}_{2}$
are single $S^1$. The intersections of $\mathfrak{C}_{1}$
and $\mathfrak{C}_{2}$ are codimension 2 elliptic points and have 1 critical point in their fibres. The linearisation $\nabla_{\bm L}B_{\bm L}(\alpha \nabla_{\bm L}G_{23} + \beta \nabla_{\bm L} \ell_{34})$ has eigenvalues $(0,0,-i \beta ,i \beta ,-i\beta ,i\beta)$ at $D_{23}=(0,1)$ making the peak of the parabola $D_{23}$ a degenerate singularity. Similar to the Lam\'e system the rank of the differential of the moment map drops by 1 at $D_{23}$ and it's fibre is $S^2$. This is also an example of a spherical type singularity. Spherical coordinates is a type of polyspherical coordinates and these have been studied in detail in \cite{RonanThesis}.  

In the limit $a\to1$, the bifurcation diagram for the oblate coordinates in 
Figure~\ref{fig:Oblate root and MM}~a) degenerates to Figure~\ref{fig: RR sph }~a). In particular, the elliptic-elliptic point
$o_{11}$ collides with the hyperbolic line $\mathcal{O}_{2}$ and
becomes degenerate. Similarly, setting $f_{2}=f_{3}$ in the Lam\'{e}
system causes the elliptic-hyperbolic point $t_{24}$ to collide with
the elliptic-elliptic point at $t_{34}$ while $T_{123}$ remains
degenerate.

\subsubsection{Actions}\label{sph-act}

The action variables for the $23-$spherical system \textcolor{black}{are
given by 
\[
\begin{aligned} {J}_{1} & =\frac{2}{\pi}\int_{0}^{r_{1}}p_1(s)ds &  &  & {J}_{2} & =\frac{2}{\pi}\int_{0}^{r_{2}}p_2(s)ds &  &  & {J}_{3} & =|\ell_{34}|\end{aligned}
\]
where $0\le r_{1}=g_{23}\le1$ and $0\le r_{2}=1-\frac{l_{23}^{2}}{1-g_{23}}\le1$.
We can simplify the non trivial actions to 
\[
\begin{aligned} {J}_{1}=1-\sqrt{1-g_{23}} &  & {J}_{2} & =\sqrt{1-g_{23}}-|l_{34}|.\end{aligned}
\]
}

The action map is shown in Figure \ref{fig: RR sph } b). Note that $(\ell_{34},\sqrt{1-G_{23}})$ defines continuous global action variables that are not differentiable at $G_{23}=1$. This is a system obtained from toric degeneration, see \cite{RonanThesis}.

\begin{figure}
\begin{centering}
\includegraphics[width=8cm,height=7cm]{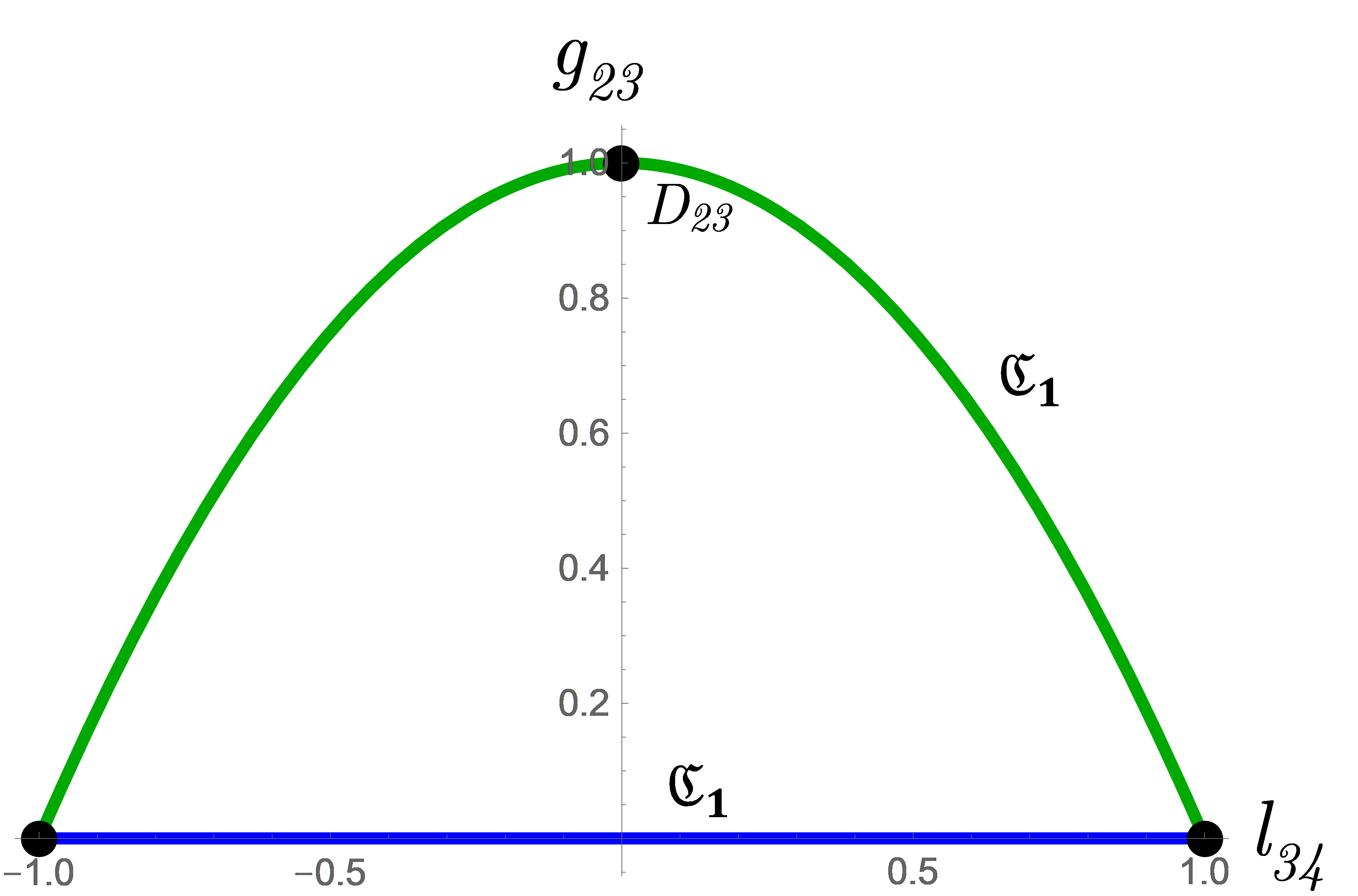}\includegraphics[width=7cm,height=7cm]{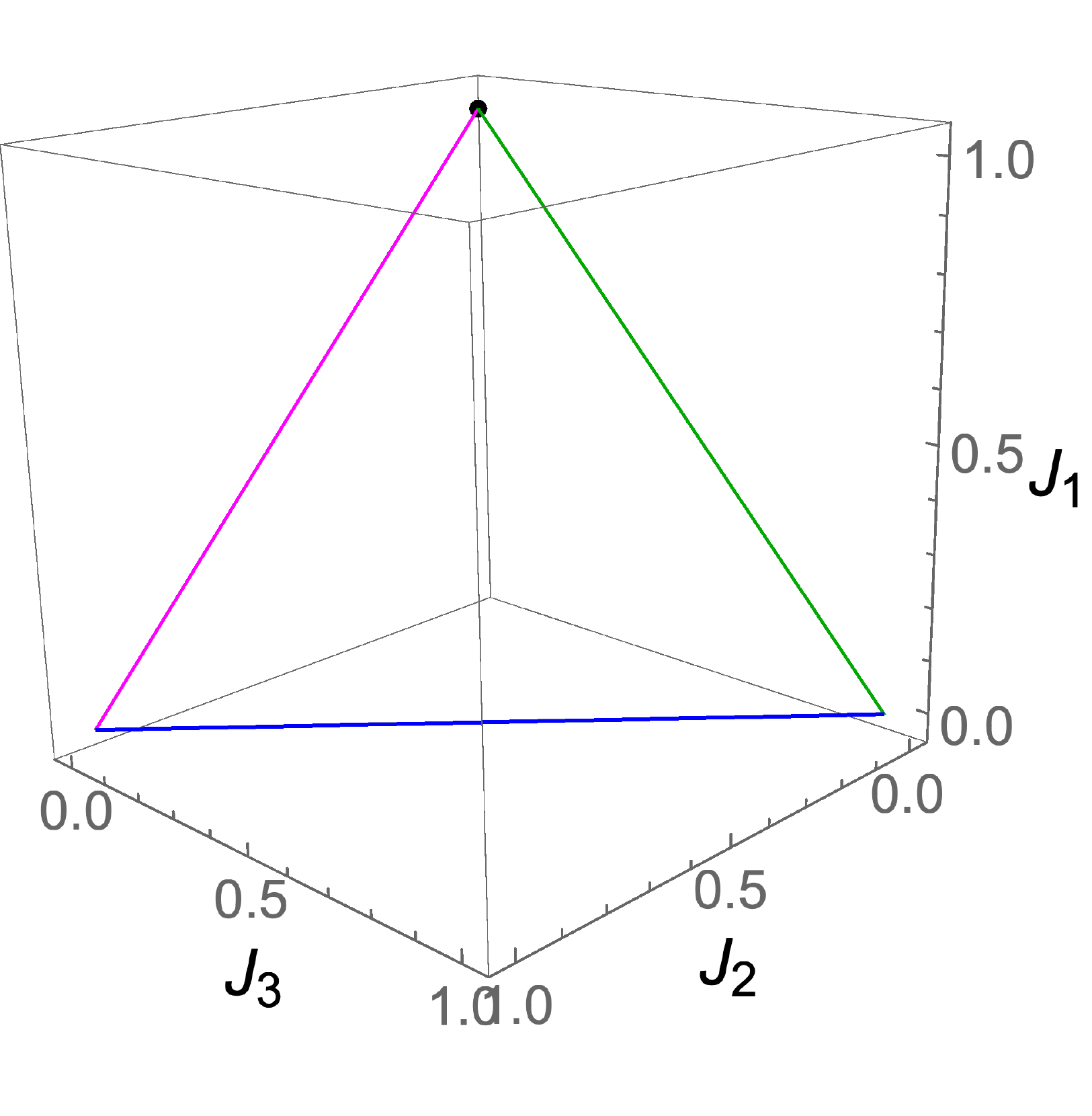}
\par\end{centering}
\caption{a) Momentum map and b) Action map for the $23$-spherical system.
\label{fig: RR sph }}
\end{figure}

\subsection{Cylindrical Coordinates}

The coordinate system with the highest symmetry has two global $S^1$ action, and there is only a single point in the Stasheff polytope for which this happens. The corresponding reduced system on $S^2\times S^2$ is toric.

\subsubsection{Separation of Variables}

The cylindrical coordinates (also called Hopf coordinates) $((1\ 2)\ (3\ 4))$ are a further degeneration
of the oblate coordinates obtained by setting both $e_{1}=e_{2}$
and $e_{3}=e_{4}$. \textcolor{black}{Specifically, the transformation
$(e_{2},e_{4})\to(e_{1}+\epsilon,e_{3}+\epsilon)$ along with $(s_{1},p_1,s_{3},p_3)\to(e_{1}+\epsilon\tilde{s}_{1},\tilde{p_1}/\epsilon,e_{3}+\epsilon\tilde{s}_{3},\tilde{p_3}/\epsilon)$
gives the following relationship between Cartesian coordinates and
cylindrical coordinates:}
\[
\begin{aligned}x_{1}^{2} & =s_{1}s_{2}, &  & x_{2}^{2}=s_{2}\left(1-s_{1}\right),\\
x_{3}^{2} & =s_{3}\left(1-s_{2}\right),&  & x_{4}^{2}=\left(1-s_{2}\right)\left(1-s_{3}\right),
\end{aligned}
\]
where $0\le s_{k}\le1$ and $k=1,2,3$. \textcolor{black}{The geodesic
Hamiltonian in these coordinates is 
\[
H_{Cyl}=-\frac{2p_{1}^{2}\left(s_{1}-1\right)s_{1}}{s_{2}}-2p_{2}^{2}\left(s_{2}-1\right)s_{2}+\frac{2p_{3}^{2}\left(s_{3}-1\right)s_{3}}{s_{2}-1}
\]
which trivially separates to give integrals $(H_{Cyl},\ell_{34}^2,\ell_{12}^2)$
} The separated equations are 
\begin{equation}
\begin{aligned}p_{2}^{2} & =\frac{l_{12}^{2}(s_{2}-1)-s_{2}\left(2h(s_{2}-1)+l_{34}^{2}\right)}{4(s_{2}-1)^{2}s_{2}^{2}} &  &  & p_{k}^{2} & =\frac{l_{\nu}^{2}}{4s_{k}(1-s_{k})}\end{aligned}
\label{eq:sep cyl}
\end{equation}
 where $\nu=12$ if $k=1$, $\nu=34$ if $k=3$ and $l_{\nu}$ denoted
the functional value of $\ell_{\nu}$.
%We call $(2H,\ell_{12},\ell_{34})$ the cylindrical integrable system on $T^{*}S^{3}$ and $(\ell_{12},\ell_{34})$ the corresponding reduced cylindrical integrable system on $S^2\times S^2$.

\subsubsection{Critical Points and Momentum map}
\begin{prop}
The bifurcation diagram for the cylindrical system $(\ell_{12},\ell_{34})$ on $S^2\times S^2$ with $2h=1$ is
composed of $4$ straight lines $\ell_{34}=\pm(1\pm\ell_{12})$ which
intersect transversally at $(\pm1,0)$ and $(0,\pm1)$ (see Figure \ref{fig:Cylindrical-coordin}). 
\end{prop}
The fibre of a regular
point on $S^2\times S^2$ is $T^2$ with multiplicity one. The lines
are all codimension one elliptic and their fibres are single $S^1$. The intersections of the lines are elliptic-elliptic
critical values with a single critical point in their fibres. 

\subsubsection{Actions}\label{cyl-act}
\textcolor{black}{The trivial actions for the cylindrical system
are $(J_{1},J_{3})=(\left|\ell_{12}\right|,\left|\ell_{34}\right|)$
while the ``non trivial'' action is easily determined by $J_{1}+J_{2}+J_{3}=1$.}
The action map is shown in Figure~\ref{fig:Cylindrical-coordin}~b). 
The relation between the symmetry reduced actions $|\ell_{12}|$ and $|\ell_{34}|$ and the global $S^1$ actions is to forget the absolute value sign. In this way 4 copies of the right triangle in $J_1, J_3$ are glued together to a diamond in $\ell_{12}, \ell_{34}$.

\begin{figure}
\begin{centering}
\includegraphics[width=7cm,height=7cm]{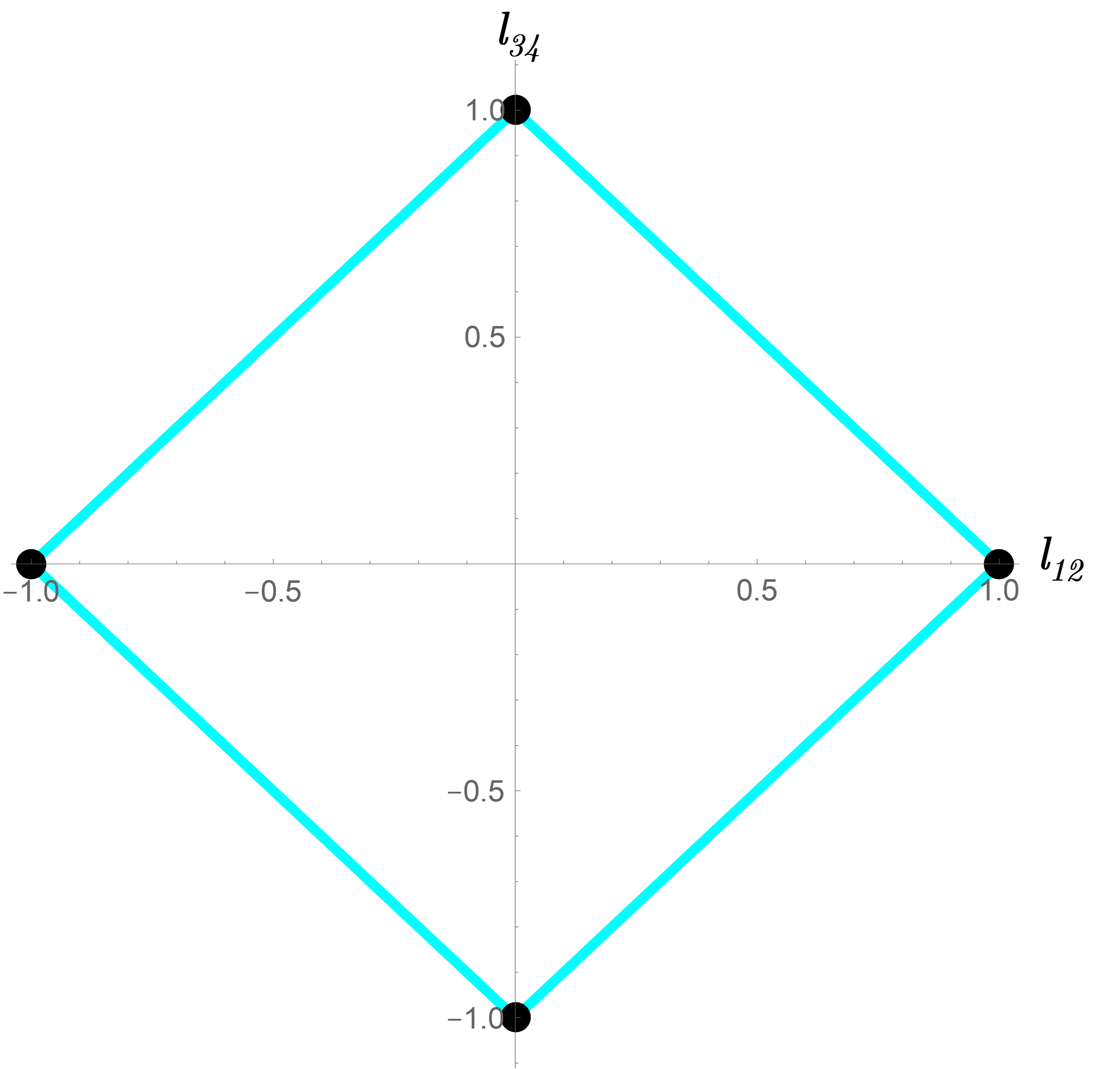}\quad\includegraphics[width=7cm,height=7cm]{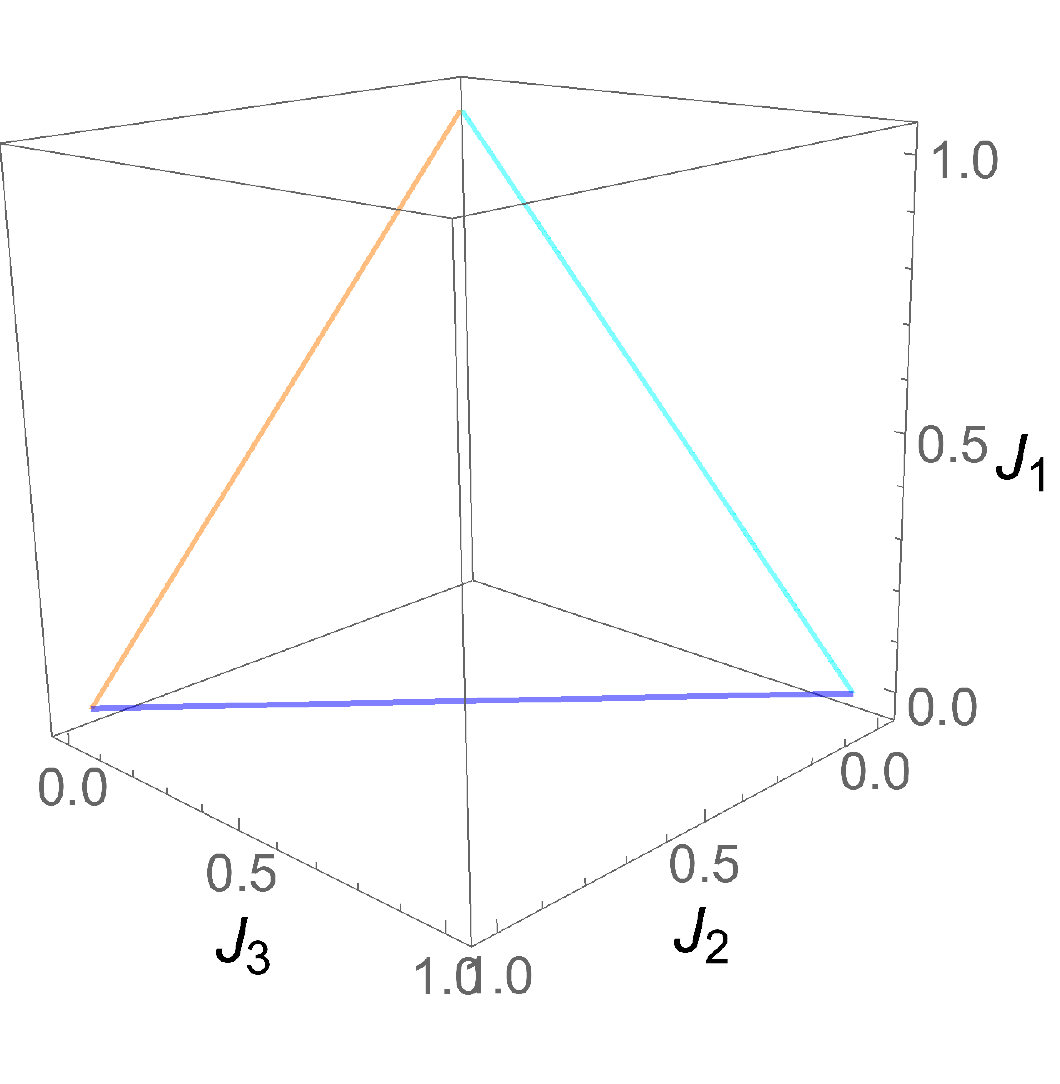}\\
\par\end{centering}
\caption{(a) Momentum map and (b) action map of the cylindrical system.\label{fig:Cylindrical-coordin}}
\end{figure}
 \begin{prop}
    The system $(X_1,Y_1)$ is a toric system on $S^2\times S^2$ where $X_1$ and $Y_1$ are defined in \eqref{eq:ls to Xy}.
\end{prop}
\begin{proof}
    Both $\ell_{12}$ and $\ell_{34}$ define smooth global $S^1$ action on $S^2\times S^2$. However, the torus action $(\ell_{12},\ell_{34})$ is not effective. Note that $\ell_{ij}$ is the generator of rotation in the $x_ix_j$-plane represented by $\exp({i t_{ij}\hat{\bm x}_i\wedge\hat{\bm x}_j})$ where $\hat{\bm x_i}$ is the unit vector in the $x_i$ axis. The action on the momenta $\bm y$ is the same. A rotation by $t_{ij}=\pi$ in the $x_ix_j$-plane has the effect
    \begin{equation*}
    \begin{aligned}
        x_i&\mapsto -x_i, & & & y_i&\mapsto -y_i,\\
        x_j&\mapsto -x_j, & & & y_j&\mapsto - y_j,
    \end{aligned}    
    \end{equation*}
    on $T^*S^3$. This induces the map 
    \begin{equation*}
    \begin{aligned}
        \ell_{ij}&\mapsto \ell_{ij}, & & & \ell_{mn}&\mapsto \ell_{mn},
    \end{aligned}    
    \end{equation*}
    for $i,j,m,n$ all distinct, and
    \begin{equation*}
    \begin{aligned}
        \ell_{ik}&\mapsto -\ell_{ik}, & & & \ell_{kj}&\mapsto -\ell_{kj},
    \end{aligned}    
    \end{equation*}
    for all $k\neq i,j$. In particular, the action with $t_{12}=\pi$ and $t_{34}=\pi$ both generate the same map 
    \begin{equation}
        (\bm X,\bm Y)\mapsto(X_1,-X_2,-X_3,Y_1,-Y_2,-Y_3)
        \label{l12l34}
    \end{equation}
    on $S^2\times S^2$. Since the flows of $\ell_{12}$ and $\ell_{34}$ commute and \eqref{l12l34} is an involution, we see that $(t_{12},t_{34})=(\pi,\pi)$ is the identity on $S^2\times S^2$, so the action is not effective. By taking half of the sum and difference, we see that $(X_1,Y_1)=\frac{1}{2}(\ell_{12}+\ell_{34},\ell_{12}-\ell_{34})$ is faithful with $2\pi$ period giving us a toric system on $S^2\times S^2$. The image of the momentum map of $(X_1,Y_1)$ is the unit square - the standard Delzant polytope for $S^2\times S^2$.

    Note that the torus action of $(\ell_{12},\ell_{34})$ is effective on $T^*S^3$ as $(t_{12},t_{34})=(\pi,\pi)$ gives $\bm x\mapsto -\bm x$ and $\bm y\mapsto-\bm y$. However, the points $(\bm x,\bm y)$ and $(-\bm x, -\bm y)$ are anti-podal points on the same great circle and thus become the same point on $S^2\times S^2$ after reduction.
\end{proof}

\section{Conclusion}\label{sec:S3-7}

The main novelty in this paper is the construction of a natural family of integrable system on $S^2 \times S^2$ in section \ref{sec:S3-4}, and the analysis of its Liouville foliation in section \ref{sec:degen}. It turns out that many properties of the reduced system are visible already in one way or another in the original St\"ackel system on $T^*S^3$. However, it should be pointed out that the upstairs system does not even have a natural Liouville foliation because it is superintegrable, and hence dynamically does not possess invariant tori, but just periodic orbits.
After reduction by the flow of the Hamiltonian, which after extracting the square root is a global $S^1$ action, an integrable system on $S^2\times S^2$ is obtained. 
The reduced system is Lie-Poisson with Lie-algebra $\mathfrak{so}(4)$. 

Since the reduction is done by the flow of the Hamiltonian the reduced system does not have a Hamiltonian any more, it just has commuting integrals. The definition of Liouville integrable system does not require a Hamiltonian, and the foliation into tori as defined by the integrals is defined independently of a Hamiltonian.
What is missing is the possibility to define the Hamiltonian vector field which induces a flow on these tori. But this is not necessary in order to study the equivalence of Liouville foliations of integrable systems.

At first it may be surprising that in the Liouville-Arnold theorem  the existence of action-angle variables near a regular torus does not need a Hamiltonian either. In fact, the action-angle variables are such that all of the integrals can be expressed as functions of the action variables. Moreover, considering diffeomorphism of the integrals changes the integrals, but does not change the action variables. We saw this explicitly   for a restricted class of transformations of the integrals in our case. Since after reduction by the Hamiltonian there is no distinguished function any more, the focus is fully on the action variables. For the foliation it makes sense to consider leaf-preserving homeomorphisms or diffeomorphisms, but from the point of view of the action variables the natural class is symplectomorphisms. 
Since the reduced symplectic manifold is compact the image of the momentum map is compact as well, and we have shown that the image of the action map (appropriately modified so that it is continous!) is a right triangle. This triangle is rigid, which means that it is the same for the whole family. What does change are the position and organisation of action values in the triangle that correspond to critical values of the momentum map. These play the role of the height invariant, and in fact for the prolate system which is semi-toric these turn into the height invariant. 

The Liouville-Arnold theorem holds near regular tori, and can be extended to open subsets of phase space bounded by separatrices. Only in rare cases are there no separatrices,  essentially this means that the system is toric. But most integrable systems do have singular fibres that are not just tori, and the classification of integrable system needs to take these  into account. It is crucial to note that the actions of the action-angle variables  can in general not be extended globally in phase space. If this is possible we call them global $S^1$ actions. Instead of $S^1$ action we may also speak of a global $SO(2)$ symmetry. A slightly less optimal situation occurs for a global $SO(3)$  symmetry, which leads to almost global $S^1$ actions, as described for the Lam\'e system, with an almost global $S^1$ action and a spherical type singularity. Examples of global $S^1$ actions do appear in our family through degenerations, and when they do appear they unfold the action map into the polygon invariant in the semi-toric case (prolate system) and into the Delzant polygon in the toric case (cylindrical system). 

Thus, for our family we have some analogues of important symplectic invariants, namely a convex polygon and generalisations of the height invariant. Certainly the semi-global symplectic invariants would need to be added, and at least in principle this is understood for the hyperbolic-hyperbolic point in the ellipsoidal family \cite{DullinVuNgoc07}, and generalisations to elliptic-hyperbolic points, degenerate points and the rank 1 hyperbolic lines would need to be worked out.
The interesting question is what kind of global invariants (like the twisting index invariant for semi-toric systems) would need to be added to the list so that it becomes the complete list of global symplectic invariants.

\appendix

\section{Appendix}

\subsection{Identification of $\widetilde{Gr}(2,4)$ with $S^{2}\times S^{2}$}\label{appen-gras}

There is a natural identification of $\widetilde{Gr}(2,4)$ with $S^{2}\times S^{2}$
via the Pl\"{u}cker embedding 
\[
\begin{aligned}i:\widetilde{Gr}(2,4) & \to\Lambda^{2}(\mathbb{R}^{4})\\
\text{span}(\bm v_{i},\bm v_{j}) & \to\frac{1}{\left|\bm v_{i}\right|\left|\bm v_{j}\right|}\bm v_{i}\wedge \bm v_{j}.
\end{aligned}
\]
If $(\bm v_{1},\bm v_{2},\bm v_{3},\bm v_{4})$ is an ordered-orthonormal basis of
$\mathbb{R}^{4}$ then 
\begin{equation}
\left(\bm v_{1}\wedge \bm v_{2},\bm v_{1}\wedge \bm v_{3},\bm v_{1}\wedge \bm v_{4},\bm v_{2}\wedge \bm v_{3},\bm v_{2}\wedge \bm v_{4},\bm v_{3}\wedge \bm v_{4}\right)\label{eq:basis vectors Gr(24)}
\end{equation}
 is a basis of $\Lambda^{2}(\mathbb{R}^{4})$ and a $\bm L\in\Lambda^{2}(\mathbb{R}^{4})$
is of the form 
\[
\bm L=\ell_{12}\bm v_{1}\wedge \bm v_{2}+\ell_{13}\bm v_{1}\wedge \bm v_{3}+\ell_{14}\bm v_{1}\wedge \bm v_{4}+\ell_{23}\bm v_{2}\wedge \bm v_{3}+\ell_{24}\bm v_{2}\wedge \bm v_{4}+\ell_{34}\bm v_{3}\wedge \bm v_{4}.
\]
The image of $\widetilde{Gr}(2,4)$ under $i$ is totally decomposable,
i.e. any $\bm L\in\widetilde{Gr}(2,4)$ can be expressed as $L=\bm x\wedge \bm y$
where $\bm x, \bm y\in\mathbb{R}^{4}$. Thus, $\bm L\in\text{im}(i)$ if and only
if $\bm L\wedge \bm L=0$. This yields the Pl\"{u}cker relation
\begin{equation}
\ell_{12}\ell_{34}-\ell_{13}\ell_{24}+\ell_{14}\ell_{23}=0.\label{eq:plucker in lij}
\end{equation}
 This relation defines the image of $\widetilde{Gr}(2,4)$ under $i$.
On $\Lambda^{2}(\mathbb{R}^{4})$ the Hodge star operator gives a
decomposition
\[
\star:\Lambda^{2}(\mathbb{R}^{4})\to\Lambda_{+}^{2}(\mathbb{R}^{4})\oplus\Lambda_{-}^{2}(\mathbb{R}^{4})
\]
 where $\Lambda_{\pm}^{2}(\mathbb{R}^{4})$ are the $\pm1$ eigenspaces.
For convenience we let $\bm V_{ij}\coloneqq \bm v_{i}\wedge \bm v_{j}$. Then
the effect of $\star$ on the basis vectors (\ref{eq:basis vectors Gr(24)})
is as follows
\[
\begin{aligned}\star(\bm V_{12}) & =\bm V_{34} &  &  & \star(\bm V_{34}) & =\bm V_{12}\\
\star(\bm V_{13}) & =-\bm V_{24} &  &  & \star(\bm V_{24}) & =-\bm V_{13}\\
\star(\bm V_{14}) & =\bm V_{23} &  &  & \star(\bm V_{23}) & =\bm V_{14}.
\end{aligned}
\]
This gives the following bases of $\Lambda_{+}^{2}(\mathbb{R}^{4})$
and $\Lambda_{-}^{2}(\mathbb{R}^{4})$:
\begin{equation}
\begin{aligned}\Lambda_{+}^{2}: & \frac{1}{2}\left\langle \bm V_{12}+\bm V_{34},\bm V_{13}-\bm V_{24},\bm V_{14}+\bm V_{23}\right\rangle \\
\Lambda_{-}^{2} & :\frac{1}{2}\left\langle \bm V_{12}-\bm V_{34},\bm V_{13}+\bm V_{24},\bm V_{14}-\bm V_{23}\right\rangle .
\end{aligned}  \label{eq:bases}
\end{equation}
\textcolor{black}{The Hodge star operator then induces a decomposition
\[
\begin{aligned}\star\circ i: & \widetilde{Gr}(2,4)\to S_{+}^{2}\times S_{-}^{2}\\
\text{span}(\bm x,\bm y) & \to\frac{1}{{2}\left|\bm x\right|\left|\bm y\right|}\left[\left(\bm x\wedge \bm y+\star(\bm x\wedge \bm y)\right)+\left(\bm x\wedge \bm y-\star(\bm x\wedge \bm y\right)\right]
\end{aligned}
\]
where $S_{+}^{2}\subseteq\Lambda_{+}^{2},S_{-}^{2}\subseteq\Lambda_{-}^{2}$.}
Note that
\[
\begin{aligned}
    \bm X =&\frac{1}{{2}\left|\bm x\right|\left|\bm y\right|}\left(\bm x\wedge \bm y+\star(\bm x\wedge \bm y)\right) & & & \bm Y = &\frac{1}{{2}\left|\bm x\right|\left|\bm y\right|}\left(\bm x\wedge \bm y-\star(\bm x\wedge \bm y)\right)
\end{aligned}\] are
are both unit vectors in $\mathbb{R}^{3}$ with the bases \eqref{eq:bases} of $\Lambda^2_\pm$. Thus, a plane spanned
by $(\bm x,\bm y)$ is identified with two unit vectors $(\bm X,\bm Y)$ in $\mathbb{R}^{3},$
i.e $S^{2}\times S^{2}$. \textcolor{black}{Geometrically $S_{+}^{2}$
and $S_{-}^{2}$ represent the left and right isoclinic rotations
in $\mathbb{R}^{4}$.} Note that $\exp(\bm x\wedge \bm y)$ is a rotation
in the plane spanned by $\bm x$ and $\bm y$. 

For a given point $(\bm{X},\bm{Y})\in S^{2}\times S^{2}$, we have $T^{-1}(\bm{X},\bm{Y})=(\ell_{12},\ell_{13},\ell_{14},\ell_{23},\ell_{24},\ell_{34})$
defines an element $\bm{L}\in\Lambda^{2}(\mathbb{R}^{4})$ where the map $T$ is defined in \eqref{eq:ls to Xy}. For a
given $\bm{L}$ we can find the corresponding circle on $S^{3}$ by
considering the linear map 
\[
\begin{aligned}M_{\bm{L}} & :\Lambda^{1}(\mathbb{R}^{4})\to\Lambda^{3}(\mathbb{R}^{4})\\
\bm v & \to \bm v\wedge\bm{L}.
\end{aligned}
\]
Since $\bm{L}$ is decomposable with $\bm{L}=\bm{x}\wedge \bm{y}$ then $\ker(M_{\bm{L}})=\text{span}(\bm{x},\bm{y})$.
Let $(b_{1},b_{2},b_{3},b_{4})$ be a basis of $\mathbb{R}^{4}$ and
$(b_{123},b_{124},b_{134},b_{234})$ be bases of $\Lambda^{3}(\mathbb{R}^{4})$.
Then we can represent $M_{\bm{L}}$ as 
\[
M_{\bm{W}}=\begin{pmatrix}\ell_{23} & -\ell_{13} & \ell_{12} & 0\\
\ell_{24} & -\ell_{14} & 0 & \ell_{12}\\
\ell_{34} & 0 & -\ell_{14} & \ell_{13}\\
0 & \ell_{34} & -\ell_{24} & \ell_{23}
\end{pmatrix}.
\]
\textcolor{black}{If $\bm{L}$ is in the image of the Pl\"{u}cker embedding
then $\text{rank}(M_{\bm{L}})$ is exactly $2$ as all $16$ of the
$3\times3$ minors vanishes under the Pl\"{u}cker relation. Thus, finding
the nullspace of $M_{\bm{L}}$ gives us an unoriented plane in $\mathbb{R}^{4}$
whose intersection with $S^{3}$ is the fibre over $S^{2}\times S^{2}$.}

\subsection{Geodesic Flow on $S^{2}$}

Following a similar construction to Section \ref{sec:Geodesic-Flow-on-S3},
we can define the geodesic flow on $S^{2}$ as a constrained system
on $T^{*}\mathbb{R}^{3}$ with Cartesian coordinates $(\bm{x},\bm{y})$
where $\bm{x}\cdot\bm{x}=1$ and $\bm{x}\cdot\bm{y}=0$. The Dirac
bracket yields a Poisson structure that can still be accurately represented
by the matrix $B_{(\bm{x},\bm{y})}$ in (\ref{eq:BPQ-1-1}), where
here $\bm{x}$ and $\bm{y}$ are vectors in $\mathbb{R}^{3}$. The
invariants of the Hamiltonian $H=\frac{1}{2}\bm{y}\cdot\bm{y}$ are
the three angular momenta $\bm{L}=(\ell_{12},\ell_{13},\ell_{23})$.
Symplectic reduction by the $S^{1}$ action of $\sqrt{ 2 H}$ gives a reduced
space that is diffeomorphic to the sphere $S^{2}$ given by the Casimir
$\ell_{12}^{2}+\ell_{13}^{2}+\ell_{23}^{2}=2h=1$. The Poisson algebra
of the invariants $\ell_{ij}$ is isomorphic to $\mathfrak{so}(3)$.

There are 2 non-equivalent orthogonal separable coordinates on the
sphere $S^{2}$ leading to 2 distinct St\"{a}ckel systems. They are the
spherical coordinates and the elliptic coordinates with semi-axes
$(e_{1},e_{2},e_{3})$.

\subsubsection{Elliptic Coordinates on $S^{2}$}\label{s2ellip}

The elliptic coordinates with semi-axes $(e_{1},e_{2},e_{3})$ are
defined by 

\[
\begin{aligned}x_{1}^{2} & =\frac{(s_{1}-e_{1})(s_{2}-e_{1})}{(e_{2}-e_{1})(e_{3}-e_{1})}\\
x_{2}^{2} & =\frac{(s_{1}-e_{2})(s_{2}-e_{2})}{(e_{1}-e_{2})(e_{3}-e_{2})}\\
x_{3}^{2} & =\frac{(s_{1}-e_{3})(s_{2}-e_{3})}{(e_{1}-e_{3})(e_{2}-e_{3})}.
\end{aligned}
\]

Performing St\"{a}ckel separation or analysis using compatible Poisson
structures both gives the separation constants $1=\ell_{12}^{2}+\ell_{13}^{2}+\ell_{23}^{2}$
and $\eta_{1}=e_{3}\ell_{12}^{2}+e_{2}\ell_{13}^{2}+e_{1}\ell_{23}^{2}$.
Using the matrix $C=\text{diag}(\lambda-c_{1},\lambda-c_{2},\lambda-c_{3})$,
the Poisson matrix for $(so_{3}^{*},\{\cdot,\cdot\}_{C})$ is given
by

\[
B_{C}=\left(\begin{array}{ccc}
0 & \ell_{23}(\lambda-e_{1}) & \ell_{13}(e_{2}-\lambda)\\
\ell_{23}(e_{1}-\lambda) & 0 & \ell_{12}(\lambda-e_{3})\\
\ell_{13}(\lambda-e_{2}) & \ell_{12}(e_{3}-\lambda) & 0
\end{array}\right).
\]

The bracket $\{\cdot,\cdot\}_{C}$ drops rank only when exactly 2
of the $\ell_{ij}'s$ vanishes, giving the 6 poles $\pm(1,0,0),\,\pm(0,1,0)$
and $\pm(0,0,1)$ as critical points on the sphere $S^2$ defined by fixing the Casimir. The image of the momentum map
is the line segment $[e_{1},e_{3}]$ with 3 critical values at $e_{1},\,e_{2}$
and $e_{3}$. The critical points at $e_{1}$ and $e_{3}$ are elliptic
and ones at $e_{2}$ are hyperbolic. The fibres of $e_{1}$ and $e_{3}$ are
the poles $\pm(0,0,1)$ and $\pm(1,0,0)$ respectively. The preimage
of $e_{2}$ is the intersection of the sphere $1=\ell_{12}^{2}+\ell_{13}^{2}+\ell_{23}^{2}$
with the ellipsoid $e_{2}=e_{3}\ell_{12}^{2}+e_{2}\ell_{13}^{2}+e_{1}\ell_{23}^{2}$.
The sphere and this ellipsoid intersect tangentially at the poles
$\pm(0,1,0)$.

The reduced system is of course the Euler top with phase space $S^2$ and Hamiltonian $\eta_1$, where $(e_1,e_2,e_3)$ are the inverse moments of inertia of the top.

\subsubsection{Spherical Coordinates on $S^{2}$}

We define the spherical coordinates on $S^{2}$ with 

\[
\begin{aligned}x_{1}^{2} & =s_{1}\\
x_{2}^{2} & =(1-s_{1})s_{2}\\
x_{3}^{2} & =(1-s_{1})(1-s_{2})
\end{aligned}
.
\]
This system easily separates with separation constants $1=\ell_{12}^{2}+\ell_{13}^{2}+\ell_{23}^{2}$
and $\eta_{1}=\ell_{23}^{2}$. The image of the momentum map is the
line segment $[0,1]$ with 2 critical values at $0$ and $1$. The
critical points at $1$ are the poles $\pm(0,0,1)$ and are elliptic.
The point $\ell_{23}^{2}=0$ is degenerate and it's fibre is the equator
of the sphere. 

The reduced system is the symmetric Euler top with phase space $S^2$ and two equal moments of inertia.

\bibliographystyle{abbrv}
\bibliography{ClassicalS3Ref}

\def\cprime{$'$}
\begin{thebibliography}{10}

\bibitem{Alonso}
J.~Alonso, H.~R. Dullin, and S.~Hohloch.
\newblock Symplectic classification of coupled angular momenta.
\newblock {\em Nonlinearity}, 33(1):417--468, 2019.

\bibitem{Alonso2019}
J.~Alonso and S.~Hohloch.
\newblock Survey on recent developments in semitoric systems.
\newblock {\em {arXiv preprint, arXiv:1901.10433v2}}, 2019.

\bibitem{atiyah}
M.~F. Atiyah.
\newblock Convexity and commuting {Hamiltonians}.
\newblock {\em Bulletin of the London Mathematical Society}, 14(1):1--15, 1982.

\bibitem{Stackel}
S.~Benenti.
\newblock {St\"ackel} systems and {Killing} tensors.
\newblock {\em Note di Matematica}, Volume 9, suppl., 1989.

\bibitem{Bolsinov-Borisov}
A.~V. Bolsinov and A.~V. Borisov.
\newblock Compatible {Poisson} brackets on {Lie} algebras.
\newblock {\em Mathematical Notes}, 72(1/2):10--30, 2002.

\bibitem{book}
A.~V. Bolsinov and A.~T. Fomenko.
\newblock {\em Integrable {Hamiltonian} Systems: {Geometry} Topology
  Classification}.
\newblock 01 2004.

\bibitem{Bolsinov-Oshemkov}
A.~V. Bolsinov and A.~A. Oshemkov.
\newblock {Bi-Hamiltonian} structures and singularities of integrable systems.
\newblock {\em Regular and Chaotic Dynamics}, 14:431--454, 2009.

\bibitem{Chiscop_2019}
I.~Chiscop, H.~R. Dullin, K.~Efstathiou, and H.~Waalkens.
\newblock A {Lagrangian} fibration of the {Isotropic 3-Dimensional Harmonic
  Oscillator} with monodromy.
\newblock {\em Journal of Mathematical Physics}, 60(3):032103, 2019.

\bibitem{Davison2007}
C.~M. Davison and H.~R. Dullin.
\newblock Geodesic flow on three dimensional ellipsoids with equal semi-axes.
\newblock {\em Regular and Chaotic Dynamics}, 12:172--197, 2007.

\bibitem{DAVISON20072437}
C.~M. Davison, H.~R. Dullin, and A.~V. Bolsinov.
\newblock Geodesics on the ellipsoid and monodromy.
\newblock {\em Journal of Geometry and Physics}, 57(12):2437 -- 2454, 2007.

\bibitem{Dawson2022}
S.~R. Dawson, H.~R. Dullin, and D.~M.~H. Nguyen.
\newblock The harmonic {L}agrange top and the confluent {H}eun equation.
\newblock {\em Regular and Chaotic Dynamics}, 27(4):443--459, 2022.

\bibitem{BSMF_1988__116_3_315_0}
T.~Delzant.
\newblock {Hamiltoniens} p{\'{e}}riodiques et images convexes de l'application
  moment.
\newblock {\em Bulletin de la Soci\'et\'e Math\'ematique de France},
  116(3):315--339, 1988.

\bibitem{Dullin2012}
H.~R. Dullin and H.~Han{\ss}mann.
\newblock The degenerate {C. N}eumann system {I}: symmetry reduction and
  convexity.
\newblock {\em Central European Journal of Mathematics}, 10(5):1627--1654,
  2012.

\bibitem{DullinVuNgoc07}
H.~R. Dullin and S.~V. Ng\d{o}c.
\newblock Symplectic invariants near hyperbolic-hyperbolic points.
\newblock {\em Regular and Chaotic Dynamics}, 12(6):689--716, 2007.

\bibitem{DullinNeuman}
H.~R. Dullin, P.~H. Richter, A.~P. Veselov, and H.~Waalkens.
\newblock Actions of the {Neumann} systems via {Picard–Fuchs} equations.
\newblock {\em Physica D: Nonlinear Phenomena}, 155:159--183, 2001.

\bibitem{Dullin2016}
H.~R. Dullin and H.~Waalkens.
\newblock Defect in the joint spectrum of {Hydrogen} due to monodromy.
\newblock {\em Phys. Rev. Lett. 120, 020507}, 120, 2018.

\bibitem{10.2307/1968433}
L.~P. Eisenhart.
\newblock Separable systems of {St\"ackel}.
\newblock {\em Annals of Mathematics}, 35(2):284--305, 1934.

\bibitem{Fasso05}
F.~Fass{\`{o}}.
\newblock Superintegrable {Hamiltonian} systems: {Geometry} and perturbations.
\newblock {\em Acta Applicandae Mathematica}, 87(1-3):93--121, 2005.

\bibitem{Palmer2018}
Y.~L. Floch and J.~Palmer.
\newblock Semitoric families.
\newblock {\em {arXiv preprint, arXiv:1810.06915}}, 2018.

\bibitem{Guillemin1982}
V.~Guillemin and S.~Sternberg.
\newblock Convexity properties of the moment mapping.
\newblock {\em Inventiones Mathematicae}, 67:491--513, 1982.

\bibitem{GURNeumanQuantum}
D.~Gurarie.
\newblock Quantized {Neumann} problem, separable potentials on {$S^n$} and the
  {Lam\'e} equation.
\newblock {\em Journal of Mathematical Physics}, 36:5355--5391, 1995.

\bibitem{Hohloch2017}
S.~Hohloch and J.~Palmer.
\newblock A family of compact semitoric systems with two focus-focus
  singularities.
\newblock {\em Journal of Geometric Mechanics}, 10(3):331--357, 2017.

\bibitem{Kalnins2005-1}
E.~G. Kalnins, J.~M. Kress, and W.~Miller.
\newblock Second-order superintegrable systems in conformally flat spaces. {I}.
  {T}wo-dimensional classical structure theory.
\newblock {\em Journal of Mathematical Physics}, 46(5):053509, May 2005.

\bibitem{Kalnins2005-2}
E.~G. Kalnins, J.~M. Kress, and W.~Miller.
\newblock Second-order superintegrable systems in conformally flat spaces.
  {II}. the classical two-dimensional {St{\"a}ckel} transform.
\newblock {\em Journal of Mathematical Physics}, 46(5):053510, May 2005.

\bibitem{Kalnins2005-3}
E.~G. Kalnins, J.~M. Kress, and W.~Miller.
\newblock Second-order superintegrable systems in conformally flat spaces.
  {III}. {T}hree-dimensional classical structure theory.
\newblock {\em Journal of Mathematical Physics}, 46(10):103507, Oct. 2005.

\bibitem{Kress06}
E.~G. Kalnins, J.~M. Kress, and W.~Miller.
\newblock Second-order superintegrable systems in conformally flat spaces.
  {IV}. {T}he classical {3D {St{\"{a}}ckel}} transform and {3D} classification
  theory.
\newblock {\em Journal of mathematical physics}, 47(4):043514, 2006.

\bibitem{Kalnins2006-5}
E.~G. Kalnins, J.~M. Kress, and W.~Miller.
\newblock Second-order superintegrable systems in conformally flat spaces. {V}.
  {T}wo- and three-dimensional quantum systems.
\newblock {\em Journal of Mathematical Physics}, 47(9):093501, Sept. 2006.

\bibitem{KKM18}
E.~G. Kalnins, J.~M. Kress, and W.~Miller.
\newblock {\em Separation of Variables and Superintegrability}.
\newblock Bristol: IOP Publishing, 2018.

\bibitem{Kalnins1986}
E.~G. Kalnins and W.~Miller.
\newblock Separation of variables on n-dimensional {Riemannian} manifolds. i.
  the n-sphere {$S^n$} and {Euclidean $n$-space} $\mathbb{R}^n$.
\newblock 27(7):1721--1736.

\bibitem{kalnins76}
E.~G. Kalnins, J.~W.~Miller, and P.~Winternitz.
\newblock The group {$O(4)$}, separation of variables and the {Hydrogen} atom.
\newblock {\em SIAM Journal on Applied Mathematics}, 30(4):630--664, 1976.

\bibitem{RonanThesis}
R.~Kerr.
\newblock {\em On spherical type singularities in integrable systems}.
\newblock Loughborough University, 2022.

\bibitem{Komarov_1991}
I.~V. Komarov and V.~B. Kuznetsov.
\newblock {Quantum Euler-Manakov} top on the three-sphere {$S^3$}.
\newblock {\em Journal of Physics A: Mathematical and General},
  24(13):L737--L742, 1991.

\bibitem{miller81}
W.~Miller, J.~Patera, and P.~Winternitz.
\newblock Subgroups of {Lie} groups and separation of variables.
\newblock {\em Journal of Mathematical Physics}, 22(2):251--260, 1981.

\bibitem{Moser:2261095}
J.~Moser.
\newblock {\em Various Aspects of Integrable {Hamiltonia} Systems}.
\newblock Lezioni Fermiane. Accademia Nazionale dei Lincei, Pisa, 1981.

\bibitem{VuNgoc07}
S.~V. Ng{\d{o}}c.
\newblock Moment polytopes for symplectic manifolds with monodromy.
\newblock {\em Advances in Mathematics}, 208(2):909--934, 2007.

\bibitem{VuNgoc09}
A.~Pelayo and S.~V. Ng\d{o}c.
\newblock Semitoric integrable systems on symplectic 4-manifolds.
\newblock {\em Invent. Math.}, 177(3):571--597, 2009.

\bibitem{Pelayo2009}
A.~Pelayo and S.~V. Ng\d{o}c.
\newblock Semitoric integrable systems on symplectic 4-manifolds.
\newblock {\em Inventiones mathematicae}, 177(3):571--597, 2009.

\bibitem{Schoebel2014}
K.~Sch{\"{o}}bel.
\newblock The variety of integrable {Killing} tensors on the {3-Sphere}.
\newblock {\em Symmetry, Integrability and Geometry: Methods and Applications
  {(SIGMA)}}, 10:080, 2014.

\bibitem{Schoebel2016}
K.~Sch{\"o}bel.
\newblock Are orthogonal separable coordinates really classified?
\newblock {\em Symmetry, Integrability and Geometry: Methods and Applications
  {(SIGMA)}}, 12:No 041, 16, 2016.

\bibitem{Schoebel2015}
K.~Sch{\"{o}}bel and A.~P. Veselov.
\newblock Separation coordinates, moduli spaces and {Stasheff} polytopes.
\newblock {\em Communications in Mathematical Physics}, 337(3):1255--1274,
  2015.

\bibitem{Sepe2017}
D.~Sepe and S.~V. Ng\d{o}c.
\newblock Integrable systems, symmetries, and quantization.
\newblock {\em Letters in Mathematical Physics}, 108(3):499--571, 2017.

\bibitem{339ab604500943c5b69d0da421f62e74}
E.~Sinitsyn and B.~Zhilinskii.
\newblock Qualitative analysis of the classical and quantum {Manakov} top.
\newblock {\em Symmetry, Integrability and Geometry: Methods and Applications
  {(SIGMA)}}, 3:046, 2007.

\end{thebibliography}

\end{document}